\documentclass[smallextended]{svjour3}

\RequirePackage{amsmath}
\RequirePackage{amsfonts}
\RequirePackage{amssymb}
\RequirePackage{hyperref}
\RequirePackage{multirow}
\RequirePackage{algorithm}
\RequirePackage{algpseudocode}
\RequirePackage{graphicx}
\RequirePackage{xcolor}
\RequirePackage{bm}
\RequirePackage{caption}
\allowdisplaybreaks

\usepackage{ifthen}
\usepackage{epstopdf}
\usepackage{mathtools}
\usepackage{booktabs}

\hypersetup{colorlinks=true}

\newcommand{\rev}[1]{#1}

\title {A Two Stepsize SQP Method for Nonlinear Equality Constrained Stochastic Optimization}
\titlerunning {Two Stepsize SQP Methods for Stochastic Optimization}
\author{Michael~J.~O'Neill}

\institute{Department of Statistics and Operations Research, University of North Carolina, Chapel Hill, NC, USA; Email: \href{mailto:mikeoneill@unc.edu}{mikeoneill@unc.edu}; 
Corresponding Author: Michael J. O'Neill
}

\date  {\today}

\begin{document}

\maketitle

\begin{abstract}
  We develop a Sequential Quadratic Optimization (SQP) algorithm for minimizing a stochastic objective function subject to deterministic equality constraints. The method utilizes two different stepsizes, one which exclusively scales the component of the step corrupted by the variance of the stochastic gradient estimates and a second which scales the entire step. We prove that this stepsize splitting scheme has a worst-case complexity result which improves over the best known result for this class of problems. In terms of approximately satisfying the constraint violation, this complexity result matches that of deterministic SQP methods, up to constant factors, while matching the known optimal rate for stochastic SQP methods to approximately minimize the norm of the gradient of the Lagrangian. We also propose and analyze multiple variants of our algorithm. One of these variants is based upon popular adaptive gradient methods for unconstrained stochastic optimization while another incorporates a safeguarded line search along the constraint violation. Preliminary numerical experiments show competitive performance against \rev{state of the art algorithms}. In addition, in these experiments, we observe an improved rate of convergence in terms of the constraint violation, as predicted by the theoretical results.
\end{abstract}



\newcommand{\flow}{f_{\text{\rm low}}}
\newcommand{\kappadec}{\kappa_{\text{\rm dec}}}
\newcommand{\bK}{\bar{K}}
\newcommand{\balpha}{\bar{\alpha}}
\newcommand{\btau}{\bar{\tau}}
\newcommand{\bd}{\bar{d}}
\newcommand{\bg}{\bar{g}}
\newcommand{\bxi}{\bar{\xi}}
\newcommand{\hdelta}{\hat{\delta}}
\newcommand{\argmin}{\text{argmin}}

\newcommand{\xiktri}{\xi_k^{\rm trial}}
\newcommand{\ktaudec}{\kappa_{\tau_{\rm dec}}}
\newcommand{\kmax}{k_{\max}}
\newcommand{\smax}{s_{\max}}
\newcommand{\tauktri}{\tau_k^{\rm trial}}
\newcommand{\taumin}{\tau_{\min}}
\newcommand{\fmin}{f_{\min}}
\newcommand{\etamin}{\eta_{\min}}
\newcommand{\etamax}{\eta_{\max}}

\newcommand{\dktrue}{d_k^{\text{\rm true}}}
\newcommand{\pktrue}{p_k^{\text{\rm true}}}
\newcommand{\uktrue}{u_k^{\text{\rm true}}}
\newcommand{\hdktrue}{\hat{d}_k^{\text{\rm true}}}
\newcommand{\huktrue}{\hat{u}_k^{\text{\rm true}}}
\newcommand{\hyktrue}{\hat{y}_k^{\text{\rm true}}}
\newcommand{\hwktrue}{\hat{w}_k^{\text{\rm true}}}
\newcommand{\wktrue}{w_k^{\text{\rm true}}}
\newcommand{\yktrue}{y_k^{\text{\rm true}}}
\newcommand{\ykstartrue}{y_{k^*}^{\text{\rm true}}}
\newcommand{\tktritrue}{\tau_k^{\rm \text{trial},\text{\rm true}}}
\newcommand{\Tktritrue}{\Tcal_k^{\rm \text{trial},\text{\rm true}}}
\newcommand{\Dktrue}{D_k^{\text{\rm true}}}
\newcommand{\Yktrue}{Y_k^{\text{\rm true}}}
\newcommand{\Uktrue}{U_k^{\text{\rm true}}}
\newcommand{\Wktrue}{W_k^{\text{\rm true}}}
\newcommand{\Uk}{U_k}
\newcommand{\Vk}{V_k}
\newcommand{\Fk}{\mathcal{F}_k}

\newcommand{\R}{\mathbb{R}}
\newcommand{\N}{\mathbb{N}}
\renewcommand{\S}{\mathbb{S}}

\newcommand{\strue}{s^{\text{\rm true}}}

\newcommand{\E}{\mathbb{E}}
\renewcommand{\P}{\mathbb{P}}

\newcommand{\Lbad}{\Lcal_{\text{\rm bad}}}
\newcommand{\Lgood}{\Lcal_{\text{\rm good}}}
\newcommand{\Ebad}{E_{\text{\rm bad}}}
\newcommand{\EbadB}{E_{\text{\rm bad},B}}
\newcommand{\Cdec}{C_{\text{\rm dec}}}
\newcommand{\Idec}{\Ical_{\text{\rm dec}}}
\newcommand{\ktau}{k_{\tau}}
\newcommand{\tcalktritrue}{\Tcal_k^{\rm \text{trial},\text{\rm true}}}
\newcommand{\alphakmin}{\alpha_{k,\min}}
\newcommand{\bkmax}{b_{k,\max}}

\def\Ek#1{E_{k,#1}}
\DeclarePairedDelimiter\ceil{\lceil}{\rceil}

\newtheorem{assumption}{Assumption}


\section{Introduction} \label{sec:intro}

We propose a new algorithm for solving equality constrained optimization problems in which the objective function is the expectation of a stochastic function. Formally, we consider the optimization problem
\begin{equation} \label{eq:fdef}
    \underset{x \in \mathbb{R}^n}{\min} \quad f(x) \quad \text{s.t.} \quad c(x) = 0 \quad \text{with} \quad f(x) = \E[F(x,\omega)],
\end{equation}
where $f : \mathbb{R}^n \rightarrow \mathbb{R}$, $c : \mathbb{R}^n \rightarrow \mathbb{R}^m$, $\omega$ is a random variable with associated probability space $(\Omega, \mathcal{F}, P)$, $F : \mathbb{R}^n \times \Omega \rightarrow \mathbb{R}$, and $\E[\cdot]$ denotes the expectation taken with respect to $P$. Problems of this form arise in numerous applications, including optimal control \cite{JTBetts_2010}, PDE-constrained optimization \cite{FSKupfer_EWSachs_1992,TRees_HSDollar_AJWathen_2010}, and resource allocation \cite{DPBertsekas_1998} as well as modern machine learning applications, such as physics informed neural networks \cite{SCuomo_DCVincenzoSchiano_FGiampaolo_GRozza_MRaissi_FPiccialli_2022,MRaissi_PPerdikaris_GEKarniadakis_2019}, constraining the output labels of deep neural networks \cite{YNandwani_APathak_PSingla_2019} and neural network compression via constraints \cite{CChen_FTung_NVedula_2018}.

The method we design is based on Sequential Quadratic Optimization (SQP) methods, a popular class of algorithms that has seen significant interest in recent years for solving stochastic equality constrained optimization problems, beginning with the influential work of \cite{ASBerahas_FECurtis_DPRobinson_BZhou_2021}. Numerous extensions of this work have been proposed, such as stochastic SQP methods for problems with rank-deficient Jacobians \cite{ASBerahas_FECurtis_MJONeill_DPRobinson_2023}, algorithms for problems with nonlinear inequality constraints \cite{FECurtis_DPRobinson_BZhou_2024B}, worst-case complexity analysis for stochastic SQP methods \cite{FECurtis_MJONeill_DPRobinson_2024}, algorithms which incorporate variance reduction \cite{ASBerahas_JShi_ZYi_BZhou_2023} or adaptive sampling \cite{ASBerahas_RBollapragada_BZhou_2022},  as well as stochastic SQP methods which utilize an exact augmented Lagrangian as a merit function \cite{SNa_MAnitescu_MKolar_2023A,SNa_MAnitescu_MKolar_2023B}.
At each iteration, these algorithms generate a search direction by solving a quadratic optimization problem defined in terms of a stochastic gradient estimate subject to a linearization of the constraints and then produce a new iterate by moving along this search direction. For stochastic SQP methods, the chosen step length is generally scaled in such a way as to control the variance of the stochastic gradient estimates, in a manner similar to stepsizes for stochastic gradient methods in unconstrained optimization. Our algorithm takes a different approach and directly utilizes the orthogonal step decomposition of SQP methods\footnote{Computation of the orthogonal decomposition may be unnecessary in certain cases, see Remark \ref{rem:scaleHk} for details.}. It is well known in the stochastic SQP literature that the normal component of the step decomposition is independent of the current stochastic gradient estimate. Therefore, it is unnecessary to rescale this component by the stepsize which controls the variance in the stochastic gradient estimates in order to ensure convergence. Using this observation, we propose a method which employs two different stepsizes: one which controls the variance of the stochastic gradient estimates and scales only the tangential component and a second stepsize which scales the entire search direction.

We demonstrate the effectiveness of this stepsize splitting approach by developing a worst-case complexity result for our proposed algorithm. We consider the worst-case complexity in terms of finding a point $x$ which satisfies,
\begin{equation} \label{eq:complexity}
    \E[\|\nabla f(x) + \nabla c(x) y\|] \leq \epsilon_{\ell}, \quad \quad \E[\|c(x)\|_1] \leq \epsilon_c,
\end{equation}
where $y \in \mathbb{R}^m$ is some Lagrange multiplier and $\epsilon_\ell$ and $\epsilon_c$ are some small \rev{positive} tolerances. \rev{Prior work has established a variety of complexity results for different SQP algorithms.} \rev{In \cite{JBolte_EPauwels_2016}, convergence rates are established for a class of SQP methods for inequality constrained optimization under the assumption the Kurdyka-\L ojasiewicz (KL) property, with the convergence rate dependent on the \L ojasiewicz exponent. In another line of work, \cite{FFacchinei_VKungurtsev_LLampariello_GScutari_2021} devised an SQP method based on ``ghost penalties" with diminishing stepsizes for inequality constrained optimization. The authors proved convergence results  under a variety of scenarios, including when an initial feasible point is known, in which case a fixed stepsize can be employed to obtain a faster convergence rate. Other approaches with complexity results include \cite{CCartis_NIMGould_PLToint_2013}, which adds a quadratic penalty term involving the constraints to a cubic subproblem at each iteration. When only a first-order method is employed and everywhere LICQ holds, the resulting complexity result is $\mathcal{O}(\epsilon_\ell^{-2})$ and $\mathcal{O}(\epsilon_c^{-2})$.} 

\rev{For first-order, equality deterministic constrained optimization, the best known result (without assuming the KL property), for an} SQP method is given in \cite{FECurtis_MJONeill_DPRobinson_2024}, which proved a worst-case complexity result of $\mathcal{O}(\epsilon_{\ell}^{-2})$ and $\mathcal{O}(\epsilon_c^{-1})$ (this result holds deterministically, not just in expectation).
This work also proved a result for the stochastic SQP method of \cite{ASBerahas_FECurtis_DPRobinson_BZhou_2021}, which was shown to have a worst-case complexity of $\mathcal{O}(\epsilon_{\ell}^{-4})$ and $\mathcal{O}(\epsilon_c^{-2})$ \rev{when a lower bound on the merit parameter is known a-priori} and $\tilde{\mathcal{O}}(\epsilon_{\ell}^{-4})$ and $\tilde{\mathcal{O}}(\epsilon_c^{-2})$ otherwise, where $\tilde{\mathcal{O}}$ ignores logarithmic factors. In terms of $\epsilon_{\ell}$, this result is optimal, due to information theoretic lower bounds for stochastic gradient methods \cite{YArjevani_YCarmon_JCDuchi_DJFoster_NSrebro_BWoodworth_2023}. However, with respect to the constraint violation, it turns out that this result can be improved. We show that the worst-case complexity of the two stepsize stochastic SQP method proposed in this work has a worst-case complexity of $\mathcal{O}(\epsilon_{\ell}^{-4})$ and $\mathcal{O}(\epsilon_c^{-1})$. That is, in terms of convergence in the constraint violation, this result matches that of a \textbf{deterministic} SQP method, modulo the expectation and constant factors. Furthermore, we avoid unnecessary assumptions which were required to derive a complexity result in \cite{FECurtis_MJONeill_DPRobinson_2024} by not estimating a merit parameter during the course of the algorithm. Previously this parameter was estimated using stochastic gradient information, which may be highly inaccurate on any given iteration and thus required additional assumptions in order to ensure convergence. In addition to these results, a number of other works have also proposed methods with known worst-case complexity results for solving \eqref{eq:fdef}, including augmented Lagrangian \cite{LJin_XWang_2022,QShi_XWang_HWang_2026} and stochastic SQP methods \cite{SNa_MAnitescu_MKolar_2023A}. A summary of these worst-case complexity results is given in Table \ref{tab:complexity}.

\begin{table}[t]
\begin{center}
\begin{tabular}{|c|c|c|c|}
\hline
{\bf Algorithm} & {\bf Conditions} & {\bf Stationarity} & {\bf Feasibility} \\ 
\hline
SPD \cite{LJin_XWang_2022} & N/A & $\mathcal{O}\left(\epsilon_\ell^{-6}\right)$ & $\mathcal{O}\left(\epsilon_c^{-6}\right)$ \\
SPD \cite{LJin_XWang_2022} & $x_0$ feasible & $\mathcal{O}\left(\epsilon_\ell^{-5}\right)$ & $\mathcal{O}\left(\epsilon_c^{-5}\right)$ \\
\hline
MLALM \cite{QShi_XWang_HWang_2026} & N/A & \rev{$\mathcal{O}\left(\epsilon_\ell^{-4}\right)$} & \rev{$\mathcal{O}\left(\epsilon_c^{-4}\right)$} \\
MLALM \cite{QShi_XWang_HWang_2026} & $x_0$ near feasible & \rev{$\mathcal{O}\left(\epsilon_\ell^{-3}\right)$} & \rev{$\mathcal{O}\left(\epsilon_c^{-3}\right)$} \\
\hline
SSQP-AL \cite{SNa_MAnitescu_MKolar_2023A}  & N/A &  $\mathcal{O}\left(\epsilon_\ell^{-4}\right)$     & $\mathcal{O} \left(\epsilon_c^{-4}\right)$ \\
\hline
SSQP \cite{FECurtis_MJONeill_DPRobinson_2024} & $\taumin$ known       &  $\mathcal{O}\left(\epsilon_\ell^{-4}\right)$     & $\mathcal{O} \left(\epsilon_c^{-2}\right)$ \\
SSQP \cite{FECurtis_MJONeill_DPRobinson_2024} & $\taumin$ unknown  & $\tilde{\mathcal{O}}\left(\epsilon_\ell^{-4}\right)$     & $\tilde{\mathcal{O}} \left(\epsilon_c^{-2}\right)$ \\
\hline
SSQP-AS \cite{ASBerahas_RBollapragada_BZhou_2022}     & N/A &  $\mathcal{O}\left(\epsilon_\ell^{-4}\right)$     & $\mathcal{O} \left(\epsilon_c^{-2}\right)$ \\
\hline
Algorithm \ref{alg:tsssqp} & non-adaptive &  $\mathcal{O}\left(\epsilon_\ell^{-4}\right)$     & $\mathcal{O}\left(\epsilon_c^{-1}\right)$ \\
Algorithm \ref{alg:tsssqp} & adaptive &  $\tilde{\mathcal{O}}\left(\epsilon_\ell^{-4}\right)$     & $\tilde{\mathcal{O}}\left(\epsilon_c^{-1}\right)$ \\
\hline
\end{tabular}
\caption{Sample complexity of algorithms for solving \eqref{eq:fdef}. Convergence of each algorithm is proven underneath slightly different conditions. All methods except MLALM assume that the Jacobian has full rank at each iteration, while MLALM assumes a certain constraint qualification as well as mean-squared smoothness of the stochastic gradients. SSQP and SSQP-AS also make additional assumptions on the behavior of the merit parameter. \vspace{-0.5cm}}
\label{tab:complexity} 
\end{center}
\end{table}

Unfortunately, the complexity result we prove for our initial algorithm requires certain choices of the stepsizes based on the potential difficultly of estimating parameters of the problem (such as Lipschitz constants and a reasonable setting of the merit parameter). To remedy this, we propose a variant of our method which incorporates stepsizes inspired by adaptive gradient methods for unconstrained stochastic optimization \cite{JDuchi_EHazan_YSinger_2011,BHMcMahan_MStreeter_2010,RWard_XWu_LBottou_2020}. Specifically, we build upon the methodolgy commonly known as Adagrad-Norm, which estimates a stepsize using the prior stochastic gradient estimates. We show that we can generate both of the stepsizes used by our algorithm under this framework and derive a worst-case complexity result for this variant of our method of the order $\tilde{\mathcal{O}}(\epsilon_{\ell}^{-4})$ and $\tilde{\mathcal{O}}(\epsilon_c^{-1})$, without requiring any knowledge of problem specific constants. In addition, both versions of our algorithm guarantee convergence when the stepsizes to be relaxed to lie in a certain set, from which the actual stepsize can be chosen, as was originally proposed in \cite{ASBerahas_FECurtis_DPRobinson_BZhou_2021}. In order to choose a stepsize from this set, we propose a safeguarded linesearch in terms of the constraint violation and show how this can be implemented when the safeguarding is done in terms of the adaptive stepsize rule based on Adagrad-Norm. Finally, we provide preliminary numerical experiments for our algorithm and show that it compares favorably with a state of the art methods. These numerical experiments also demonstrate faster convergence in constraint violation when compared with previously proposed stochastic SQP methods, providing confirmation of our theoretical results.

The rest of this work is organized as follows. In Section \ref{sec:alg}, we formally define and discuss our proposed algorithm and prove some basic properties. We provide a worst-case complexity analysis in Section \ref{sec:analysis} for two variants of our algorithm. A safeguarded linesearch procedure is developed in Section \ref{sec:ls} and numerical experiments are presented in Section \ref{sec:numerical}. We provide concluding remarks in Section \ref{sec:conclusion}.

\subsection{Notation}
We adopt the notation that $\| \cdot \|$ denotes the $\ell_2$-norm for vectors and the vector-induced $\ell_2$-norm
for matrices. The set of
nonnegative integers is denoted as $\N := \{0, 1, 2, \dots, \}$  and we denote the positive real numbers by $\R_{>0}$.

Given $\phi : \R{} \to \R{}$ and $\varphi : \R{} \to [0,\infty)$, we write $\phi(\cdot) = \mathcal{O}(\varphi(\cdot))$ to indicate that $|\phi(\cdot)| \leq c\varphi(\cdot)$ for some $c \in (0,\infty)$.  Similarly, we write $\phi(\cdot) = \tilde{\mathcal{O}}(\varphi(\cdot))$ to indicate that $|\phi(\cdot)| \leq c\varphi(\cdot)|\log^{\bar{c}}(\cdot)|$ for some $c \in (0,\infty)$ and $\bar{c} \in (0,\infty)$.  In this manner, one finds that $\mathcal{O}(\varphi(\cdot)|\log^{\bar{c}}(\cdot)|) \equiv \tilde{\mathcal{O}}(\varphi(\cdot))$ for any $\bar{c} \in (0,\infty)$.

The algorithm that we analyze is iterative, generating in each realization a sequence $\{x_k\}$. We also append the iteration number to other quantities corresponding an iteration, e.g., $f_k := f(x_k)$ for all $k \in \N$.

\subsection{Assumptions and Background}

Throughout, we require the following assumptions on $f$ and $c$:
\begin{assumption}\label{assum:fcsmooth}
  The objective function $f : \R^n \rightarrow \R$ is continuously differentiable and bounded below by $\flow \in \R{}$ and the corresponding gradient function $\nabla f: \R^n \rightarrow \R^n$ is bounded and Lipschitz continuous with constant $L \in (0,\infty)$. The constraint function $c : \R^n \rightarrow \R^m$ $($where $m \leq n$$)$ and the corresponding Jacobian function $J := \nabla c^\top : \R^n \rightarrow \R^{m\times n}$ are bounded, each gradient function $\nabla c_i : \R^{n} \rightarrow \R^{n}$ is Lipschitz continuous with constant $\gamma_i$ for all $i \in \{1,\dots,m\}$, and the singular values of $J \equiv \nabla c^\top$ are bounded below and away from zero.
\end{assumption}

Under this assumption both the gradient of $f$ and the constraint violation are bounded in norm by constants. We denote these constants as $\|\nabla f(x)\| \leq \kappa_g$ and $\|c_k\|_1 \leq \kappa_c$. \rev{In addition, we denote the Lipschitz constant of the Jacobian as $\Gamma := \sum_{i=1}^m \gamma_i$, where $\gamma_i$ is the Lipschitz constant of each $\nabla c_i^{T}$.}

Defining the Lagrangian $\ell : \mathbb{R}^n \times \R^m \rightarrow \R$ corresponding to \eqref{eq:fdef} by $\ell(x,y) := f(x) + c(x)^\top y$, first-order primal-dual stationarity conditions for~\eqref{eq:fdef}, which are necessary for optimality under Assumption~\ref{assum:fcsmooth}, are given by
\begin{equation} \label{eq:optconds}
  0 =
  \left[ \begin{matrix}
    \nabla_x \ell(x,y) \\
    \nabla_y \ell(x,y)
  \end{matrix} \right]
  =
  \left[ \begin{matrix}
    \nabla f(x) + \nabla c(x) y \\
    c(x)
  \end{matrix} \right].
\end{equation}
We note that the complexity measure \eqref{eq:complexity} is simply an approximate version of these optimality conditions.

As stated above, our algorithm generates a search direction at iteration $k$ by solving the following quadratic optimization problem:
\begin{equation} \label{eq:sqpsubproblem}
    \underset{p \in \mathbb{R}^n}{\min} \ f_k + g_k^T p + \frac12 p^T H_k p \ \ \text{subject to} \ \ c_k + J_k p = 0,
\end{equation}
where $g_k$ is the current stochastic gradient estimate. It is well known that this is equivalent to solving the ``Newton SQP system":
\begin{equation} \label{eq:linsys}
    \left[\begin{matrix}
    H_k & J_k^T \\
    J_k & 0
    \end{matrix}
    \right]
    \left[\begin{matrix}
    p_k \\
    y_k
    \end{matrix}
    \right] =
    -\left[\begin{matrix}
    g_k \\
    c_k
    \end{matrix}
    \right].
\end{equation}
In order to ensure the solution of this sub-problem is unique, we require the following assumption on $H_k$.

\begin{assumption} \label{assum:Hbound}
    The sequence $\{H_k\}$ is bounded in norm by $\kappa_H \in \R_{>0}$. In addition, there exists a constant $\zeta \in \R_{>0}$ such that, for all $k \in \N$, the matrix $H_k$ has the property that $u^T H_k u \geq \zeta \|u\|^2$ for all $u \in \R^n$ such that $J_k u = 0$.
\end{assumption}

In order to analyze our algorithm, we utilize the $\ell_1$ merit function $\phi : \R^n \times \R_{>0} \rightarrow \R$:
\begin{equation} \label{eq:phidef}
    \phi(x,\tau) = \tau f(x) + \|c(x)\|_1.
\end{equation}

In the above equation, $\tau$ is the merit parameter which balances between the function value and constraint violation. For the analysis, we also use the following local model of the merit function $l : \R^n \times \R_{>0} \times \R^n \rightarrow \R$ defined as
\begin{equation} \label{eq:qdef}
    l(x, \tau, d) = \tau(f(x) + \nabla f(x)^T d)   + \|c(x) + \nabla c(x)^T d\|_1.
\end{equation}

In addition, we consider the reduction in the model for a direction $d \in \R^n$ with $c(x) + \nabla c(x)^T d = 0$ which is $\Delta l : \R^n \times \R_{>0} \times \R^n \rightarrow \R$ defined as
\begin{equation} \label{eq:deltal}
    \begin{split}
        \Delta l(x, \tau, d) &:= l(x, \tau, 0) - l(x, \tau, d) \\
        &= -\tau \nabla f(x)^T d + \|c(x)\|_1.
    \end{split}
\end{equation}

We wish to stress here that unlike previous works, we do not attempt to estimate a good value of $\tau$. We choose to avoid this as previous work relied upon strong assumptions (such as uniformly bounded stochastic gradients \cite{ASBerahas_FECurtis_DPRobinson_BZhou_2021}, sub-Gaussian stochastic gradients \cite{FECurtis_MJONeill_DPRobinson_2024}, \rev{or direct assumptions on ``good behavior" of the merit parameter \cite{ASBerahas_FECurtis_MJONeill_DPRobinson_2023,FECurtis_DPRobinson_BZhou_2024B,FECurtis_DPRobinson_BZhou_2024A}, which can be implied by the prior assumptions}) in order to prove \rev{the existence of a lower bound on a stochastically estimated merit parameter sequence}. By choosing to relegate the merit function and parameter exclusively to the analysis, we are able to avoid overcomplicating the analysis and adding unnecessary assumptions. \rev{We note that other work on complexity of SQP methods has taken a similar approach, notably \cite{FFacchinei_VKungurtsev_LLampariello_GScutari_2021}.}

\section{Algorithm and Basic Properties} \label{sec:alg}

Recall that at each iteration, a search direction $p_k$ is computed as the solution of \eqref{eq:linsys}. We assume that this step computation is performed in such a way that the orthogonal decomposition
\begin{equation} \label{eq:decomp}
    p_k = u_k + v_k \ \text{where} \ u_k \in \text{Null}(J_k) \ \text{and} \ v_k \in \text{Range}(J_k^T),
\end{equation}
is known\footnote{This is not necessary in certain circumstances, see Remark \ref{rem:scaleHk} for details.}. One important consequence of this decomposition is that the normal component, $v_k$, does not depend on the current stochastic gradient estimate $g_k$. Unlike prior work, we do not directly use $p_k$ as our search direction. Instead, we rescale the tangential component, $u_k$, in order to generate our search direction $d_k$ as follows,
\begin{equation} \label{eq:dkdef}
    d_k = \beta_k u_k + v_k,
\end{equation}
where $\beta_k \in \R_{>0}$. Then, we find the next iterate $x_{k+1}$ by setting $x_{k+1} = x_k + \alpha_k d_k$ for some $\alpha_k \in \mathbb{R}_{>0}$. This procedure is formalized in Algorithm \ref{alg:tsssqp}.

\begin{algorithm}[ht]
  \caption{Generic Two Stepsize Stochastic SQP Algorithm}
  \label{alg:tsssqp}
  \begin{algorithmic}[1]
    \Require $x_0 \in \R^{n}$;
    \For{$k=0,1,\dots$}
    \State Compute stochastic gradient $g_k$.
	  \State Compute $(p_k,y_k)$ as the solution of \eqref{eq:linsys}.
   \State Choose $\beta_k \in \mathbb{R}_{>0}$.
    \State Set $d_k \leftarrow v_k + \beta_k u_k$, where $v_k \in \text{Range}(J_k^T)$ and $u_k \in \text{Null}(J_k)$ are the orthogonal decomposition of $p_k$.
    \State Choose $\alpha_k \in \mathbb{R}_{>0}$.
    \State Set $x_{k+1} \leftarrow x_k + \alpha_k d_k$.
    \EndFor
  \end{algorithmic}
\end{algorithm}

The choice of $\beta_k$ is crucial to ensure convergence of our algorithm, as it controls the variance of the stochastic gradient estimates and plays a similar role as the stepsize in stochastic gradient methods. As such, it is natural to consider $\beta_k$ to be quite small. Indeed, to ensure our \rev{complexity} result, we set $\beta_k = O(1/\sqrt{K})$, where $K$ is the total number of iterations we intend to perform. On the other hand, $\alpha_k$ does \textbf{not} need to control the error in the stochastic gradients and thus may be set independent of $K$.  Thus, $v_k$, which is the component of $d_k$ that drives the algorithm towards constraint satisfaction, is only scaled by a stepsize which is independent of $K$. This is the key insight that leads to our improved complexity.

Algorithm \ref{alg:tsssqp} is written generically, without specifying how to choose the stepsizes $\alpha_k$ and $\beta_k$. We consider two variants for choosing these stepsizes in Section \ref{sec:analysis} and analyze their behavior. First, in Section \ref{subsec:prechosen}, we consider the case where $\beta_k$ is defined by a pre-specified sequence and
\begin{equation}
    \alpha_k \in [\nu, \nu + \theta \beta_k],
\end{equation}
where $\nu \in \mathbb{R}_{>0}$ and  $\theta \in \mathbb{R}_{>0}$. This case is essentially equivalent to the standard stochastic gradient regime with a pre-specified stepsize sequence (modulo the relaxation of $\alpha_k$ into a range, which was originally suggested for stochastic SQP methods in \cite{ASBerahas_FECurtis_DPRobinson_BZhou_2021}). For this method, we prove the complexity result foreshadowed in Section \ref{sec:intro}. However, this result only holds under certain conditions on $\nu$ which depend on the Lipschitz constants of the gradient of $f$ and Jacobian of $c$ as well as a good estimate of the merit parameter $\tau$. Unfortunately, it may not be reasonable to estimate these parameters a-priori.

To remedy this, in Section \ref{subsec:adaptive} we analyze a version of Algorithm \ref{alg:tsssqp} which utilizes adaptive stepsizes based on Adagrad-Norm, which is a popular approach for choosing stepsizes in the stochastic gradient literature \cite{JDuchi_EHazan_YSinger_2011,BHMcMahan_MStreeter_2010,RWard_XWu_LBottou_2020}. In this case, some additional logarithmic factors appear in the final complexity result, but this approach does not require any knowledge of the Lipschitz constants or the merit parameter.

In addition, in both of the cases we analyze in Section \ref{sec:analysis}, $\alpha_k$ may be chosen from a specific range. In Section \ref{sec:ls}, we describe a safeguarded line search procedure which can be used to determine $\alpha_k$. We take adavantage of the assumption that the constraints are deterministic and design a line search which only relies on the constraint violation and does not include stochastic gradient information when computing an $\alpha_k$. In addition, we provide a fully specified algorithm in Algorithm \ref{alg:tsssqpls} that combines this linesearch procedure with an adaptive lower bound based on the Adagrad-Norm stepsizes developed in Section \ref{subsec:adaptive}.

\subsection{Properties of Algorithm \ref{alg:tsssqp}}

First, we restate a basic result from \cite{ASBerahas_FECurtis_DPRobinson_BZhou_2021}.
\begin{lemma} \label{lem:vbound}  (\cite[Lemma 2.9]{ASBerahas_FECurtis_DPRobinson_BZhou_2021})
    There exists $\kappa_v \in \R_{>0}$ such that, for all $k \in \N$, the normal component $v_k$ satisfies $\max\{\|v_k\|, \|v_k\|^2\} \leq \kappa_v \|c_k\|$.
\end{lemma}

During the analysis of our algorithm, it is often useful to consider the ``true" step computation that would occur if the step was computed using the true gradient, $\nabla f(x_k)$, in place of the stochastic gradient estimate, $g_k$. Specifically, let $(\pktrue, \yktrue)$ be the solution of the linear system:
\begin{equation} \label{eq:linsystrue}
    \left[\begin{matrix}
    H_k & J_k^T \\
    J_k & 0
    \end{matrix}
    \right]
    \left[\begin{matrix}
    \pktrue \\
    \yktrue
    \end{matrix}
    \right] =
    -\left[\begin{matrix}
    \nabla f(x_k) \\
    c_k
    \end{matrix}
    \right].
\end{equation}
In addition, we define
\begin{equation} \label{eq:dktrue}
    \dktrue = \beta_k \uktrue + v_k,
\end{equation}
where $\pktrue = \uktrue + v_k$ with $\uktrue \in \text{Null}(J_k)$ (we recall here that $v_k$ is independent of $g_k$ and $\nabla f(x_k)$ and thus is the same $v_k$ as in \eqref{eq:decomp}).

\begin{lemma} \label{lem:uktruebound}
    Let Assumptions \ref{assum:fcsmooth} and \ref{assum:Hbound} hold. Then,
    \begin{equation*}
        \|\uktrue\| \leq \zeta^{-1} \|\nabla f(x_k)\| + \zeta^{-1} \kappa_H \kappa_v \|c_k\| \leq \zeta^{-1} \kappa_g + \zeta^{-1} \kappa_H \kappa_v \kappa_c =: \kappa_u.
    \end{equation*}
\end{lemma}

\begin{proof}
    By the first equation of \eqref{eq:linsystrue} and the definition of $\uktrue$, we have \\$ (\uktrue)^T H_k (\uktrue + v_k) = -\nabla f(x_k)^T \uktrue$. Then, by Assumption \ref{assum:Hbound} and Lemma \ref{lem:vbound},
    \begin{align*}
        \zeta \|\uktrue\|^2
        &\leq (\uktrue)^T H_k \uktrue \\
        &= -\nabla f(x_k)^T \uktrue - v_k^T H_k \uktrue \\
        &\leq \|\nabla f(x_k)\| \|\uktrue\| + \kappa_H \kappa_v \|c_k\| \|\uktrue\|.
    \end{align*}
    Dividing this inequality through by $\|\uktrue\|$ proves the first result. The final result follows by Assumption \ref{assum:fcsmooth} and Lemma \ref{lem:vbound}.
    \qed
    \end{proof}

Now we state an important property about the merit parameter $\tau$.

\begin{lemma} \label{lem:taumin}
    Let Assumptions \ref{assum:fcsmooth} and \ref{assum:Hbound} hold and let $\sigma \in (0,1)$. Let $\beta_k \leq \kappa_{\beta}$ hold for all $k$ and define
    \begin{equation} \label{eq:taumin}
        \taumin := \frac{1-\sigma}{\kappa_v (\kappa_\beta \kappa_H \kappa_u + \kappa_g)}.
    \end{equation}
    
    Then,
    \begin{equation} \label{eq:tauminineq}
        \taumin \left(\nabla f(x_k)^T \dktrue + \beta_k (\uktrue)^T H_k \uktrue \right) \leq (1-\sigma) \|c_k\|_1.
    \end{equation}
\end{lemma}

\begin{proof}
    By \eqref{eq:linsystrue} and the definition of $\uktrue$,
    \begin{align*}
        \nabla f(x_k)^T \dktrue &= \nabla f(x_k)^T (\beta_k \uktrue + v_k) \\
        &= - \beta_k (\uktrue)^T H_k \uktrue - \beta_k v_k^T H_k \uktrue + \nabla f(x_k)^T v_k.
    \end{align*}
    Thus, by Assumptions \ref{assum:fcsmooth} and \ref{assum:Hbound}, Lemma \ref{lem:vbound} and Lemma \ref{lem:uktruebound},
    \begin{align*}
        \nabla f(x_k)^T \dktrue + \beta_k (\uktrue)^T H_k \uktrue &= - \beta_k v_k^T H_k \uktrue + \nabla f(x_k)^T v_k \\
        &\leq (\beta_k \kappa_H \|\uktrue\| + \|\nabla f(x_k)\|) \|v_k\| \\
        &\leq \kappa_v (\kappa_\beta \kappa_H \kappa_u + \kappa_g) \|c_k\|_1.
    \end{align*}
    Combining this with \eqref{eq:taumin}, proves \eqref{eq:tauminineq}.
    \qed
\end{proof}

\begin{remark} \label{rem:scaleHk}
Under the condition that $H_k$ preserves the null space of $J_k$ (i.e. for any $u \in \text{Null}(J_k), H_k u \in \text{Null}(J_k)$), we can sidestep the requirement to compute the orthogonal decomposition of $p_k$ by simply rescaling the matrix $H_k$ by $\beta_k^{-1}$ and directly use the computed direction as $d_k$. \rev{We note that this} additional requirement is necessary when using rescaling in order to prove a result similar to Lemma \ref{lem:taumin}, as otherwise the crossing term $v_k^T H_k \uktrue$ picks up a factor of $\beta_k^{-1}$. This, in turn, means that it is not possible to provide a bound on $\taumin$ that is independent of a \textbf{lower} bound on $\beta_k$, thus causing serious issues in the final complexity result. For the sake of generality, we don't consider this rescaling approach, though when $H_k$ preserves the nullspace of $J_k$ our results still hold, albeit with potentially different constant factors.

\rev{In cases where this condition does not hold, one can compute the step decomposition in a variety of ways. For example, one could directly solve the quadratic subproblem using the nullspace method (see \cite[Section 16.2]{JNocedal_SJWright_2006}), which directly solves for $u_k$ and $v_k$. Alternatively, a simple approach is to compute $p_k$ by \eqref{eq:linsys} and then compute the projection of $p_k$ onto $\text{Null}(J_k)$ to find $u_k$ (and subsequently $v_k$). We take this approach throughout our numerical experiments and observe it incurs only a minor additional cost, see Appendix \ref{app:time} for details. In addition, if one allows for inexact computation of  $v_k$ (which is outside the scope of the current manuscript), a stepsize decomposition such as the one employed in \cite{ASBerahas_FECurtis_MJONeill_DPRobinson_2023} could be used, where $v_k$ is computed as the Cauchy step of a simple trust region subproblem while $u_k$ is computed by solving a linear system comparable to \eqref{eq:linsys}. In such a case, the cost compute $u_k$ and $v_k$ is comparable to the cost to compute $p_k$ by \eqref{eq:linsys}.}
\end{remark}

A direct consequence of the previous lemma is
\begin{equation} \label{eq:Dllb}
    \Delta l(x_k, \taumin, \dktrue) \geq \taumin \beta_k (\uktrue)^T H_k \uktrue + \sigma \|c_k\|_1,
\end{equation}
which will be used to prove the final convergence result. Given this inequality, it should be clear that with an upper bound on $\Delta l$, we would expect convergence in the constraint violation. To see the connection between the quantities in \eqref{eq:Dllb} and first order stationarity, we prove the following lemma, which shows that the quadratic term can be lower bounded in terms of the gradient of the Lagrangian at $x_k$ for a specific Lagrange multipler.

\begin{lemma} \label{lem:ulb}
    Let Assumptions \ref{assum:fcsmooth} and \ref{assum:Hbound} hold. Then,
    \begin{equation*}
        (\uktrue)^T H_k \uktrue \geq \zeta \kappa_H^{-2} \|\nabla f(x_k) + J_k^T \yktrue\|^2 - (1 +2 \kappa_u) \zeta \kappa_v \|c_k\|_1.
    \end{equation*}
\end{lemma}

\begin{proof}
    By Assumption \ref{assum:Hbound},
    \begin{equation*}
        (\uktrue)^T H_k \uktrue \geq \zeta \|\uktrue\|^2 \geq \zeta \kappa_H^{-2} \|H_k \uktrue\|^2.
    \end{equation*}
    Then, by \eqref{eq:linsystrue} and Lemmas \ref{lem:vbound} and \ref{lem:uktruebound},
    \begin{align*}
        &\|H_k \uktrue + H_k v_k - H_k v_k\|^2 \\
        &\quad= \|H_k \uktrue + H_k v_k\|^2 - 2 v_k^T H_k H_k (\uktrue + v_k) + \|H_k v_k\|^2 \\
        &\quad\geq \|\nabla f(x_k) + J_k^T \yktrue\|^2 - 2 v_k^T H_k H_k \uktrue - \|H_k v_k\|^2 \\
        &\quad\geq \|\nabla f(x_k) + J_k^T \yktrue\|^2 - (1 +2 \kappa_u) \kappa_H^2 \kappa_v \|c_k\|_1,
    \end{align*}
    which proves the result. \qed
\end{proof}

Thus, given these results, we can see that the convergence rate in terms of the gradient of the Lagrangian should be directly related to the choice of $\beta_k$ while convergence in the constraint violation will be largely independent of this stepsize (provided it is chosen to sufficiently control the noise in $g_k$). This is in contrast to the results in \cite{FECurtis_MJONeill_DPRobinson_2024}, where the norm of the constraint violation is multiplied by $\beta_k$ and is the root cause of the improvement in the complexity result for the constraint violation that we prove in the sequel.

We finish this subsection with the following generic descent lemma.

\begin{lemma} \label{lem:descentlemma}
        Let Assumptions \ref{assum:fcsmooth} and \ref{assum:Hbound} hold. Then, with $\taumin$ defined as in \eqref{eq:taumin},
    \begin{align} 
        &\phi(x_k + \alpha_k d_k, \taumin) - \phi(x_k, \taumin) \nonumber \\
        &\quad\leq -\alpha_k \Delta l(x_k, \taumin, \dktrue) + \frac{\alpha_k^2 \beta_k^2}{2}(\taumin L + \Gamma) \|u_k\|^2\label{eq:descent} \\
        &\quad\quad+ \frac{\alpha_k^2}{2} (\kappa_v(\taumin L + \Gamma) + 4)  \|c_k\|_1 + \alpha_k \taumin \nabla f(x_k)^T (d_k - \dktrue). \nonumber
    \end{align}
\end{lemma}

\begin{proof}
    By $L$-Lipschitz continuity of $\nabla f(x)$ and $\Gamma$-Lipschitz continuity of $J_k$, we have
    \begin{align*}
        &\phi(x_k + \alpha_k d_k, \taumin) - \phi(x_k, \taumin) \\\
        &\leq \alpha_k \taumin \nabla f(x_k)^T d_k + \|c_k + \alpha_k J_k d_k\|_1 - \|c_k\|_1 + \frac{\alpha_k^2}{2}(\taumin L + \Gamma) \|d_k\|^2 \\
        &= \alpha_k \taumin \nabla f(x_k)^T \dktrue + |1-\alpha_k| \|c_k\|_1 - \|c_k\|_1 + \frac{\alpha_k^2}{2}(\taumin L + \Gamma) \|d_k\|^2 \\
        &\quad+ \alpha_k \taumin \nabla f(x_k)^T (d_k - \dktrue),
    \end{align*}
    where the equality follows from $J_k d_k = -c_k$.

    \rev{Using the fact that $|1-\alpha_k| \leq 1 - \alpha_k + 2\alpha_k^2$, we have
    \begin{align*}
        &\phi(x_k + \alpha_k d_k, \taumin) - \phi(x_k, \taumin) \\
        &\leq \alpha_k \taumin \nabla f(x_k)^T \dktrue -\alpha_k \|c_k\|_1 + 2 \alpha_k^2 \|c_k\|_1 + \frac{\alpha_k^2}{2}(\taumin L + \Gamma) \|d_k\|^2 \\
        &\quad+ \alpha_k \taumin \nabla f(x_k)^T (d_k - \dktrue).
    \end{align*}
    }Then, using the orthogonal decomposition $d_k = \beta_k u_k + v_k$, Lemma \ref{lem:vbound}, and \eqref{eq:deltal}, we have
    \begin{align*}
        &\phi(x_k + \alpha_k d_k, \taumin) - \phi(x_k, \taumin) \nonumber \\
        &\leq \alpha_k \taumin \nabla f(x_k)^T \dktrue - \alpha_k \|c_k\|_1 + 2\alpha_k^2 \|c_k\|_1 \\
        &\quad+ \frac{\alpha_k^2}{2}(\taumin L + \Gamma) (\beta_k^2 \|u_k\|^2 + \|v_k\|^2) + \alpha_k \taumin \nabla f(x_k)^T (d_k - \dktrue) \\
        &\leq -\alpha_k \Delta l(x_k, \taumin, \dktrue) + \frac{\alpha_k^2 \beta_k^2}{2}(\taumin L + \Gamma) \|u_k\|^2 \\
        &\quad+ \frac{\alpha_k^2}{2} (\kappa_v(\taumin L + \Gamma) + 4)  \|c_k\|_1 + \alpha_k \taumin \nabla f(x_k)^T (d_k - \dktrue),
    \end{align*}
    proving the result. \qed
\end{proof}

\subsection{Stochastic Assumptions and Properties}

In order to analyze the convergence of our algorithm, let $\mathcal{F}_k$ denote the natural filtration adapted to Algorithm \ref{alg:tsssqp} and let $\E_k[\cdot] = \E[\cdot | \mathcal{F}_k]$. Under these definitions, we have the following assumption on our stochastic gradient estimates, $g_k$.

\begin{assumption}\label{assum:stochG}
  There exists $M \in \R{}_{>0}$ such that, for all $k$, one finds
  \begin{equation}\label{eq:sg}
    \E_k[g_k] = \nabla f(x_k)\ \ \text{and}\ \ \E_k[\|g_k - \nabla f(x_k)\|_2^2] \leq M.
  \end{equation}
\end{assumption}

This assumption is largely standard in the stochastic gradient literature. We note that relaxing the uniformly bounded variance assumption to an assumption which allows the variance to grow with the norm of the gradient of $f$ (such as in \cite{LBottou_FECurtis_JNocedal_2018}) is no more general under Assumption \ref{assum:fcsmooth} since $\|\nabla f(x)\| \leq \kappa_g$.

Under Assumption \ref{assum:stochG}, we have the following properties.

\begin{lemma} \label{lem:stochasticbounds}
    Let Assumptions \ref{assum:fcsmooth}, \ref{assum:Hbound}, and \ref{assum:stochG} hold. Then, $\E_k[u_k] = \uktrue$, $\E_k[y_k] = \yktrue$,
    \begin{equation*}
        \E_k[\|u_k-\uktrue\|^2] \leq \zeta^{-2} M,
    \end{equation*}
    and
    \begin{equation*}
        \rev{\E_k[\|u_k\|^2] \leq \|\uktrue\|^2 + \zeta^{-2} M \leq \zeta^{-1} (\uktrue)^T H_k \uktrue + \zeta^{-2} M.}
    \end{equation*}
\end{lemma}

\begin{proof}
    The first two claims follow directly by the statement of \cite[Lemma 3.8]{ASBerahas_FECurtis_DPRobinson_BZhou_2021} To prove the third result, let $Z_k$ be an orthogonal basis for the null space of $J_k$ (which, by Assumption \ref{assum:fcsmooth} is a matrix in $\R^{n \times (n-m)})$ and let $u_k = Z_k w_k$ and $\uktrue = Z_k \wktrue$. Then, by \eqref{eq:linsys}, it follows that
    \begin{equation*}
        Z_k w_k = - Z_k(Z_k^T H_k Z_k)^{-1} Z_k^T (g_k + H_k v_k).
    \end{equation*}
    Similarly,
    \begin{equation*}
        Z_k \wktrue = - Z_k(Z_k^T H_k Z_k)^{-1} Z_k^T (\nabla f(x_k) + H_k v_k),
    \end{equation*}
    so that
    \begin{equation*}
        u_k - \uktrue = Z_k(Z_k^T H_k Z_k)^{-1} Z_k^T (\nabla f(x_k) - g_k)
    \end{equation*}
    and thus, by Assumptions \ref{assum:Hbound} and \ref{assum:stochG},
    \begin{equation*}
        \E_k[\|u_k - \uktrue\|^2] \leq \E_k[\|Z_k(Z_k^T H_k Z_k)^{-1} Z_k^T\|^2 \|\nabla f(x_k) - g_k\|^2] \leq \zeta^{-2} M.
    \end{equation*}

    \rev{The final set of inequalities follows directly by the first and third claim, with final inequality due to Assumption \ref{assum:Hbound}.}
    \qed
\end{proof}

\rev{In addition, we have the following property under the stronger assumption that $\beta_k$ is measurable to $\mathcal{F}_k$.}

\begin{lemma} \label{lem:stochasticdkbound}
    \rev{Let Assumptions \ref{assum:fcsmooth}, \ref{assum:Hbound}, and \ref{assum:stochG} hold and let $\beta_k$ be measurable to $\mathcal{F}_k$. Then, $\E_k[d_k] = \dktrue$ and 
    \begin{equation*}
        \E_k[\|d_k-\dktrue\|] \leq \beta_k \zeta^{-1} \sqrt{M}.
    \end{equation*}}
\end{lemma}

\begin{proof}
    Since $\beta_k$ is measurable to $\mathcal{F}_k$, it follows that
    \begin{equation*}
        \E_k[d_k] = \beta_k \E_k[u_k] + v_k = \beta_k \uktrue + v_k = \dktrue.
    \end{equation*}
    For the second result, we have that 
    \begin{equation*}
        d_k - \dktrue = \beta_k(u_k - \uktrue) = \beta_k Z_k(Z_k^T H_k Z_k)^{-1} Z_k^T (\nabla f(x_k) - g_k).
    \end{equation*}
    Thus, by Assumptions \ref{assum:Hbound} and \ref{assum:stochG}, as well as Jensen's inequality,
    \begin{equation*}
        \E_k[\|d_k - \dktrue\|] \leq \beta_k \E_k[\|Z_k(Z_k^T H_k Z_k)^{-1} Z_k^T\| \|\nabla f(x_k) - g_k\|] \leq \beta_k \zeta^{-1} \sqrt{M}.
    \end{equation*} \qed
\end{proof}

\section{Convergence Analysis} \label{sec:analysis}

In this section, we derive our main convergence results for two variants of Algorithm \ref{alg:tsssqp}, which differ on how $\alpha_k$ and $\beta_k$ are chosen at each iteration.

\subsection{Covergence with Pre-specified Stepsize Sequences} \label{subsec:prechosen}

Throughout this subsection, we analyze Algorithm \ref{alg:tsssqp} when $\{\beta_k\}$ is a pre-specified sequence  and $\alpha_k$ lies in a pre-specified range, i.e.,
\begin{equation} \label{eq:fixedab}
    \{\beta_k\} \subset \R_{>0}, \quad \quad \alpha_k \in [\nu, \nu + \theta \beta_k], \ \ \forall k,
\end{equation}
for some $\nu \in \mathbb{R}_{>0}$ and $\theta \in \mathbb{R}_{>0}$.

Under this stepsize scheme, we prove a preliminary result about the final term that appears in Lemma \ref{lem:descentlemma}.
\begin{lemma} \label{lem:innerprodfixed}
    Let Assumptions \ref{assum:fcsmooth}, \ref{assum:Hbound}, and \ref{assum:stochG} hold. Then,
    \begin{equation*}
        \E_k[\alpha_k \taumin \nabla f(x_k)^T(d_k - \dktrue)] \leq \beta_k^2 \theta \taumin \kappa_g \zeta^{-1} \sqrt{M}.
    \end{equation*}
\end{lemma}

\begin{proof}
    Let $\xi_k \in [0,1]$ be the random variable such that $\alpha_k = \nu + \xi_k \theta \beta_k$. Then, by Lemma \rev{\ref{lem:stochasticdkbound}} and the fact that $\nu$ and $\beta_k$ are measurable to $\mathcal{F}_k$,
    \begin{align*}
        \E_k[\alpha_k \taumin \nabla f(x_k)^T(d_k - \dktrue)] &= \E_k[(\nu + \xi_k \theta \beta_k)  \taumin \nabla f(x_k)^T(d_k - \dktrue)] \\
        &= \E_k[\xi_k \theta \beta_k  \taumin \nabla f(x_k)^T(d_k - \dktrue)] \\
        &\leq \E_k[\theta \beta_k  \taumin \|\nabla f(x_k)\|\|d_k - \dktrue\|] \\
        &\leq \beta_k^2 \theta \taumin \kappa_g \zeta^{-1} \sqrt{M}.
    \end{align*}
    \qed
\end{proof}

Now, we are ready to derive our first main result.
\begin{theorem} \label{thm:fixedstep}
    Let Assumptions \ref{assum:fcsmooth}, \ref{assum:Hbound}, and \ref{assum:stochG} hold. Let $\sigma \in (0,1)$, let $\{\beta_k\} \subset \mathbb{R}_{>0}$ be a pre-specified sequence such that $\beta_k \leq \kappa_\beta$ holds for all $k$, let $\alpha_k \in [\nu, \nu + \theta \beta_k]$, for some $\theta \in \mathbb{R}_{>0}$, $\nu \in (0, \sigma/(2\kappa_v(\taumin L + \Gamma) + 4)]$, and let $\taumin$ be defined as in Lemma \ref{lem:taumin}. Let
    \begin{align} \label{eq:kappa1}
        \kappa_1 &:= \frac{(\nu + \theta \kappa_\beta)^2(\taumin L + \Gamma) (\rev{\kappa_u^2} + \zeta^{-2} M)}{2} \\
        &\quad+ \theta(\theta \kappa_c (\kappa_v(\taumin L + \Gamma) + 4) + \taumin \kappa_g \zeta^{-1} \sqrt{M}). \nonumber
    \end{align}
    Then, for any $K \in \N$,
    \begin{equation} \label{eq:thm1bound}
        \begin{split}
            \sum_{k=0}^{K-1} \E[\alpha_k \beta_k \taumin (\uktrue)^T H_k \uktrue + \frac{\alpha_k \sigma}{2} \|c_k\|_1] \\
            \leq \taumin (f(x_0) - \flow) + \|c_0\|_1 + \kappa_1 \sum_{k=0}^{K-1} \beta_k^2.
        \end{split}
    \end{equation}
\end{theorem}

\begin{proof}
    Taking the conditional expectation on both sides of \eqref{eq:descent} and applying the results of Lemma \ref{lem:uktruebound}, Lemma \rev{\ref{lem:stochasticdkbound}}, and Lemma \ref{lem:innerprodfixed} (noting that $\beta_k$ is measurable to $\mathcal{F}_k$),
    \begin{align*}
        &\E_k[\phi(x_k + \alpha_k d_k, \taumin)] - \phi(x_k, \taumin) \\
        &\quad\leq -\E_k[\alpha_k \Delta l(x_k, \taumin, \nabla f(x_k), \dktrue)] + \E_k\left[\frac{\alpha_k^2 \beta_k^2}{2}(\taumin L + \Gamma) \|u_k\|^2\right] \\
        &\quad\quad+ \E_k\left[\frac{\alpha_k^2}{2} (\kappa_v(\taumin L + \Gamma) + 4)  \|c_k\|_1\right] + \E_k[\alpha_k \taumin \nabla f(x_k)^T (d_k - \dktrue)] \\
        &\quad\leq -\alpha_k \Delta l(x_k, \taumin, \nabla f(x_k), \dktrue) + (\nu^2 + \theta^2 \beta_k^2) (\kappa_v(\taumin L + \Gamma) + 4)  \|c_k\|_1 \\
        &\quad\quad+ \frac{\alpha_k^2 \beta_k^2}{2}(\taumin L + \Gamma) (\rev{\|\uktrue\|^2} + \zeta^{-2} M) + \theta \beta_k^2 \taumin \kappa_g \zeta^{-1} \sqrt{M} \\
        &\quad\leq -\alpha_k \Delta l(x_k, \taumin, \nabla f(x_k), \dktrue) + \frac{\alpha_k^2 \beta_k^2}{2}(\taumin L + \Gamma) (\rev{\kappa_u^2} + \zeta^{-2} M) \\
        &\quad \ + \frac{\alpha_k \sigma}{2} \|c_k\|_1 + \beta_k^2 \theta(\theta \kappa_c (\kappa_v(\taumin L + \Gamma) + 4) + \taumin \kappa_g \zeta^{-1} \sqrt{M}) \\
        &\quad= -\alpha_k \Delta l(x_k, \taumin, \nabla f(x_k), \dktrue) + \frac{\alpha_k \sigma}{2} \|c_k\|_1 + \beta_k^2 \kappa_1,
    \end{align*}
    where the final inequality follows by $\nu \leq \alpha_k$.
    
    Now, by \eqref{eq:Dllb},
    \begin{align*}
        &\E_k[\phi(x_k + \alpha_k d_k, \taumin)] - \phi(x_k, \taumin) \\
        &\quad \leq -\alpha_k \Delta l(x_k, \taumin, \nabla f(x_k), \dktrue) + \frac{\alpha_k \sigma}{2} \|c_k\|_1 + \beta_k^2 \kappa_1 \\
        &\quad \leq -\alpha_k (\beta_k \taumin  (\uktrue)^T H_k \uktrue + \sigma \|c_k\|_1) + \frac{\alpha_k \sigma}{2} \|c_k\|_1 + \beta_k^2 \kappa_1 \\
        &\quad = -\alpha_k \beta_k \taumin (\uktrue)^T H_k \uktrue - \frac{\alpha_k \sigma}{2} \|c_k\|_1 + \beta_k^2 \kappa_1.
    \end{align*}
    Taking the total expectation of this inequality, rearranging and summing from $k = 0, \dots, K-1$,
    \begin{equation*}
        \begin{split}
        \sum_{k=0}^{K-1} \E[\alpha_k \beta_k \taumin (\uktrue)^T H_k \uktrue + \frac{\alpha_k \sigma}{2} \|c_k\|_1] \\
        \leq \phi(x_0, \taumin) - \E[\phi(x_K, \taumin)] + \kappa_1 \sum_{k=0}^{K-1} \beta_k^2.
        \end{split}
    \end{equation*}
    Due to Assumption \ref{assum:fcsmooth}, we have,
    \begin{equation*}
        -\E[\phi(x_K,\taumin)] = -\E[\taumin f(x_K) + \|c_K\|_1] \leq -\taumin \flow,
    \end{equation*}
    so that
    \begin{equation*}
        \begin{split}
        \sum_{k=0}^{K-1} \E[\alpha_k \beta_k \taumin  (\uktrue)^T H_k \uktrue + \frac{\alpha_k \sigma}{2} \|c_k\|_1] \\
        \leq \phi(x_0, \taumin) - \taumin \flow + \kappa_1 \sum_{k=0}^{K-1} \beta_k^2,
        \end{split}
    \end{equation*}
    which proves the result. \qed
\end{proof}

\rev{Next, we establish some convergence results under different choices of $\beta_k$.}

\rev{
\begin{corollary} \label{cor:convergence}
    Let the assumptions of Theorem \ref{thm:fixedstep} hold.
    Then, if $\beta_k = \beta > 0$ for all $K \in \mathbb{N}$,
    \begin{equation} \label{eq:cconvergencefixed}
        \frac{1}{K} \sum_{k=0}^{K-1} \E[\|c_k\|_1] \leq \frac{2( \taumin (f(x_0) - \flow) + \|c_0\|_1)}{\nu \sigma K} + \frac{\beta^2 \kappa_1}{\nu \sigma} \xrightarrow{K \rightarrow \infty} \frac{\beta^2 \kappa_1}{\nu \sigma},
    \end{equation}
    where $\kappa_1$ is defined in \eqref{eq:kappa1}, and
    \begin{align} 
        &\frac{1}{K} \sum_{k=0}^{K-1} \E[\|\nabla f(x_k) + J_k^T \yktrue\|^2] \leq \frac{\kappa_H^2(\taumin (f(x_0) - \flow) + \|c_0\|_1)}{\beta \nu \eta \taumin \zeta K} + \frac{\kappa_H^2 \kappa_1 \beta}{\nu \eta \taumin \zeta} \nonumber \\
        &\quad+ \frac{2 (1 +2 \kappa_u) \kappa_v \kappa_H^2 (\taumin(f(x_0) - \flow) + \|c_0\|_1)}{\nu \sigma K} +  \frac{2 (1 +2 \kappa_u) \kappa_v \kappa_H^2 \beta^2 \kappa_1}{\nu \sigma} \nonumber  \\ & \quad \quad\xrightarrow{K \rightarrow \infty} \frac{\kappa_H^2 \kappa_1 \beta}{\nu \eta \taumin \zeta} + \frac{2 (1 +2 \kappa_u) \kappa_v \kappa_H^2 \beta^2 \kappa_1}{\nu \sigma}. \label{eq:gconvergencefixed}
    \end{align}
    If $\sum_{k=0}^{\infty} \beta_k^2 < \infty$, then,
    \begin{equation} \label{eq:limcconvergence}
        \underset{k \rightarrow \infty}{\lim} \E[\|c_k\|_1] = 0.
    \end{equation}
    If, in addition, $\sum_{k=0}^{\infty} \beta_k = \infty$, then,
    \begin{equation}
        \underset{k \rightarrow \infty}{\liminf} \ \E[\|\nabla f(x_k) + J_k^T \yktrue\|^2] = 0.
    \end{equation}
\end{corollary}
}
\begin{proof}

    By Theorem \ref{thm:fixedstep}, the definition of $\beta_k$, and $\nu \leq \alpha_k$, it follows that
    \begin{equation} 
        \sum_{k=0}^{K-1} \E[\|c_k\|_1] \leq \frac{2(\taumin (f(x_0) - \flow) + \|c_0\|_1)}{\nu \sigma} + \frac{2 K \kappa_1 \beta^2}{\nu \sigma} \label{eq:cboundcorcon}
    \end{equation}
    Dividing both sides of this inequality by $K$ yields the first result.

    Now, by Theorem \ref{thm:fixedstep} and Lemma \ref{lem:ulb} as well as $\alpha_k \geq \nu$, we have
    \begin{align}
        &\sum_{k=0}^{K-1} \E[\nu \beta_k \taumin \zeta \kappa_H^{-2} \|\nabla f(x_k) + J_k^T \yktrue\|^2 + \frac{\nu \sigma}{2} \|c_k\|_1  \nonumber \\
        &\quad \ - \nu \taumin \beta_k (1 +2 \kappa_u) \zeta \kappa_v \|c_k\|_1]  \nonumber \\
        &\quad\leq \sum_{k=0}^{K-1} \E[\alpha_k \beta_k (\uktrue)^T H_k (\uktrue) + \frac{\alpha_k \sigma}{2} \|c_k\|_1] \nonumber \\
        &\quad\leq \taumin (f(x_0) - \flow) + \|c_0\|_1 + \kappa_1 \sum_{k=0}^{K-1} \beta_k^2. \label{eq:corconvineq1}
    \end{align}
    Rearranging this inequality and using $\beta_k = \beta$,
    \begin{align*}
        \sum_{k=0}^{K-1} \E[\|\nabla f(x_k) + J_k^T \yktrue\|^2] &\leq \frac{\kappa_H^2(\taumin (f(x_0) - \flow) + \|c_0\|_1)}{\beta \nu \eta \taumin \zeta} + \frac{\kappa_H^2 \kappa_1 K \beta}{\nu \eta \taumin \zeta} \\
        &\quad+ (1 +2 \kappa_u) \kappa_v \kappa_H^2\sum_{k=0}^{K-1} \E[\|c_k\|_1].
    \end{align*}
    Dividing through by $K$ and applying \eqref{eq:cconvergencefixed}, it follows that
    \begin{align*}
        \frac{1}{K} \sum_{k=0}^{K-1} \E[\|\nabla f(x_k) + J_k^T \yktrue\|^2] &\leq \frac{\kappa_H^2(\taumin (f(x_0) - \flow) + \|c_0\|_1)}{\beta \nu \eta \taumin \zeta K} \\
        &\quad+ \frac{2 (1 +2 \kappa_u) \kappa_v \kappa_H^2 (\taumin(f(x_0) - \flow) + \|c_0\|_1)}{\nu \sigma K} \\
        &\quad+ \frac{\kappa_H^2 \kappa_1 \beta}{\nu \eta \taumin \zeta} + \frac{2 (1 +2 \kappa_u) \kappa_v \kappa_H^2 \beta^2 \kappa_1}{\nu \sigma}.
    \end{align*}
    which proves the second result.
    
    By Theorem \ref{thm:fixedstep}, and $\nu \leq \alpha_k$, it follows that 
    \begin{equation} 
        \sum_{k=0}^{\infty} \E[\|c_k\|_1] \leq \frac{2(\taumin (f(x_0) - \flow) + \|c_0\|_1)}{\nu \sigma} + \frac{2  \kappa_1 \sum_{k=0}^\infty\beta_k^2}{\nu \sigma}. \label{eq:cboundcorcon2}
    \end{equation}
    By assumption, $\sum_{k=0}^\infty \beta_k^2 < \infty$, so the series converges, which directly implies the desired result.

    For the final result, by \eqref{eq:corconvineq1} and $\beta_k \leq \kappa_\beta$ for all $k$, we have
    \begin{align*}
        &\sum_{k=0}^{\infty} \E[\nu \beta_k \taumin \zeta \kappa_H^{-2} \|\nabla f(x_k) + J_k^T \yktrue\|^2]  \nonumber \\
        &\quad\leq \taumin (f(x_0) - \flow) + \|c_0\|_1 + \kappa_1 \sum_{k=0}^{\infty} \beta_k^2 + \nu \taumin (1 +2 \kappa_u) \zeta \kappa_v \kappa_\beta \sum_{k=0}^{\infty} \E[\|c_k\|_1]. 
    \end{align*}
    By \eqref{eq:cboundcorcon2} and $\sum_{k=0}^\infty \beta_k^2 < \infty$, it follows that
    \begin{equation*}
        \sum_{k=0}^{\infty} \E[\beta_k \|\nabla f(x_k) + J_k^T \yktrue\|^2] = 0.
    \end{equation*}
    Recalling that $\sum_{k=0}^\infty \beta_k = \infty$, the desired result follows directly. \qed

\end{proof}

\rev{When comparing these results to those of \cite[Corollary 3.14]{ASBerahas_FECurtis_DPRobinson_BZhou_2021}, we obtain the same radius of convergence for fixed $\beta$ with respect to the gradient of the Lagrangian, which is proportional to $\beta$. In addition, we establish the same limit inferior result for this stationarity measure with a decaying $\beta_k$ sequence. However, our results differ significantly with respect to the constraint violation. In particular, we establish a tighter radius of convergence, proportional to $\beta^2$, for fixed $\beta$ as well as convergence in the limit when $\beta_k$ is a decaying sequence. These properties suggest superior convergence in the constraint violation, due to the two stepsize scheme we employ.}

\rev{Now, we present our main complexity result of this section.}
\begin{corollary} \label{cor:complexityfixed}
    For any $K \in \mathbb{N}_{>0}$, let $\beta_k := \beta = \eta/\sqrt{K}$ for all $k \in [0,K-1]$ where $\eta \in \mathbb{R}_{>0}$, let $\kappa_1$ be defined in \eqref{eq:kappa1} and let
    \begin{equation} \label{eq:kappa2}
        \kappa_2 := \taumin (f(x_0) - \flow) + \|c_0\|_1 + \eta^2 \kappa_1.
    \end{equation}
    Then, under the conditions of Theorem \ref{thm:fixedstep}, we have
    \begin{equation} \label{eq:ccomplexityfixed}
        \frac{1}{K} \sum_{k=0}^{K-1} \E[\|c_k\|_1] \leq \frac{2 \kappa_2}{\nu \sigma K},
    \end{equation}
        and
    \begin{equation} \label{eq:gcomplexityfixed}
        \frac{1}{K}\sum_{k=0}^{K-1} \E[\|\nabla f(x_k) + J_k^T \yktrue\|^2] \leq \frac{\kappa_H^2\kappa_2}{\taumin \zeta \nu \eta\sqrt{K}} + \frac{2\zeta(1 +2 \kappa_u) \kappa_v \kappa_H^2\kappa_2}{\nu \sigma K}. 
    \end{equation}
    Finally, with probability at least $1-\delta$,
    \begin{equation} \label{eq:complexityfixedmarkov}
    \begin{split}
         \underset{k \in [0,K-1]}{\min} \ \taumin \zeta \kappa_H^{-2} \|\nabla f(x_k) + J_k^T \yktrue\|^2 + \frac{\sigma \sqrt{K}}{2 \eta} \|c_k\|_1 \\ \leq  \frac{\kappa_2}{\nu \eta \delta \sqrt{K}} + \frac{2(1 +2 \kappa_u) \zeta \taumin \kappa_v\kappa_2}{\sigma \delta K}.
    \end{split}
    \end{equation}

\end{corollary}
\begin{proof}
    \rev{
    The first two results follow directly from Corollary \ref{cor:convergence} with this specific choice of $\beta$.}
    
    To prove the final result, by \eqref{eq:corconvineq1},
    \begin{align*}
        &\sum_{k=0}^{K-1} \E[\nu \beta_k \taumin \zeta \kappa_H^{-2} \|\nabla f(x_k) + J_k^T \yktrue\|^2 + \frac{\nu \sigma}{2} \|c_k\|_1]  \nonumber \\
        &\quad\leq \taumin (f(x_0) - \flow) + \|c_0\|_1 + \kappa_1 \sum_{k=0}^{K-1} \beta_k^2 \\
        &\quad+ \sum_{k=0}^{K-1} \E[\nu \beta_k \taumin (1 +2 \kappa_u) \zeta \kappa_v \|c_k\|_1].
    \end{align*}
    Applying the definition of $\beta$, multiplying through by $\frac{1}{\nu \beta K}$, and using \eqref{eq:ccomplexityfixed},
    \begin{equation*}
    \begin{split}
        \frac{1}{K} \rev{\sum_{k=0}^{K-1}}\E[\taumin \zeta \kappa_H^{-2} \|\nabla f(x_k) + J_k^T \yktrue\|^2 + \frac{\sigma \sqrt{K}}{2\eta} \|c_k\|_1] \\ \leq \frac{\kappa_2}{\nu \eta \sqrt{K}} + \frac{2(1 +2 \kappa_u) \zeta \taumin \kappa_v\kappa_2}{\sigma K}
    \end{split}
    \end{equation*}
    and thus
    \begin{equation*}
        \begin{split}
        \underset{k \in [0,K-1]}{\min} \ \E[\taumin \zeta \kappa_H^{-2} \|\nabla f(x_k) + J_k^T \yktrue\|^2 + \frac{\sigma \sqrt{K}}{2\eta} \|c_k\|_1] \\ \leq \frac{\kappa_2}{\nu \sqrt{K}} + \frac{2(1 +2 \kappa_u) \zeta \taumin \kappa_v\kappa_2}{\nu \sigma K}.
        \end{split}
    \end{equation*}
    Applying Markov's inequality \rev{and Jensen's inequality}, it follows that with probability at least $1-\delta$ that
    \begin{equation*}
        \begin{split}
        \underset{k \in [0,K-1]}{\min} \ \taumin \zeta \kappa_H^{-2} \|\nabla f(x_k) + J_k^T \yktrue\|^2 + \frac{\sigma \sqrt{K}}{2 \eta} \|c_k\|_1 \\
        \leq \frac{\kappa_2}{\nu\delta \sqrt{K}} + \frac{2(1 +2 \kappa_u) \zeta \taumin \kappa_v\kappa_2}{\sigma \delta K},
        \end{split}
    \end{equation*}
    which proves the final result.
    \qed
\end{proof}

From the result of Corollary \ref{cor:complexityfixed}, we can easily derive our worst-case complexity results, as promised in Section \ref{sec:intro}. It should be clear that in terms of the constraint violation, by \eqref{eq:ccomplexityfixed}, the maximum number of iterations until $\E[\|c_k\|_1]$ falls below $\epsilon_c$ is at most $\mathcal{O}(\epsilon_c^{-1})$. Similarly, by Jensen's inequality and \eqref{eq:gcomplexityfixed}, the maximum number of iterations until $\E[\|\nabla f(x_k) + J_k^T\yktrue\|] \leq \epsilon_\ell$ is $\mathcal{O}(\epsilon_\ell^{-4})$. Finally, if one is interested in a combined result, we obtain the same $\mathcal{O}(K^{-1/2})$ convergence rate as \cite{FECurtis_MJONeill_DPRobinson_2024}, however, our convergence is in terms of a much stronger measure with respect to the constraint violation $\|c_k\|_1$, which is scaled by an additional factor of $\sqrt{K}$. Thus, we expect much faster convergence with respect to the constraint violation than the algorithm in \cite{FECurtis_MJONeill_DPRobinson_2024} without harming the convergence rate in terms of the gradient of the Lagrangian.

\subsection{Convergence with Adaptive Stepsizes} \label{subsec:adaptive}

Now, we analyze the case where $\beta_k$ and $\alpha_k$ are set adaptively, in a manner inspired by Adagrad-Norm \cite{RWard_XWu_LBottou_2020}. Specifically, at each iteration $k$, let
\begin{equation} \label{eq:adabq}
    b_k^2 = b_{k-1}^2 + \|u_k\|^2, \quad q_k^2 = q_{k-1}^2 + \|c_k\|_1,
\end{equation}
and
\begin{equation} \label{eq:adastep}
    \beta_k = \frac{\eta}{b_k}, \quad \quad \alpha_k \in \left[\frac{\nu}{q_k}, \frac{\nu}{q_k} + \theta \min\left\{\frac{1}{b_k}, \frac{1}{q_k}\right\}\right],
\end{equation}
for some constants $\eta > 0$ and $\nu > 0$. We note here that the additional term at the upper end of the range for $\alpha_k$ is due to our adaptive setting of $\beta_k$ using $b_k$, which is sufficient to control the stochasticity in $g_k$, but may be insufficient to control second order terms involving the constraint violation. We remedy this situation via the inclusion of the $\theta/q_k$ term. In addition, we remark that $q_k$ can be set in many different ways, such as using $\|v_k\|^2$ or $\|c_k\|_2$ in place of $\|c_k\|_1$. \rev{In principle, one simply needs to choose quantities which are upper bounds on $\|v_k\|^2$ that are computable during the course of the algorithm, which is a relatively flexible condition given Lemma \ref{lem:vbound} and Assumption \ref{assum:fcsmooth}.} These other strategies may lead to longer stepsizes, which could have important practical implications, however, we choose to use $\|c_k\|_1$ as it obtains the best constant factors in the convergence analysis among the relevant choices. \rev{Indeed, one can actually take the minimum over multiple quantities and the subsequent analysis follows with only minor changes; in all of our numerical experiments, we use $q_k^2 = q_{k-1}^k + \min\{\|c_k\|_1, \|v_k\|, \|v_k\|^2\}$ which significantly improves the numerical results.} \rev{We wish to highlight in the case where $c_k = 0$ for all $k$ and $\theta = 0$, then $\alpha_k$ remains fixed at $\nu/q_{-1}$ for all $k$. In addition, if $H_k = I$, then $u_k = P_k g_k$, where $P_k$ is the orthogonal projection matrix onto $\text{Null}(J_k)$, so that the step $d_k$ is equivalent to the one generated by applying Adagrad-Norm \cite{RWard_XWu_LBottou_2020} directly to the projected gradient mapping.}

Throughout this section, since $\beta_k$ is dependent on $g_k$, we redefine $\dktrue$ as
\begin{equation} \label{eq:dktrueada}
    \dktrue := v_k + \beta_{k-1} \uktrue,
\end{equation}
so that it remains measurable to $\mathcal{F}_k$. We note that under this re-definition, the results of Lemmas \ref{lem:taumin} and \ref{lem:descentlemma} still hold.

Our subsequent analysis relies on the following lemma, which we give without proof as it is a well-known result in the adaptive gradient literature (see for example, \cite[Lemma 10]{BWang_HZhang_ZMa_WChen_2023}).
\begin{lemma} \label{lem:logbound}
    Let $\{a_i\}_{i=0}^\infty$ be a series of non-negative real numbers with $a_0 \in \R_{>0}$. Then,
    \begin{equation} \label{eq:logboundlem}
        \sum_{k=1}^T \frac{a_k}{\sum_{i=0}^k a_i} \leq \log \left(\sum_{k=0}^T a_k\right) - \log(a_0)
    \end{equation}
\end{lemma}

In order to prove convergence of our algorithm, the key issue posed by the adaptive stepsizes is the final term in \eqref{eq:descent}, which requires a more detailed analysis than in Lemma \ref{lem:innerprodfixed} as $\beta_k$ is no longer measurable to $\mathcal{F}_k$ and $\dktrue$ has been redefined in \eqref{eq:dktrueada}. We give a bound on this term in the following lemma.
\begin{lemma} \label{lem:innerprodada}
    Let Assumptions \ref{assum:fcsmooth}, \ref{assum:Hbound}, and \ref{assum:stochG} hold and let
    \begin{equation} \label{eq:kappa3}
        \kappa_3 := \rev{\kappa_u^2 + \zeta^{-2} M}
    \end{equation}
    and
    \begin{equation} \label{eq:kappa4}
        \kappa_4 := \max\left\{\zeta^{-1} \kappa_H^2, \beta_{-1}(1+2 \kappa_u)\kappa_H^2 \kappa_v \taumin/\sigma\right\}
    \end{equation}
    Then,
    \begin{align}
        &\E_k\left[\alpha_k \nabla f(x_k)^T (d_k - \dktrue)\right] \nonumber \leq \E_k \Bigg[\frac{\alpha_k \beta_{k-1}}{2} (\uktrue)^T H_k \uktrue + \frac{\alpha_k \sigma}{2 \taumin} \|c_k\|_1 \nonumber \\
        &\quad+ \left(\frac{3 \eta^2 \kappa_3 \kappa_4 (\nu + \theta)^2}{2 q_{-1} b_{-1}} + \frac{3 \kappa_4 \theta^2 (\eta^2 + \beta_{-1}^2 \kappa_3)}{2 \eta \nu} + \frac{3 \rev{\zeta^{-2}} M \kappa_4 \theta^2 \beta_{-1}}{2 q_{-1}}\right)\frac{\|u_k\|^2}{b_k^2}\Bigg].\label{eq:innerprodadaresult}
    \end{align}
\end{lemma}

\begin{proof}
    By the definition of $\dktrue$, we have
    \begin{equation*}
        \E_k\left[\alpha_k \nabla f(x_k)^T (d_k - \dktrue)\right] = \E_k\left[\alpha_k \nabla f(x_k)^T (\beta_k u_k - \beta_{k-1} \uktrue)\right].
    \end{equation*}
    Let $\xi_k \in [0,1]$ be the random variable such that $\alpha_k = \frac{\nu}{q_k} +\xi_k \min\{\frac{\theta}{b_k}, \frac{\theta}{q_k}\}$. Then, by Lemma \ref{lem:stochasticdkbound} and the fact that $\beta_{k-1}$, $\nu$, and $q_k$ are measurable to $\mathcal{F}_k$,
    \begin{align}
        &\E_k\left[\alpha_k \nabla f(x_k)^T (\beta_k u_k - \beta_{k-1} \uktrue)\right] \nonumber \\
        &= \E_k\left[\left(\frac{\nu}{q_k} +\xi_k \min\left\{\frac{\theta}{b_k}, \frac{\theta}{q_k}\right\}\right) \nabla f(x_k)^T (\beta_k u_k - \beta_{k-1} \uktrue)\right] \nonumber \\
        &= \E_k\Big[\frac{\nu}{q_k} (\beta_k - \beta_{k-1}) \nabla f(x_k)^T u_k \nonumber \\
        &\quad+ \xi_k \min\left\{\frac{\theta}{b_k}, \frac{\theta}{q_k}\right\} \nabla f(x_k)^T (\beta_k u_k - \beta_{k-1} \uktrue) \Big] \nonumber \\
        &= \E_k\Big[\frac{\nu}{q_k} (\beta_k - \beta_{k-1}) (\nabla f(x_k) + J_k^T \yktrue)^T u_k \nonumber \\
        &\quad+ \xi_k \min\left\{\frac{\theta}{b_k}, \frac{\theta}{q_k}\right\} (\nabla f(x_k) + J_k^T \yktrue)^T (\beta_k u_k - \beta_{k-1} \uktrue) \Big] \nonumber \\
        &= \E_k\Big[\frac{\nu}{q_k}  (\beta_k - \beta_{k-1}) (\nabla f(x_k) + J_k^T \yktrue)^T u_k \nonumber \\
        &\quad+  \xi_k \min\left\{\frac{\theta}{b_k}, \frac{\theta}{q_k}\right\} \beta_k (\nabla f(x_k) + J_k^T \yktrue)^T (u_k - \uktrue) \nonumber \\
        &\quad+ \xi_k \min\left\{\frac{\theta}{b_k}, \frac{\theta}{q_k}\right\} (\beta_k-\beta_{k-1}) (\nabla f(x_k) + J_k^T \yktrue)^T \uktrue \Big] \nonumber \\
        &= \E_k\Big[\left(\frac{\nu}{q_k} +\xi_k \min\left\{\frac{\theta}{b_k}, \frac{\theta}{q_k}\right\}\right)  (\beta_k - \beta_{k-1}) (\nabla f(x_k) + J_k^T \yktrue)^T u_k \nonumber \\
        &\quad+  \xi_k \min\left\{\frac{\theta}{b_k}, \frac{\theta}{q_k}\right\} \beta_k (\nabla f(x_k) + J_k^T \yktrue)^T (u_k - \uktrue) \nonumber \\
        &\quad+ \xi_k \min\left\{\frac{\theta}{b_k}, \frac{\theta}{q_k}\right\} (\beta_k-\beta_{k-1}) (\nabla f(x_k) + J_k^T \yktrue)^T (\uktrue - u_k) \Big] \nonumber \\
        &\leq \E_k\Big[\frac{\nu + \theta}{q_k}  |\beta_k - \beta_{k-1}| \|\nabla f(x_k) + J_k^T \yktrue\| \|u_k\| \nonumber \\
        &\quad+ \min\left\{\frac{\theta}{b_k}, \frac{\theta}{q_k}\right\} \beta_k \|\nabla f(x_k) + J_k^T \yktrue\| \|u_k - \uktrue\| \label{eq:innerprodada1} \\
        &\quad+ \frac{\theta}{q_k} |\beta_k-\beta_{k-1}| \|\nabla f(x_k) + J_k^T \yktrue\| \|\uktrue - u_k\| \Big], \nonumber
    \end{align}
    where the third equality follows by $u_k, \uktrue \in \text{Null}(J_k)$ and the inequality by the Cauchy-Schwarz inequality and $\xi_k \leq 1$.

    Now, we focus on the first term in \eqref{eq:innerprodada1},
    \begin{equation} \label{eq:innerprodada}
        |\beta_k-\beta_{k-1}| = \frac{\eta}{b_{k-1}}-\frac{\eta}{b_k} = \frac{\eta\|u_k\|^2}{b_{k-1}b_k(b_k+b_{k-1})} \leq \frac{\eta\|u_k\|}{b_{k-1}b_k},
    \end{equation}
    where the inequality follows by $\|u_k\| \leq b_k$. Therefore, applying Young's inequality, we have, for any $\lambda_1 > 0$ measurable to $\mathcal{F}_k$,
    \begin{align}
        &\E_k\left[\frac{\nu + \theta}{q_k} |\beta_{k-1} -\beta_k|\|\nabla f(x_k)+J_k^T \yktrue\| \|u_k\|\right] \nonumber \\
        &\leq \eta \E_k\left[\frac{(\nu + \theta) \|\nabla f(x_k)+J_k^T \yktrue\| \|u_k\|^2}{q_k b_{k-1} b_k}\right] \nonumber \\
        &\leq \frac{\eta \|\nabla f(x_k)+J_k^T \yktrue\|^2}{2 b_{k-1} q_k \lambda_1} \E_k[\|u_k\|^2] + \E_k \left[\frac{\eta \left(\nu + \theta\right)^2 \lambda_1 \|u_k\|^2}{2 q_k b_{k-1} b_k^2}\right] \nonumber \\
        &\leq \frac{\eta \kappa_3 \beta_{k-1} \|\nabla f(x_k)+J_k^T \yktrue\|^2 }{2q_k\lambda_1} + \E_k \left[\frac{\eta (\nu + \theta)^2 \lambda_1 \|u_k\|^2}{2 q_{k} b_{k-1}b_k^2}\right] \label{eq:innerprodada2}
    \end{align}
    where the final inequality follows by Assumption \ref{assum:Hbound} as well as the results of  Lemma \ref{lem:uktruebound} and Lemma \ref{lem:stochasticbounds}.

    Now, for the second term in \eqref{eq:innerprodada1}, by Young's inequality, for any $\lambda_2 > 0$ measurable to $\mathcal{F}_k$,
    \begin{align}
        &\E_k\left[\beta_k \min\left\{\frac{\theta}{b_k}, \frac{\theta}{q_k}\right\} \|\nabla f(x_k) + J_k^T \yktrue\| \|u_k - \uktrue\| \right] \nonumber \\
        &\leq \frac{1}{2 q_k b_k \lambda_2} \|\nabla f(x_k) + J_k^T \yktrue\|^2 + \E_k\left[\frac{\lambda_2 \theta^2 \beta_k^2}{2} \|u_k - \uktrue\|^2 \right]. \label{eq:innerprodada3}
    \end{align}
    Working with the last term in this inequality, since $q_k$ and $\beta_{k-1}$ are measurable to $\mathcal{F}_k$, by Lemma \ref{lem:stochasticbounds},
    \begin{align}
        &\E_k\left[\frac{\lambda_2 \theta^2 \beta_k^2}{2} \|u_k - \uktrue\|^2 \right] \nonumber \\
        &= \frac{\lambda_2 \theta^2}{2} \E_k\left[\beta_k^2(\|u_k\|^2 + \|\uktrue\|^2 - 2 u_k^T \uktrue) \right] \nonumber \\
        &= \frac{\lambda_2 \theta^2}{2} \E_k\left[\beta_k^2(\|u_k\|^2 + \|\uktrue\|^2 - 2 u_k^T \uktrue) + \beta_{k-1}^2(2 u_k^T \uktrue -2\|\uktrue\|^2)  \right] \nonumber \\
        &\leq \frac{\lambda_2 \theta^2}{2} \E_k\left[\beta_k^2 \|u_k\|^2 + 2|\beta_k^2 - \beta_{k-1}^2| \|u_k\| \|\uktrue\| - \beta_{k-1}^2 \|\uktrue\|^2 \right] \nonumber \\
        &= \frac{\lambda_2 \theta^2}{2} \E_k\left[\beta_k^2 \|u_k\|^2 + 2\eta^2 \frac{\|u_k\|^2}{b_{k-1}^2 b_k^2} \|u_k\| \|\uktrue\| - \beta_{k-1}^2 \|\uktrue\|^2\right] \nonumber \\
        &\leq \frac{\lambda_2 \theta^2}{2} \E_k\left[\beta_k^2 \|u_k\|^2 + 2\eta^2 \frac{\|u_k\|^2}{b_{k-1}^2 b_k} \|\uktrue\| - \beta_{k-1}^2 \|\uktrue\|^2\right] \label{eq:innerprodada4}
    \end{align}
    where the first inequality follows by $\beta_k \leq \beta_{k-1}$ and the second inequality follows by $\|u_k\| \leq b_k$. Dealing with the second term in this inequality, again applying Young's inequality and using $\beta_{k-1} = \eta/b_{k-1}$, by Lemmas \ref{lem:uktruebound} and \ref{lem:stochasticbounds} as well as Assumption \ref{assum:Hbound},
    \begin{align*}
        &\frac{\lambda_2 \theta^2}{2} \E_k\left[\frac{2\beta_{k-1}^2 \|u_k\|^2}{b_k} \|\uktrue\| \right] \\
        &\leq \frac{\lambda_2 \theta^2 \eta^2}{2} \E_k\left[\frac{\beta_{k-1}^2 \|u_k\|^2 \|\uktrue\|^2}{\kappa_3} + \frac{\beta_{k-1}^2\kappa_3\|u_k\|^2}{b_k^2}\right] \\
        &\leq \frac{\lambda_2 \theta^2}{2} \E_k\left[\frac{\beta_{k-1}^2 \rev{(\kappa_u^2 + \zeta^{-2} M)} \|\uktrue\|^2}{\kappa_3} + \frac{\beta_{k-1}^2 \kappa_3 \|u_k\|^2}{b_k^2}\right] \\
        &=\frac{\lambda_2 \theta^2}{2} \E_k\left[\beta_{k-1}^2\|\uktrue\|^2 + \frac{\beta_{k-1}^2 \kappa_3 \|u_k\|^2}{b_k^2}\right],
    \end{align*}
    so that the first term cancels with the last in \eqref{eq:innerprodada4}.

    Now, for the final term in \eqref{eq:innerprodada1}, by \eqref{eq:innerprodada}, applying Young's inequality for some $\lambda_3 > 0$ that is measurable to $\mathcal{F}_k$, by Lemma \ref{lem:stochasticbounds},
    \begin{align*}
        &\E_k\left[\frac{\theta}{q_k} |\beta_k-\beta_{k-1}| \|\nabla f(x_k) + J_k^T \yktrue\| \|\uktrue - u_k\| \right] \\
        &\leq \E_k\left[\frac{\theta}{q_k} \frac{\beta_{k-1} \|u_k\|}{b_k} \|\nabla f(x_k) + J_k^T \yktrue\| \|\uktrue - u_k\| \right] \\
        &\leq \E_k\left[\frac{\beta_{k-1}}{2 \lambda_3 q_k} \|\nabla f(x_k) + J_k^T \yktrue\|^2 \|\uktrue - u_k\|^2  + \frac{\theta^2 \lambda_3 \beta_{k-1} \|u_k\|^2 }{2 q_k b_k^2}  \right] \\
        &\leq \frac{\rev{\zeta^{-2}} M\beta_{k-1}}{2 \lambda_3 q_k} \|\nabla f(x_k) + J_k^T \yktrue\|^2  + \E_k\left[\frac{\theta^2 \lambda_3 \beta_{k-1} \|u_k\|^2 }{2 q_k b_k^2}  \right]
    \end{align*}
    
    Therefore, combining \eqref{eq:innerprodada1}, \eqref{eq:innerprodada2}, \eqref{eq:innerprodada3}, and \eqref{eq:innerprodada4} we have
    \begin{align*}
        &\E_k\left[\alpha_k \nabla f(x_k)^T (\beta_k u_k - \beta_{k-1} \uktrue)\right] \nonumber \\
        &\leq \E_k\Bigg[\left(\frac{\eta \kappa_3 \beta_{k-1}}{2q_k\lambda_1} + \frac{1}{2 q_k b_k \lambda_2} + \frac{\rev{\zeta^{-2}} M \beta_{k-1}}{2 q_k\lambda_3}\right) \|\nabla f(x_k)+J_k^T \yktrue\|^2 \\
        &\quad+ \left(\frac{\lambda_1\eta (\nu + \theta)^2}{2 q_{k} b_{k-1}} + \frac{\lambda_2 \theta^2 (\eta^2 + \beta_{k-1}^2 \kappa_3)}{2} + \frac{\lambda_3 \theta^2 \beta_{k-1}}{2 q_k}\right)\frac{\|u_k^2\|}{b_k^2} \Bigg]
    \end{align*}
   Applying Lemma \ref{lem:ulb},
    \begin{align*}
        &\E_k\left[\alpha_k \nabla f(x_k)^T (\beta_k u_k - \beta_{k-1} \uktrue)\right] \nonumber \\
        &\leq \E_k\Bigg[\left(\frac{\eta \kappa_3 \beta_{k-1}}{2q_k\lambda_1} + \frac{1}{2 q_k b_k \lambda_2} + \frac{\rev{\zeta^{-2}} M \beta_{k-1}}{2 q_k\lambda_3}\right) (\zeta^{-1} \kappa_H^2 (\uktrue)^T H_k \uktrue + (1 + 2 \kappa_u)\kappa_H^2 \kappa_v \|c_k\|_1) \\
        &\quad+ \left(\frac{\lambda_1\eta (\nu + \theta)^2}{2 q_{k} b_{k-1}} + \frac{\lambda_2 \theta^2 (\eta^2 + \beta_{k-1}^2 \kappa_3)}{2} + \frac{\lambda_3 \theta^2 \beta_{k-1}}{2 q_k}\right)\frac{\|u_k\|^2}{b_k^2} \Bigg].
    \end{align*}
    Choosing $\lambda_1 = \frac{3 \eta \kappa_3 \kappa_4}{\nu}$, $\lambda_2 = \frac{3 \kappa_4}{\nu \eta}$, and $\lambda_3 = \frac{3 \rev{\zeta^{-2}} M \kappa_4}{\nu}$ and using $\nu/q_k \leq \alpha_k$, $q_k \geq q_{-1}$, and $b_{k-1} \geq b_{-1}$ proves the result. \qed
\end{proof}

Now, we are prepared to present the first main result of this subsection.
\begin{theorem} \label{thm:adaconvergence}
    Let Assumptions \ref{assum:fcsmooth}, \ref{assum:Hbound}, and \ref{assum:stochG} hold. 
    Let
    \begin{equation} \label{eq:kappa5}
        \kappa_5 := \frac{(\nu+\theta)^2 (\kappa_v(\taumin L + \Gamma) + 4)}{2}.
    \end{equation}

    \begin{align} \label{eq:kappa6}
        \kappa_6 &:= \frac{\eta^2 (\nu+\theta)^2 (\taumin L + \Gamma)}{2q_{-1}^2} + \frac{3 \eta^2 \kappa_3 \kappa_4 (\nu + \theta)^2}{2 q_{-1} b_{-1}} \\
        &\quad \ + \frac{3 \kappa_4 \theta^2 (\eta^2 + \beta_{-1}^2 \kappa_3)}{2 \eta \nu} + \frac{3 \rev{\zeta^{-2}} M \kappa_4 \theta^2 \beta_{-1}}{2 q_{-1} b_{-1}^2}.\nonumber
    \end{align}
    Then,
    \begin{align} 
        &\E\left[\sum_{k=0}^{K-1} \frac{\alpha_k \taumin \beta_{k-1}}{2} (\uktrue)^T H_k \uktrue + \frac{\alpha_k \sigma}{2}\|c_k\|_1\right] \nonumber \\
        &\leq \taumin (f_{-1} - \fmin) + \|c_{-1}\|_1 + \kappa_5 \log(1+\kappa_c K/q_{-1}^2) \label{eq:adaconvergence} \\
        &\quad+ \kappa_6 \log (1+\rev{(\kappa_u^2 + \zeta^{-2} M)}K/b_{-1}^2)).\nonumber
    \end{align}
    In addition,
    \begin{align*}
        \E[q_{K-1}] &\leq q_{-1} + \frac{2}{\nu \sigma}(\taumin (f_{-1} - \fmin) + \|c_{-1}\|_1 + \kappa_5 \log(1+\kappa_c K/q_{-1}^2) \\
        &\quad+ \kappa_6 \log (1+\rev{(\kappa_u^2 + \zeta^{-2} M)}K/b_{-1}^2).
    \end{align*}
\end{theorem}

\begin{proof}
    By Lemma \ref{lem:descentlemma}, we have
    \begin{align*}
        &\E_k[\phi(x_k + \alpha_k d_k, \taumin)] - \phi(x_k, \taumin) \\
        &\quad\leq -\E_k[\alpha_k \Delta l(x_k, \taumin, \dktrue)] + \E_k\left[\frac{\alpha_k^2 \beta_k^2}{2}(\taumin L + \Gamma) \|u_k\|^2\right] \\
        &\quad+ \E_k\left[\frac{\alpha_k^2}{2} (\kappa_v(\taumin L + \Gamma) + 4)  \|c_k\|_1\right] + \E_k[\alpha_k \taumin \nabla f(x_k)^T (d_k - \dktrue)].
    \end{align*}

    To prove the result, we need to bound the final three terms. Starting with the first of these, we have that
    \begin{equation*}
        \E_k\left[\frac{\alpha_k^2 \beta_k^2}{2}(\taumin L + \Gamma) \|u_k\|^2\right] \leq \frac{\eta^2 (\nu+\theta)^2 (\taumin L + \Gamma)}{2q_{-1}^2}\E_k\left[ \frac{\|u_k\|^2}{b_k^2} \right],
    \end{equation*}
    where the inequality follows due to the definition of $\alpha_k$ and $q_k \geq q_{-1}$. For the next term,
    \begin{equation*}
        \frac{\alpha_k^2 (\kappa_v(\taumin L + \Gamma) + 4)}{2}   \|c_k\|_1 \leq \frac{(\nu+\theta)^2 (\kappa_v(\taumin L + \Gamma) + 4)}{2q_k^2}   \|c_k\|_1 = \frac{\kappa_5 \|c_k\|_1}{q_k^2}.
    \end{equation*}
    Now, applying the result of Lemma \ref{lem:innerprodada}, we have
    \begin{align*}
        &\E_k[\phi(x_k + \alpha_k d_k, \taumin)] - \phi(x_k, \taumin) \\
        &\quad\leq -\E_k[\alpha_k \Delta l(x_k, \taumin, \dktrue)] + \E_k \left[\frac{\alpha_k \taumin \beta_{k-1}}{2} (\uktrue)^T H_k \uktrue\right]  \\
        &\quad+ \E_k\left[\frac{\alpha_k \sigma}{2} \|c_k\|_1\right] + \frac{\kappa_5 \|c_k\|_1}{q_k^2} + \kappa_6 \E_k\left[ \frac{\|u_k\|^2}{b_k^2} \right].
    \end{align*}

    Then, applying \eqref{eq:Dllb} (where we note that under the re-definition of $\dktrue$ in \eqref{eq:dktrueada}, $\beta_k$ is replaced by $\beta_{k-1}$), it follows that
    \begin{align*}
        &\E_k[\phi(x_k + \alpha_k d_k, \taumin)] - \phi(x_k, \taumin) \\
        &\quad\leq -\E_k \left[\frac{\alpha_k \taumin \beta_{k-1}}{2} (\uktrue)^T H_k \uktrue\right] - \E_k\left[\frac{\alpha_k \sigma}{2} \|c_k\|_1\right]  \\
        &\quad+ \frac{\kappa_5 \|c_k\|_1}{q_k^2} + \kappa_6 \E_k\left[ \frac{\|u_k\|^2}{b_k^2} \right].
    \end{align*}

    Next, taking the total expectation of this inequality and summing for all $k=0,\dots,K-1$,
    \begin{align*}
        &\E\left[\sum_{k=0}^{K-1} \frac{\alpha_k \taumin \beta_{k-1}}{2} (\uktrue)^T H_k \uktrue + \frac{\alpha_k \sigma}{2}\|c_k\|_1\right] \\
        &\leq \phi(x_{-1},\taumin) - \E[\phi(x_{K},\taumin)] + \E\left[\kappa_5 \sum_{k=0}^{K-1} \frac{\|c_k\|_1}{q_k^2}\right] + \E\left[\kappa_6 \sum_{k=0}^{K-1} \frac{\|u_k\|^2}{b_k^2} \right].
    \end{align*}
    By the definition of $\phi$ and Assumption \ref{assum:fcsmooth}, it follows that
    \begin{align*}
        \phi(x_{-1},\taumin) - \E[\phi(x_{K},\taumin)] &= \taumin f_{-1} + \|c_{-1}\|_1 - \E[\taumin f_K - \|c_K\|_1] \\
        &\leq \taumin (f_{-1} - \fmin) + \|c_{-1}\|_1.
    \end{align*}
    Now, applying Lemma \ref{lem:logbound} twice, by Assumption \ref{assum:fcsmooth}, it follows that
    \begin{align*}
        &\E\left[\sum_{k=0}^{K-1} \frac{\alpha_k \taumin \beta_{k-1}}{2} (\uktrue)^T H_k \uktrue + \frac{\alpha_k \sigma}{2}\|c_k\|_1\right] \\
        &\leq \taumin (f_{-1} - \fmin) + \|c_{-1}\|_1 + \kappa_5 \log(1+\kappa_c K/q_{-1}^2) \\
        &\quad+ \kappa_6 \E\left[\log\left(\frac{b_{-1}^2 + \sum_{k=0}^{K-1} \|u_k\|^2}{b_{-1}^2}\right)\right].
    \end{align*}
    Using Jensen's inequality, the tower rule, and the results of Lemma \ref{lem:uktruebound} and Lemma \ref{lem:stochasticbounds},
    \begin{equation} \label{eq:adaptiveubound}
    \E\left[\log\left(\frac{b_{-1}^2 + \sum_{k=0}^{K-1} \|u_k\|^2}{b_{-1}^2}\right)\right] \leq \log \left(1+\rev{(\kappa_u^2 + \zeta^{-2} M)}K/b_{-1}^2 \right)
    \end{equation}
    and thus,
    \begin{align*}
        &\E\left[\sum_{k=0}^{K-1} \frac{\alpha_k \taumin \beta_{k-1}}{2} (\uktrue)^T H_k \uktrue + \frac{\alpha_k \sigma}{2}\|c_k\|_1\right] \\
        &\leq \taumin (f_{-1} - \fmin) + \|c_{-1}\|_1 + \kappa_5 \log(1+\kappa_c K/q_{-1}^2) \\
        &\quad+ \kappa_6 \log (1+\rev{(\kappa_u^2 + \zeta^{-2} M)}K/b_{-1}^2),
    \end{align*}
    proving the first result.

    To prove the second result, note that
    \begin{equation*}
        q_{K-1} = \frac{q_{-1}^2 + \sum_{k=0}^{K-1} \|c_k\|_1}{q_{K-1}} \leq q_{-1} + \sum_{k=0}^{K-1} \frac{\|c_k\|_1}{q_k} \leq q_{-1} + \frac{1}{\nu}\sum_{k=0}^{K-1} \alpha_k \|c_k\|_1,
    \end{equation*}
    and therefore, by \eqref{eq:adaconvergence},
    \begin{align*}
        \E[q_{K-1}] &\leq q_{-1} + \frac{2}{\nu \sigma}(\taumin (f_{-1} - \fmin) + \|c_{-1}\|_1 + \kappa_5 \log(1+\kappa_c K/q_{-1}^2) \\
        &\quad+ \kappa_6 \log (1+\rev{(\kappa_u^2 + \zeta^{-2} M)}K/b_{-1}^2).
    \end{align*}
    \qed
\end{proof}

Next, we derive the following corollary, from which our complexity results for this subsection will follow directly.

\begin{corollary} \label{cor:adacomplexity}
    Let the assumptions of Theorem \ref{thm:adaconvergence} hold. Let 
    \begin{align}
        \kappa_7(K) &:= \taumin (f_{-1} - \fmin) + \|c_{-1}\|_1 + \kappa_5 \log(1+\kappa_c K/q_{-1})\label{eq:kappa7} \\
        &\quad+ \kappa_6 \log (1+\rev{(\kappa_u^2 + \zeta^{-2} M)}K/b_{-1}) \nonumber
    \end{align}
    and 
    \begin{equation}
        \kappa_8(K) := \sqrt{b_{-1}^2 + \rev{(\kappa_u^2 + \zeta^{-2} M)}K}. \label{eq:kappa8}
    \end{equation}
    Then, with probability at least $1-\delta_1$,
    \begin{equation} \label{eq:ccomplexityada}
        \E\left[\frac{1}{K}\sum_{k=0}^{K-1} \|c_k\|_1\right] \leq \frac{2(\nu \sigma q_{-1} + 2\kappa_7(K))\kappa_7(K) }{\nu^2 \sigma^2 \delta_1 K},
    \end{equation}
    with probability at least $1-\delta_2$,
    \begin{align} 
        &\E\left[\frac{1}{K}\sum_{k=0}^{K-1} \|\nabla f(x_k) + J_k^T \yktrue\|^2\right] \nonumber \\
        &\quad\leq \frac{8 \kappa_H^2 \kappa_8(K) (\nu \sigma q_{-1} + 2\kappa_7(K))\kappa_7(K)}{\taumin \nu^2 \eta \zeta \sigma \delta_2^2 K} \label{eq:gcomplexityada} \\
        &\quad+\frac{4(\nu \sigma q_{-1} + 2\kappa_7(K))\kappa_H^2(1 +2 \kappa_u) \kappa_v \kappa_7(K) }{\nu^2 \sigma^2 \delta_2 K}, \nonumber
    \end{align}
    and with probability at least $1-\delta_3$,
    \begin{align}
        &\underset{k\in [0,K-1]}{\min} \ \taumin \zeta \kappa_H^{-2} \|\nabla f(x_k) + J_k^T \yktrue\|^2 + \frac{\sigma \kappa_8(K)}{\eta} \|c_k\|_1 \nonumber \\
        &\quad\leq \frac{54\kappa_8(K)(\nu \sigma q_{-1} + 2\kappa_7(K))\kappa_7(K)}{\nu^2 \eta \sigma \delta_3^3 K} \label{eq:combinedcomplexityada} \\
        &\quad+\frac{18(\nu \sigma q_{-1}
        + 2\kappa_7(K))\kappa_7(K) \zeta \taumin (1 +2 \kappa_u) \kappa_v}{\nu^2 \sigma^2 \delta_3^2 K}. \nonumber
    \end{align}
\end{corollary}

\begin{proof}
    \rev{Note that for all $k \in [0, K-1]$, we have that
    \begin{equation*}
        \alpha_k \geq \frac{\nu}{q_k} \geq \frac{\nu}{q_{K-1}}.
    \end{equation*}}
    \rev{Additionally, by} Theorem \ref{thm:adaconvergence} and Markov's inequality, it follows that with probability at least $1-\delta_1$,
    \begin{equation} \label{eq:alphadalb}
        \rev{\frac{\nu}{q_{K-1}}} \geq \frac{\nu^2 \sigma \delta_1}{\nu \sigma q_{-1} + 2\kappa_7(K)}.
    \end{equation}
    Therefore, by \eqref{eq:adaconvergence} and Assumption \ref{assum:Hbound}, it follows that with probability at least $1-\delta_1$,
    \begin{equation*} 
        \E\left[\sum_{k=0}^{K-1} \frac{\nu^2 \sigma^2 \delta_1}{2(\nu \sigma q_{-1} + 2\kappa_7(K)) }\|c_k\|_1\right] \leq \E\left[\sum_{k=0}^{K-1} \frac{\alpha_k \sigma}{2}\|c_k\|_1\right] \leq \kappa_7(K),
    \end{equation*}
    and thus
    \begin{equation*}
        \frac{1}{K}\E\left[\sum_{k=0}^{K-1} \|c_k\|_1\right] \leq \frac{2(\nu \sigma q_{-1} + 2\kappa_7(K))\kappa_7(K) }{\nu^2 \sigma^2 \delta_1 K},
    \end{equation*}
    which proves the first result.
    
    Next, by the law of iterated expectation, Jensen's inequality, and the results of Lemmas \ref{lem:uktruebound} and \ref{lem:stochasticbounds},
    \begin{equation} \label{eq:bkbound}
        \E[b_{K-1}] = \E\left[\sqrt{b_{-1}^2 + \sum_{k=0}^{K-1} \|u_k\|^2}\right] \leq \sqrt{b_{-1}^2 + \rev{(\kappa_u^2 + \zeta^{-2} M)}K} = \kappa_8(K).
    \end{equation}
    Therefore, using \eqref{eq:alphadalb} with $\delta_1 = \delta_2/2$ and Markov's inequality with \eqref{eq:bkbound}, with probability at least $1-\delta_2$, by Assumption \ref{assum:Hbound}, \eqref{eq:adaconvergence}, and the union bound,
    \begin{equation*}
        \E\left[\sum_{k=0}^{K-1} (\uktrue)^T H_k \uktrue \right] \leq \frac{8 \kappa_8(K)(\nu \sigma q_{-1} + 2\kappa_7(K))\kappa_7(K)}{\taumin \nu^2 \eta \sigma \delta_2^2}.
    \end{equation*}
    Next, applying Lemma \ref{lem:ulb},
    \begin{align*}
        \E\left[\sum_{k=0}^{K-1} \|\nabla f(x_k) + J_k^T \yktrue\|^2\right]
        &\leq \frac{8 \kappa_H^2\kappa_8(K)(\nu \sigma q_{-1} + 2\kappa_7(K))\kappa_7(K)}{\taumin \nu^2 \eta \zeta \sigma \delta_2^2} \\
        &\quad+ \kappa_H^2(1 +2 \kappa_u) \kappa_v\E\left[\sum_{k=0}^{K-1}\|c_k\|_1\right].
    \end{align*}
    Noting that this result holds under the same event as in \eqref{eq:ccomplexityada} (with $\delta_1 = \delta_2/2$), it follows that with probability at least $1-\delta_2$,
    \begin{align*}
        \E\left[\frac{1}{K}\sum_{k=0}^{K-1} \|\nabla f(x_k) + J_k^T \yktrue\|^2\right]
        &\leq \frac{8 \kappa_H^2\kappa_8(K)(\nu \sigma q_{-1} + 2\kappa_7(K))\kappa_7(K)}{\taumin \nu^2 \eta \zeta \sigma \delta_2^2 K} \\
        &+\frac{4(\nu \sigma q_{-1} + 2\kappa_7(K))\kappa_H^2(1 +2 \kappa_u) \kappa_v \kappa_7(K) }{\nu^2 \sigma^2 \delta_2 K}.
    \end{align*}

    Finally, using \eqref{eq:adaconvergence}, \eqref{eq:alphadalb}, \eqref{eq:bkbound}, Markov's inequality and the union bound, with probability at least $1-\frac23\delta_3$,
    \begin{equation*}
        \begin{split}
        \E\left[\sum_{k=0}^{K-1} \taumin (\uktrue)^T H_k \uktrue + \frac{\sigma \kappa_8(K)}{\eta} \|c_k\|_1\right] \\
        \leq \frac{18\kappa_8(K)(\nu \sigma q_{-1} + 2\kappa_7(K))\kappa_7(K)}{\nu^2 \eta \sigma \delta_3^2}.
        \end{split}
    \end{equation*}
    Thus, by Lemma \ref{lem:ulb} 
    \begin{align*}
        &\E\left[\frac{1}{K}\sum_{k=0}^{K-1} \taumin \zeta \kappa_H^{-2} \|\nabla f(x_k) + J_k^T \yktrue\|^2 + \frac{\sigma \kappa_8(K)}{\eta} \|c_k\|_1\right] \\
        &\leq \frac{18\kappa_8(K)(\nu \sigma q_{-1} + 2\kappa_7(K))\kappa_7(K)}{\nu^2 \eta \sigma \delta_3^2 K} +\zeta \taumin (1 +2 \kappa_u) \kappa_v\E\left[\frac{1}{K}\sum_{k=0}^{K-1}\|c_k\|_1\right] \\
        &\leq \frac{18\kappa_8(K)(\nu \sigma q_{-1} + 2\kappa_7(K))\kappa_7(K)}{\nu^2 \eta \sigma \delta_3^2 K} \\
        &\quad+\frac{6(\nu \sigma q_{-1} + 2\kappa_7(K))\kappa_7(K) \zeta \taumin (1 +2 \kappa_u) \kappa_v}{\nu^2 \sigma^2 \delta_3 K}.
    \end{align*}
    Therefore, applying Markov's inequality, \rev{Jensen's inequality,} and the union bound, it follows that with probability at least $1-\delta_3$,
    \begin{align*}
        &\underset{k\in [0,K-1]}{\min} \ \taumin \zeta \kappa_H^{-2} \|\nabla f(x_k) + J_k^T \yktrue\|^2 + \frac{\sigma \kappa_8(K)}{\eta} \|c_k\|_1 \\
        &\leq \frac{54\kappa_8(K)(\nu \sigma q_{-1} + 2\kappa_7(K))\kappa_7(K)}{\nu^2 \eta \sigma \delta_3^3 K} \\
        &\quad+\frac{18(\nu \sigma q_{-1} + 2\kappa_7(K))\kappa_7(K) \zeta \taumin (1 +2 \kappa_u) \kappa_v}{\nu^2 \sigma^2 \delta_3^2 K}.
    \end{align*}
    \qed
\end{proof}

By the definitions of $\kappa_7(K) = \mathcal{O}(\log(K))$ and $\kappa_8(K) = \mathcal{O}(\sqrt{K})$, it follows that the results of Corollary \ref{cor:adacomplexity} match, up to log factors, those we derived in Section \ref{subsec:prechosen} for the pre-specified stepsize setting. Thus, in terms of the complexity measures \eqref{eq:complexity}, this variant of Algorithm \ref{alg:tsssqp} has a worst-case complexity of $\tilde{\mathcal{O}}(\epsilon_\ell^{-4})$ and $\tilde{\mathcal{O}}(\epsilon_{c}^{-1})$.

\section{Safeguarded Line Search} \label{sec:ls}

The convergence analysis in Section \ref{sec:analysis} specifies proper ranges for $\alpha_k$ in Algorithm \ref{alg:tsssqp} in order to ensure convergence, but does not provide any recommendations on how to choose $\alpha_k$ in this range. Commonly, in other stochastic SQP methods, the procedure used to set $\alpha_k$ incorporates the merit parameter $\tau_k$, which is adaptively estimated at each iteration. However, the estimation of $\tau_k$ may be highly inaccurate and noisy due to only having stochastic access to the gradient of $f$. For this reason, we do not attempt to rely on the stochastic gradient information in order to choose $\alpha_k$ and instead solely utilize the constraints.

Consider first the case where $\alpha_k$ satisfies $\alpha_k \in [\nu, \nu + \theta \beta_k]$ as it does in the analysis in Section \ref{subsec:prechosen}. Then, we can find an $\alpha_k$ in this range through a safeguarded backtracking procedure. Starting from $\hat{\alpha}_k = \nu + \theta \beta_k$, we backtrack until
\begin{equation} \label{eq:btsuccess}
    \|c(x_k + \hat{\alpha}_k d_k)\|_1 \leq (1-\xi \hat{\alpha}_k) \|c_k\|_1,
\end{equation}
holds for some $\xi \in (0,1)$ where, when \eqref{eq:btsuccess} fails to hold for $\hat{\alpha}_k$, we set $\hat{\alpha}_k = \rho \hat{\alpha}_k$ for some $\rho \in (0,1)$. However, as we cannot guarantee termination, we safeguard this linesearch by ceasing the search procedure if $\hat{\alpha}_k$ ever falls below $\nu$. When \eqref{eq:btsuccess} holds for some $\hat{\alpha}_k \geq \nu$, we set $\alpha_k = \hat{\alpha}_k$. On the other hand, if \eqref{eq:btsuccess} fails to hold prior to $\hat{\alpha}_k < \nu$, we instead set $\alpha_k = \nu$. Thus, this procedure is guaranteed to output an $\alpha_k$ in the specified range and therefore the convergence results of Section \ref{subsec:prechosen} hold. In addition, on any step where \eqref{eq:btsuccess} is satisfied for $\hat{\alpha}_k \geq \nu$, we have confirmation of sufficient decrease in the constraint violation. Finally, we note that the number of backtracking steps at any iteration $k$ is at most $\log(\nu/(\nu + \theta \beta_k))/\log(\rho)$ due to terminating the backtracking as soon as $\hat{\alpha}_k < \nu$.

Unfortunately, the convergence theory only holds for the previous procedure under certain conditions on $\nu$. To relax these conditions, we once again turn to the one of the adaptive stepsize rules of Section \ref{subsec:adaptive}. In particular, we consider the case where $\beta_k$ is chosen as a pre-specified sequence and the lower bound for $\alpha_k$ is chosen in a manner similar to that of \eqref{eq:adastep}. As we use a slight modification of this stepsize, we give the full procedure (which is the algorithm used in the computational results of Section \ref{sec:numerical}) in Algorithm \ref{alg:tsssqpls}.

The backtracking procedure in Algorithm \ref{alg:tsssqpls} is very similar to the one described above, with a few minor differences. In particular, we set the lower bound adaptively, using the stepsize rule in Section \ref{subsec:adaptive}. In addition, unlike in Section \ref{subsec:adaptive}, we only update the lower bound when the backtracking procedure fails to satisfy the sufficient decrease condition prior to reducing $\hat{\alpha}_k$ below the lower bound. The logic for this is simple; if the lower bound was reached, then it is probably too large and should be reduced. On the other hand, when the sufficient decrease condition is satisfied at iteration $k$, we keep the lower bound as it was at the start of this iteration, since it is already sufficiently small to find a good steplength in terms of reducing the constraint violation. \rev{We numerically explore the impact of only updating the lower bound when the line search fails to satisfy the sufficient decrease condition in Appendix \ref{app:tssqpa}.}

\begin{algorithm}[ht]
  \caption{Two Stepsize Stochastic SQP with Adaptive Backtracking}
  \label{alg:tsssqpls}
  \begin{algorithmic}[1]
    \Require $x_0 \in \R^{n}$, $\{\beta_k\} \subset \R_{>0}$, $\nu \in \R_{>0}$, $q_{-1} \in \R_{>0}$, $\theta \in \R_{>0}$, $\xi \in (0,1)$, $\rho \in (0,1)$;
    \For{$k=0,1,\dots$}
    \State Compute stochastic gradient $g_k$.
	  \State Compute $(p_k,y_k)$ as the solution of \eqref{eq:linsys}.
    \State Set $d_k \leftarrow v_k + \beta_k u_k$, where $v_k \in \text{Range}(J_k^T)$ and $u_k \in \text{Null}(J_k)$ are the orthogonal decomposition of $p_k$.
    \State Set $\hat{q}_k^2 \leftarrow q_{k-1}^2 + \|c_k\|_1$ and $\hat{\alpha}_k \leftarrow \frac{\nu}{\hat{q}_k} + \theta \beta_k$.
    \While{$\|c(x_k + \hat{\alpha}_k d_k)\|_1 > (1-\xi \hat{\alpha}_k) \|c_k\|_1$ and $\hat{\alpha}_k \geq \frac{\nu}{\hat{q}_k}$}
    \State Set $\hat{\alpha}_k \leftarrow \rho \hat{\alpha}_k$.
    \EndWhile
    \If{$\hat{\alpha}_k \rev{>} \frac{\nu}{\hat{q}_k}$}
    \State Set $\alpha_k \leftarrow \hat{\alpha}_k$ and $q_k = q_{k-1}$.
    \Else
    \State Set $\alpha_k \leftarrow \frac{\nu}{\hat{q}_k}$ and $q_k = \hat{q}_k$.
    \EndIf
    \State Set $x_{k+1} \leftarrow x_k + \alpha_k d_k$.
    \EndFor
  \end{algorithmic}
\end{algorithm}

While this is a relatively simple variant of Algorithm \ref{alg:tsssqp}, the analysis in Section \ref{sec:analysis} does not directly translate. We provide the following lemma which provides a starting point for the analysis that can then easily be combined with the techniques in Section \ref{sec:analysis} to obtain a worst-case complexity result.

\begin{lemma} \label{lem:safeguardedls1}
    Let Assumptions \ref{assum:fcsmooth}, \ref{assum:Hbound}, and \ref{assum:stochG} hold and let $x_k$ be generated by Algorithm \ref{alg:tsssqpls}. Let $\beta_k = \eta/\sqrt{K}$ hold for all $k$. Let
    \begin{equation} \label{eq:kappa9}
        \kappa_9 := \xi^{-1}(\|c_0\|_1 - (2 + \Gamma \kappa_v/2) \nu^2 \log(q_{-1}^2) + \nu^2 \eta^2 \Gamma \rev{(\kappa_u^2 + \zeta^{-2} M)}/(2 q_{-1}^2))
    \end{equation}
    and
    \begin{equation} \label{eq:kappa10}
        \kappa_{10} := 2(q_{-1} + \kappa_9/\nu) + 8(4 + \Gamma \kappa_v)\xi^{-1} \nu \log(e + (4 + \Gamma \kappa_v)\xi^{-1} \nu).
    \end{equation}
    
    Then, $\E[q_{K-1}] \leq \kappa_{10}$ and
    \begin{equation}
        \sum_{k=0}^{K-1} \E\left[\alpha_k \|c_k\|_1\right] \leq \kappa_9 + (4 + \Gamma \kappa_v) \xi^{-1} \nu \log(\kappa_{10})
    \end{equation}
\end{lemma}

\begin{proof}
    Let $\mathcal{K}_\alpha$ denote the index set of iterations $k$ such that $\alpha_k = \frac{\nu}{\hat{q}_k}$. Then, for any $k \in \mathcal{K}_\alpha$, by $\Gamma$-Lipschitz continuity of the Jacobian of $c$,
    \begin{align*}
        \|c(x_k + \alpha_k d_k)\|_1 - \|c_k\|_1 &\leq \|c_k + \alpha_k J_k d_k\|_1 - \|c_k\|_1 + \frac{\alpha_k^2 \Gamma}{2} \|d_k\|^2 \\
        &= |1-\alpha_k| \|c_k\|_1 - \|c_k\|_1 + \frac{\alpha_k^2 \Gamma}{2} \|d_k\|^2,
    \end{align*}
    where the equality follows from $J_k d_k = -c_k$. 
    \rev{Using the fact that \\$|1-\alpha_k| \leq 1 - \alpha_k + 2\alpha_k^2$, whenever $k \in \mathcal{K}_\alpha$, we have}
    \begin{equation*}
        \|c(x_k + \alpha_k d_k)\|_1 - \|c_k\|_1 \leq -\xi \alpha_k \|c_k\|_1 + 2\alpha_k^2\|c_k\|_1 + \frac{\alpha_k^2 \Gamma}{2} \|d_k\|^2,
    \end{equation*}
    where we used $\xi \in (0,1)$.

    Next, for any iteration where $k \in \mathcal{K}_\alpha^c$, it follows that
    \begin{equation*}
        \|c(x_k + \alpha_k d_k)\|_1 - \|c_k\|_1 \leq (1-\xi \alpha_k) \|c_k\|_1 - \|c_k\|_1 = -\xi \alpha_k \|c_k\|_1.
    \end{equation*}

    Combining these cases and summing this inequality for $k =0,\dots, K-1$, it follows that
    \begin{equation*}
        \|c(x_K)\|_1 - \|c_0\|_1 \leq -\sum_{k=0}^{K-1} \xi \alpha_k \|c_k\|_1 + \sum_{j \in \mathcal{K}_\alpha}  2 \alpha_j^2 \|c_j\|_1 + \frac{\alpha_j^2 \Gamma}{2} \|d_j\|^2.
    \end{equation*}
    Next, by the orthogonal decomposition $d_k = v_k + \beta_k u_k$ and Lemma \ref{lem:vbound}, we have
    \begin{align}
        &\|c(x_K)\|_1 - \|c_0\|_1 \nonumber \\
        &\leq -\sum_{k=0}^{K-1} \xi \alpha_k \|c_k\|_1 + \sum_{j \in \mathcal{K}_\alpha}  2 \alpha_j^2 \|c_j\|_1 + \frac{\alpha_j^2 \Gamma}{2} (\|v_j\|^2 + \beta_j^2 \|u_j\|^2) \nonumber \\
        &\leq -\sum_{k=0}^{K-1} \xi \alpha_k \|c_k\|_1 + \sum_{j \in \mathcal{K}_\alpha}  (2 + \Gamma \kappa_v/2) \alpha_j^2 \|c_j\|_1 + \frac{\alpha_j^2 \beta_j^2 \Gamma}{2} \|u_j\|^2 \nonumber \\
        &\leq -\sum_{k=0}^{K-1} \xi \alpha_k \|c_k\|_1 + \sum_{j \in \mathcal{K}_\alpha}  \frac{(2 + \Gamma \kappa_v/2) \nu^2}{q_{-1}^2 + \sum_{\ell \in \mathcal{K}_{\alpha}, \ell \leq j} \|c_{\ell}\|_1} \|c_j\|_1 + \frac{\nu^2 \beta_j^2 \Gamma}{2 q_{-1}^2} \|u_j\|^2, \label{eq:linesearchlemma}
    \end{align}
    where the final inequality follows by the definition of $\alpha_k$ for any $k \in \mathcal{K}_\alpha$. Next, taking the expectation of both sides of this inequality, using the definition of $\beta$, rearranging, and using the law of iterated expectation with the result of Lemma \ref{lem:stochasticbounds}, we have
     \begin{align}
        &\sum_{k=0}^{K-1} \E\left[\alpha_k \|c_k\|_1\right] \nonumber \\
        &\leq \xi^{-1} \|c_0\|_1 + \E\left[\sum_{j \in \mathcal{K}_\alpha}  \frac{\xi^{-1} (2 + \Gamma \kappa_v/2) \nu^2}{q_{-1}^2 + \sum_{\ell \in \mathcal{K}_{\alpha}, \ell \leq j} \|c_{\ell}\|_1} \|c_j\|_1 + \frac{\nu^2 \eta^2 \Gamma}{2 K  q_{-1}^2} \|u_j\|^2\right] \nonumber \\
        &\leq \xi^{-1} \|c_0\|_1 + \xi^{-1}(2 + \Gamma \kappa_v/2) \nu^2 \E\left[\log(q_{-1}^2 + \sum_{j \in \mathcal{K}_\alpha} \|c_j\|_1)-\log(q_{-1}^2)\right] \nonumber \\
        &\quad+ \frac{\nu^2 \eta^2 \xi^{-1} \Gamma \rev{(\kappa_u^2 + \zeta^{-2} M)}}{2 q_{-1}^2} \nonumber \\
        &\leq \xi^{-1} \|c_0\|_1 + \xi^{-1}(2 + \Gamma \kappa_v/2) \nu^2 (2\log(\E[q_{K-1}])-\log(q_{-1}^2)) \label{eq:btlemma1} \\
        &\quad+ \frac{\nu^2 \eta^2 \xi^{-1} \Gamma \rev{(\kappa_u^2 + \zeta^{-2} M)}}{2 q_{-1}^2}, \nonumber
    \end{align}
    where the second inequality follows by Lemma \ref{lem:logbound} and the final inequality follows by Jensen's inequality and the concavity of $\log(x)$.

    Therefore, since $\alpha_k \geq \nu/q_{K-1}$ for all $k \leq K-1$, it follows that
    \begin{equation*}
        \E\left[\sum_{k \in \mathcal{K}_{\alpha}} \frac{\|c_k\|_1}{q_{K-1}}\right] \leq \kappa_9/\nu + (4 + \Gamma \kappa_v) \xi^{-1} \nu \log(\E[q_{K-1}]).
    \end{equation*}
    Now, by the definition of $q_{K-1}$ and the fact that $x \leq a + b \log(x)$ implies $x \leq 2a + 8b \log(e + b)$ for any $a > 0$ and $b > 0$ \cite{BWang_HZhang_ZMa_WChen_2023},
    \begin{align*}
        \E[q_{K-1}] &= \E\left[\frac{q_{-1}^2 + \sum_{k \in \mathcal{K}_{\alpha}} \|c_k\|_1}{q_{K-1}}\right] \\
        &\leq q_{-1} + \kappa_9/\nu + (4 + \Gamma \kappa_v) \xi^{-1} \nu \log(\E[q_{K-1}]) \\
        &\leq 2(q_{-1} + \kappa_9/\nu) + 8(4 + \Gamma \kappa_v)\xi^{-1} \nu \log(e + (4 + \Gamma \kappa_v)\xi^{-1} \nu),
    \end{align*}
    proving the first result. Thus, by \eqref{eq:btlemma1}, it follows that
    \begin{equation*}
        \sum_{k=0}^{K-1} \E\left[\alpha_k \|c_k\|_1\right] \leq \kappa_9 + (4 + \Gamma \kappa_v) \xi^{-1} \nu \log(\E[q_{K-1}]) \leq \kappa_9 + (4 + \Gamma \kappa_v) \xi^{-1} \nu \log(\kappa_{10}).
    \end{equation*}
    \qed
\end{proof}

From this proof, we can see that we still obtain a convergence rate of $\mathcal{O}(1/K)$ in terms of the average constraint violation. In addition, as in Section \ref{subsec:prechosen}, any second order terms in the convergence analysis can be split into either terms involving $\beta_k^2$ or $\alpha_k^2 \|c_k\|_1$ terms. Since $\alpha_k$ is bounded from above by a constant, it should be clear by the prior lemma that the sum of the $\alpha_k^2 \|c_k\|_1$ terms are bounded, in expectation, by a constant factor. In addition, given the bound on $\E[q_{K-1}]$, we can combine the analysis of Sections \ref{subsec:prechosen} and \ref{subsec:adaptive} to derive a converegence result with a worst-case complexity of $\mathcal{O}(\epsilon_\ell^{-4})$ and $\mathcal{O}(\epsilon_c^{-1})$, matching the results of Section \ref{subsec:prechosen}. We leave the full complexity analysis as an exercise to the reader.

\rev{Finally, we repeat a similar lemma in the case where $\beta_k$ is chosen adaptively by \eqref{eq:adabq} and \eqref{eq:adastep}.}

\begin{lemma} \label{lem:safeguardedls2}
    \rev{Let Assumptions \ref{assum:fcsmooth}, \ref{assum:Hbound}, and \ref{assum:stochG} hold and let $x_k$ be generated by Algorithm \ref{alg:tsssqpls}. Let $\beta_k$ be defined as in \eqref{eq:adabq} and \eqref{eq:adastep} for all $k$. Let
    \begin{align} \label{eq:kappa11}
        \kappa_{11}(K) &:= \xi^{-1}(\|c_0\|_1 - (2 + \Gamma \kappa_v/2) \nu^2 \log(q_{-1}^2) \\
        &\quad+ \nu^2 \eta^2 \Gamma \log \left(1+(\kappa_u^2 + \zeta^{-2} M)K/b_{-1}^2 \right)/(2 q_{-1}^2)) \nonumber
    \end{align}
    and
    \begin{equation} \label{eq:kappa12}
        \kappa_{12}(K) := 2(q_{-1} + \kappa_{11}(K)/\nu) + 8(4 + \Gamma \kappa_v)\xi^{-1} \nu \log(e + (4 + \Gamma \kappa_v)\xi^{-1} \nu).
    \end{equation}}
    
    \rev{Then, $\E[q_{K-1}] \leq \kappa_{11}$ and
    \begin{equation}
        \sum_{k=0}^{K-1} \E\left[\alpha_k \|c_k\|_1\right] \leq \kappa_9 + (4 + \Gamma \kappa_v) \xi^{-1} \nu \log(\kappa_{10})
    \end{equation}}
\end{lemma}

\begin{proof}
    By the proof of Lemma \ref{lem:safeguardedls1}, \eqref{eq:linesearchlemma} holds in this case as well. Thus, taking the expectation of this inequality, rearranging, and applying Lemma \ref{lem:logbound} twice, we have,
    \begin{align}
        &\sum_{k=0}^{K-1} \E\left[\alpha_k \|c_k\|_1\right] \nonumber \\
        &\leq \xi^{-1} \|c_0\|_1 + \E\left[\sum_{j \in \mathcal{K}_\alpha}  \frac{\xi^{-1} (2 + \Gamma \kappa_v/2) \nu^2}{q_{-1}^2 + \sum_{\ell \in \mathcal{K}_{\alpha}, \ell \leq j} \|c_{\ell}\|_1} \|c_j\|_1 + \frac{\xi^{-1} \nu^2 \beta_j^2 \Gamma}{2  q_{-1}^2} \|u_j\|^2\right] \nonumber \\
        &\leq \xi^{-1} \|c_0\|_1 + \xi^{-1}(2 + \Gamma \kappa_v/2) \nu^2 \E\left[\log(q_{-1}^2 + \sum_{j \in \mathcal{K}_\alpha} \|c_j\|_1)-\log(q_{-1}^2)\right] \nonumber \\
        &\quad+ \frac{\xi^{-1} \nu^2 \eta^2 \Gamma}{2  q_{-1}^2}\sum_{j=0}^{K-1} \E\left[\log\left(\frac{b_{-1}^2 + \sum_{k=0}^{K-1} \|u_k\|^2}{b_{-1}^2}\right)\right] \nonumber \\
        &\leq \xi^{-1} \|c_0\|_1 + \xi^{-1}(2 + \Gamma \kappa_v/2) \nu^2 \eta^2 (2\log(\E[q_{K-1}])-\log(q_{-1}^2)) \label{eq:btlemma1} \\
        &\quad+ \frac{\xi^{-1} \nu^2 \Gamma \log \left(1+(\kappa_u^2 + \zeta^{-2} M)K/b_{-1}^2 \right)}{2 q_{-1}^2}, \nonumber
    \end{align}
    where the final inequality follows by \eqref{eq:adaptiveubound}. From here, the proof follows by an identical argument to that of Lemma \ref{lem:safeguardedls1} with replacement of $\kappa_9$ by $\kappa_{11}(K)$ and $\kappa_{10}$ by $\kappa_{12}(K)$, respectively. \qed
\end{proof}

\rev{Clearly, this proof shows that we still obtain an $\tilde{\mathcal{O}}(1/K)$ in terms of the average constraint violation, as was the case in Section \ref{subsec:adaptive}. Similarly to the previous case, one can directly use this lemma with the proofs in Section \ref{subsec:adaptive} to prove the analagous $\tilde{\mathcal{O}}(\epsilon_\ell^{-4})$ and $\tilde{\mathcal{O}}(\epsilon_c^{-1})$ complexity result.}

\section{Numerical Experiments} \label{sec:numerical}

In this section, we numerically validate the performance of our proposed algorithm. We focus our attention on Algorithm \ref{alg:tsssqpls}, as it is a fully specified version of the generic Algorithm \ref{alg:tsssqp}. We consider two settings, equality constrained problems from the CUTEst collection \cite{NIMGould_DOrban_PLToint_2015} with simulated noise and \rev{constrained logistic regression problems using datasets from the LIBSVM collection \cite{CChang_CLin_2011}}.

In both cases, we compare Algorithm \ref{alg:tsssqpls} with the Github implementation of Algorithm 3 (which we refer to as ``SSQP" throughout this section) in \cite{ASBerahas_FECurtis_DPRobinson_BZhou_2021} and use the parameter settings provided in \cite{ASBerahas_FECurtis_DPRobinson_BZhou_2021}\footnote{\url{https://github.com/frankecurtis/StochasticSQP}}, with the exception of $\theta$, which we set as $\theta = 10^4$ for fair comparison. \rev{We use a variant of this algorithm which finds $\alpha_k$ using the procedure described in Section 5.1 of \cite{ASBerahas_FECurtis_MJONeill_DPRobinson_2023}}. \rev{In addition, for the constrained logistic regression problems, we implemented the Momentum-based Linearized Augmented Lagrangian Method (MLALM) of \cite{QShi_XWang_HWang_2026}. This algorithm was not implemented for the CUTEst problems as it requires access to specific mini-batches, which is not available given the structure of the noise. We describe the parameter settings for this algorithm in more detail in Section \ref{subsec:lr}.}

\rev{We consider three variants of Algorithm \ref{alg:tsssqp}, Algorithm \ref{alg:tsssqpls} with a fixed stepsize $\beta$, which refer to as ``TSSQP" throughout this section, Algorithm \ref{alg:tsssqpls} with $\beta_k$ chosen adaptively by \eqref{eq:adabq} and \eqref{eq:adastep} (``TSSQPU"), and Algorithm \ref{alg:tsssqp} with $\alpha_k$ and $\beta_k$ chosen by \eqref{eq:adabq} and \eqref{eq:adastep}, where we always select the stepsize $\alpha_k = \frac{\nu}{q_k}$ (``TSSQPUV"). As mentioned in Section \ref{subsec:adaptive}, we use $q_k^2 = q_{k-1}^k + \min\{\|c_k\|_1, \|v_k\|, \|v_k\|^2\}$ when updating the lower bound for $\alpha_k$ as it significantly improves the numerical performance of the algorithms, without harming the complexity results, beyond constant factors. In order to not overwhelm the main body of the paper with numerical results, the findings for the fully adaptive methods (TSSQPU and TSSQPUV) are presented in Appendix \ref{app:adaptive}.}

For all of the TSSQP algorithms, we use the following default parameter settings $\nu = 1$, $q_{-1} = 10^{-9}$, $\theta = 10^4$, $\xi = 10^{-3}$ and $\rho = 1/2$, while $\beta_k$ is specified below based on the specific problem and algorithm variant in use. \rev{In all of our experiments, we choose $H_k$ to be the identity matrix and compute the decompostion by first computing $p_k$ directly by \eqref{eq:linsys} and then find $u_k$ by projecting $p_k$ on $\text{Null}(J_k)$ and $v_k$ by setting $v_k = p_k - u_k$.}

\rev{Additional experiments involving variants of Algorithm \ref{alg:tsssqp} and \ref{alg:tsssqpls} applied to these problems can be found in Appendix \ref{app:adaptive} and Appendix \ref{app:tssqpa}. Information regarding the wallclock time required to solve these problems is given in \ref{app:time}.}

\subsection{CUTEst Experiments} \label{subsec:cutest}

We consider the performance of Algorithm \ref{alg:tsssqpls} on a subset of the equality constrained problems from the CUTEst collection \cite{NIMGould_DOrban_PLToint_2015}. We follow the experimental setup of \cite{ASBerahas_FECurtis_DPRobinson_BZhou_2021} and select equality constrained optimization problems for which (i) $f$ is not a constant function, (ii) n$+m \leq 1000$ and (iii) the Jacobian of $c$ was non-singular at every iteration performed in our experiments. This selection resulted in a total of 60 problems, each of which has specified initial point, which we used in our experiments. We consider these problems at six different noise levels of $\epsilon_N \in \{10^{-5}, 10^{-4}, 10^{-3}, 10^{-2}, 10^{-1}, 1\}$. At iteration $k$, a stochastic gradient is generated such that $g_k \sim \mathcal{N}(\nabla f(x_k), \epsilon_N I)$. For each problem and noise level, we ran a total of 20 instances for each algorithm. For each instance, all algorithms were given a total budget of 1000 iterations \rev{or constraint evaluations, in order to control for the potentially additional effor imposed by the backtracking routine in Algorithm \ref{alg:tsssqpls}.}

For every trial performed, we computed a resulting feasibility and optimality error. If a trial produced a sufficiently feasible iterate in the sense that $\|c_k\|_\infty \leq 10^{-6}$ for some $k$, then, we report the feasibility
error as $\|c_k\|_\infty$ and the optimality error was reported as $\|\nabla f(x_k) + J_k^T \yktrue\|_\infty$, where $\yktrue$ was computed
as a least-squares multiplier using the true gradient $\nabla f(x_k)$ and $J_k$. (This ensures that the reported optimality error is not
based on a stochastic gradient and is instead an accurate measure of optimality corresponding to the iterate $x_k$.)
On the other hand, if no sufficiently feasible iterate was produced on a given run, then the feasibility error and optimality
error were computed using the same measures at the least infeasible iterate computed. In addition to terminating when the maximum iteration limit is reached, the algorithms were terminated if they ever computed a point $x_k$ which was both sufficiently feasible and the stationarity error was smaller than $10^{-4}$. Algorithms SSQP and TSSQP were tuned over 5 stepsizes, $\beta \in \{10^{-4}, 10^{-3}, 10^{-2}, 10^{-1}, 1\}$. \rev{For each problem and noise level, the best performing $\beta$ was found by choosing the parameter setting which yielded the lowest average infeasibility error, when no parameter acheived sufficient feasibility on average, and the best average optimality error among sufficiently feasible solutions, when sufficient feasibility was acheived.} The results of this experiment are presented in Figure \ref{fig:boxplots}.

\begin{figure}[t]
\hspace{-1.5cm}
\includegraphics[width=1.25\textwidth]{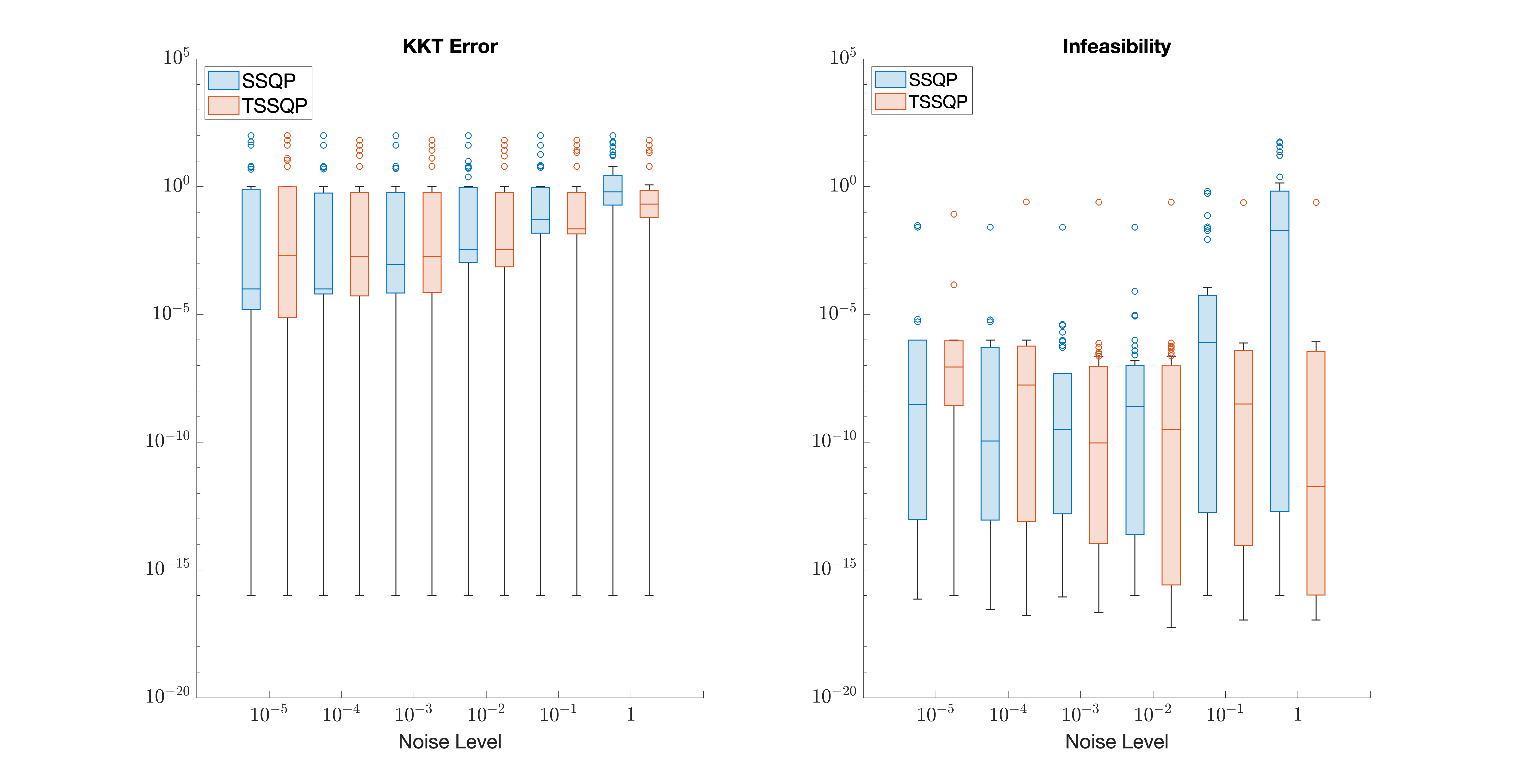}
\caption{Box plots of optimality error (left) and feasibility error (right) across various noise levels on CUTEst problems. SSQP is Algorithm 3 of \cite{ASBerahas_FECurtis_DPRobinson_BZhou_2021} and TSSQP is Algorithm \ref{alg:tsssqpls}.}
\label{fig:boxplots}
\end{figure}

As we can see from this plot, the computed stationarity error are relatively similar between these algorithms across all noise levels, with SSQP slightly outperforming TSSQP in stationarity error when the noise level is lower. This may be attributed to SSQP using an estimate of the merit parameter $\tau$, which is more likely to be well-behaved in a low noise setting. \rev{However, once the noise level increases to $\epsilon_N = 10^{-2}$, the gap between these algorithms vanishes for the stationarity error and is eventually in favor of TSSQP. On the other hand, when the noise level is low, these algorithms perform similarly in terms of the infeasibility error, with SSQP slightly outperforming TSSQP. However, as the noise level increases, the performance of SSQP degrades significantly with respect to infeasibility, while the performance of TSSQP is largely unchanged. We view this as confirmation of our theoretical results as it demonstrates the superior ability of TSSQP to converge with respect to constraint violation while having minimal to no impact on its ability converge with respect to the KKT error.}

\subsection{\rev{Constrained Logistic Regression}} \label{subsec:lr}

In this subsection, we consider equality constrained logistic regression problems of the form
\begin{equation}
    \underset{x \in \mathbb{R}^n}{\min} \ f(x) = \frac{1}{N} \sum_{i=1}^N \log\left(1 + e^{-y_i (X_i^T x)} \right) \ \text{s.t.} \ Ax = b, \ \|x\|_2^2 = 1,
\end{equation}
where $X \in \mathbb{R}^{n\times N}$ contains feature data of $N$ data points (with $X_i$ representing the $i$-th column of $X$), $y_i \in \{-1,1\}^N$ contains label data, and $A \in \mathbb{R}^{m \times n}$ and $b \in \mathbb{R}^m$. For the datasets $(X, y)$, we consider all binary classification datasets from the LIBSVM collection \cite{CChang_CLin_2011} for which $22 \leq n \leq 1000$ and $202 \leq N \leq 100000$, which resulted in 9 datasets. For datasets with multiple versions, e.g., the \{a1a,\dots, a9a\} datasets, we consider only the largest version.) The names and sizes of the datasets are given in Table \ref{tbl:lrprobs}. For the linear constraints, we chose $m = 10$ across all problems and generated $A$ and $b$ randomly with entries drawn from the standard normal distribution. The initial vector was chosen to be randomly distributed on the ball of norm $10^{-4}$, a small, random initialization.

\begin{table}[ht]
  \caption{Names and sizes of datasets. (Source: \cite{CChang_CLin_2011}.)}
  \label{tbl:lrprobs}
  \centering
  {\footnotesize
\begin{tabular}{lcc}\toprule
\textbf{Dataset}    & \textbf{Dimension ($n$)} & \textbf{ Datapoints ($N$)} \\ \hline
\texttt{a9a}             & $123$                          & $32,561$                                                                             \\
\texttt{ijcnn1}          & $22$                           & $49,990$                                                                             \\
\texttt{ionosphere}      & $34$                           & $351$                                                                                \\
\texttt{madelon}         & $500$                          & $2,000$                                                                              \\
\texttt{mushrooms}       & $112$                          & $8,124$                                                                              \\
\texttt{phising}         & $68$                           & $11,055$                                                                             \\
\texttt{sonar}           & $60$                           & $208$                                                                                \\
\texttt{splice}          & $60$                           & $1,000$                                                                              \\
\texttt{w8a}             & $300$                          & $49,749$       \\
\hline
\end{tabular}}
\end{table}

\begin{table}[tb]
  \caption{Average feasibility error, along with 95\% confidence intervals, of MLALM, SSQP, and TSSQP.  The results for the best-performing algorithm are shown in bold.}
  \label{tbl:logistic_regression1}
\centering
{\tiny
\begin{tabular}{l|c|c|c|c} \toprule
                                &                                 & \multicolumn{1}{c|}{\begin{tabular}[c]{@{}c@{}}MLALM\end{tabular}} & \multicolumn{1}{c|}{\begin{tabular}[c]{@{}c@{}}SSQP\end{tabular}} & \multicolumn{1}{c}{\begin{tabular}[c]{@{}c@{}}TSSQP \end{tabular}}                               \\ \hline
\multicolumn{1}{l|}{dataset}  & \multicolumn{1}{c|}{batch} & \multicolumn{1}{c|}{Feasibility}          & \multicolumn{1}{c|}{Feasibility}              & \multicolumn{1}{c}{Feasibility} \\ \hline
\texttt{a9a} & 16                     &       $1.46e-07 \pm 7.19e-08             $                    &     $      1.32e-05 \pm 2.99e-12          $                                                                 &     $ \pmb{3.95e-08 \pm 2.77e-16}           $                                                  \\
\texttt{a9a} & 128                     &      $\pmb{1.51e-07 \pm 2.81e-08}           $                      &        $   1.48e-05 \pm 1.12e-10        $                                                                  &     $ 5.25e-07 \pm 3.07e-15            $                                             \\ \hline
\texttt{ijccn1} & 16                     &    $   5.02e-07 \pm 1.10e-07           $                      &        $   1.20e-06 \pm 1.17e-13           $                                                                 &   $   \pmb{5.17e-09 \pm 1.22e-16}          $                                            \\
\texttt{ijccn1} & 128                     &   $    8.82e-07 \pm 5.55e-08            $                     &        $  1.20e-06 \pm 8.90e-14            $                                                             &    $  \pmb{1.03e-07 \pm 1.33e-16}            $                                  \\ \hline
\texttt{ionosphere} & 16                     &    $   3.67e-07 \pm 1.49e-07         $                        &         $  1.20e-03 \pm 5.75e-07        $                                                                   &      $  \pmb{6.90e-08 \pm 9.06e-16}        $                                   \\ 
\texttt{ionosphere} & 128                     &    $  2.39e-03 \pm 2.39e-04        $                         &       $    1.79e-03 \pm 3.04e-09    $                                                                        &   $   \pmb{4.49e-08 \pm 9.46e-17}  $                                       \\ \hline
\texttt{madelon} & 16                     &     $  \pmb{3.73e-07 \pm 1.56e-07}            $                      &       $    9.76e-01 \pm 3.02e-03           $                                                                &      $2.79e-04 \pm 1.44e-04           $                               \\
\texttt{madelon} & 128                     &      $ 7.36e-02 \pm 5.84e-02         $                        &       $    9.91e-01 \pm 1.03e-03      $                                                                      &   $  \pmb{ 1.31e-04 \pm 7.12e-08  }   $                                          \\ \hline
\texttt{mushrooms} & 16                     &    $   8.32e-07 \pm 5.35e-08        $                         &      $     1.00e-06 \pm 1.89e-12       $                                                                    &     $ \pmb{7.99e-08 \pm 3.06e-15}   $                                        \\
\texttt{mushrooms} & 128                     &   $    \pmb{8.59e-07 \pm 5.47e-08}     $                            &     $      1.52e-06 \pm 3.62e-12     $                                                                       & $      8.67e-07 \pm 4.00e-14  $                                          \\ \hline
\texttt{phishing} & 16                     &     $  1.01e-07 \pm 1.31e-07         $                        &         $  1.16e-06 \pm 1.14e-12         $                                                                   &    $   \pmb{6.01e-08 \pm 2.66e-16}   $                                        \\
\texttt{phishing} & 128                     &     $  \pmb{5.10e-07 \pm 1.50e-07}      $                           &      $     4.56e-06 \pm 3.67e-12   $                                                                        &   $    9.37e-07 \pm 2.33e-14   $                                          \\ \hline
\texttt{sonar} & 16                     &      $ 3.84e-03 \pm 4.68e-04 $                       &         $  3.35e-03 \pm 4.86e-09       $                                                                     &     $ \pmb{8.59e-10 \pm 1.51e-16}    $                                    \\
\texttt{sonar} & 128                     &    $   3.41e-01 \pm 1.22e-04      $                          &         $  2.35e-03 \pm 4.54e-09      $                                                                     &    $  \pmb{2.60e-06 \pm 1.04e-16}   $                                           \\ \hline
\texttt{splice} & 16                     &      $ \pmb{6.78e-07 \pm 1.01e-07}        $                         &         $  3.24e-03 \pm 2.06e-07      $                                                                      &     $\pmb{6.82e-07 \pm 2.01e-14}    $                                          \\
\texttt{splice} & 128                     &    $   1.61e-02 \pm 1.95e-04      $                           &       $    3.91e-03 \pm 3.24e-09    $                                                                       &    $  \pmb{5.46e-07 \pm 1.85e-14}  $                                           \\ \hline
\texttt{w8a} & 16                     &      $ 9.07e-07 \pm 4.67e-08      $                           &         $  4.87e-06 \pm 4.79e-13      $                                                                  &      $ \pmb{2.43e-08 \pm 1.94e-16}        $                                      \\
\texttt{w8a} & 128                     &     $  6.65e-07 \pm 9.64e-08    $                             &      $     5.31e-06 \pm 1.06e-12    $                                                                       &   $   \pmb{3.04e-07 \pm 3.92e-16}    $                                        \\ \hline
\end{tabular}}

  \caption{Average stationarity error, along with 95\% confidence intervals, of MLALM, SSQP, and TSSQP.  The results for the best-performing algorithm are shown in bold.}
  \label{tbl:logistic_regression2}
\centering
{\tiny
\begin{tabular}{l|c|c|c|c} \toprule
                                &                                 & \multicolumn{1}{c|}{\begin{tabular}[c]{@{}c@{}}MLALM\end{tabular}} & \multicolumn{1}{c|}{\begin{tabular}[c]{@{}c@{}}SSQP\end{tabular}} & \multicolumn{1}{c}{\begin{tabular}[c]{@{}c@{}}TSSQP \end{tabular}}                               \\ \hline
\multicolumn{1}{l|}{dataset}  & \multicolumn{1}{c|}{batch} & \multicolumn{1}{c|}{Stationarity}          & \multicolumn{1}{c|}{Stationarity}              & \multicolumn{1}{c}{Stationarity} \\ \hline
\texttt{a9a} & 16                     &       $7.98e-02 \pm 3.70e-02             $                    &     $      4.30e-02 \pm 3.61e-08          $                                                                 &     $ \pmb{2.92e-02 \pm 2.25e-10}           $                                                  \\
\texttt{a9a} & 128                     &      $6.19e-02 \pm 1.13e-03           $                      &        $   \pmb{1.07e-02 \pm 5.98e-07}        $                                                                  &     $ 1.16e-02 \pm 3.80e-10           $                                             \\ \hline
\texttt{ijccn1} & 16                     &    $   1.94e-01 \pm 2.20e-03           $                      &        $   \pmb{4.82e-02 \pm 2.68e-08}          $                                                                 &   $   4.88e-02 \pm 9.18e-15         $                                            \\
\texttt{ijccn1} & 128                     &   $    2.03e-01 \pm 7.47e-04            $                     &        $  1.59e-02 \pm 6.46e-11           $                                                             &    $  \pmb{9.10e-03 \pm 3.62e-13}            $                                  \\ \hline
\texttt{ionosphere} & 16                     &    $   1.77e-01 \pm 4.16e-02         $                        &         $  \pmb{7.22e-02 \pm 5.00e-05}        $                                                                   &      $  1.03e-01 \pm 7.93e-10       $                                   \\ 
\texttt{ionosphere} & 128                     &    $  1.44e-01 \pm 2.00e-02       $                         &       $    \pmb{5.20e-02 \pm 5.88e-08}    $                                                                        &   $   6.92e-02 \pm 4.21e-10  $                                       \\ \hline
\texttt{madelon} & 16                     &     $  5.38e+03 \pm 8.23e+00            $                      &       $    \pmb{1.88e+02 \pm 3.69e+01}           $                                                                &      $\pmb{1.57e+02 \pm 4.59e+01}          $                               \\
\texttt{madelon} & 128                     &      $ 4.76e+03 \pm 6.53e+02         $                        &       $    2.19e+02 \pm 5.13e+01      $                                                                      &   $  \pmb{ 8.79e+01 \pm 3.02e-03 }   $                                          \\ \hline
\texttt{mushrooms} & 16                     &    $   1.27e-01 \pm 1.47e-04        $                         &      $     3.83e-03 \pm 1.45e-08       $                                                                    &     $ \pmb{3.81e-03 \pm 1.56e-10}   $                                        \\
\texttt{mushrooms} & 128                     &   $    1.32e-01 \pm 4.39e-04     $                            &     $      2.84e-03 \pm 9.11e-09     $                                                                       & $      \pmb{2.48e-03 \pm 5.33e-10}  $                                          \\ \hline
\texttt{phishing} & 16                     &     $  1.76e-01 \pm 9.98e-05       $                        &         $  1.10e-02 \pm 1.07e-08         $                                                                   &    $   \pmb{7.30e-03 \pm 7.03e-11}   $                                        \\
\texttt{phishing} & 128                     &     $  1.76e-01 \pm 4.04e-05      $                           &      $     8.20e-03 \pm 4.23e-09   $                                                                        &   $    \pmb{4.04e-03 \pm 8.70e-11}   $                                          \\ \hline
\texttt{sonar} & 16                     &      $ 3.16e-01 \pm 6.76e-02 $                       &         $  \pmb{1.03e-01 \pm 1.25e-07}      $                                                                     &     $ 1.17e-01 \pm 1.37e-10    $                                    \\
\texttt{sonar} & 128                     &    $   8.52e-01 \pm 1.11e-02      $                          &         $  \pmb{2.80e-02 \pm 3.29e-08}      $                                                                     &    $  1.68e-01 \pm 1.38e-09   $                                           \\ \hline
\texttt{splice} & 16                     &      $ 4.28e+00 \pm 5.94e-01        $                         &         $  1.06e+00 \pm 3.09e-06      $                                                                      &     $ \pmb{4.15e-01 \pm 1.11e-08}   $                                          \\
\texttt{splice} & 128                     &    $   \pmb{4.08e-01 \pm 1.03e-02}      $                           &       $    1.92e-01 \pm 7.93e-09    $                                                                       &    $  \pmb{3.97e-01 \pm 5.03e-09}  $                                           \\ \hline
\texttt{w8a} & 16                     &      $ 2.07e-01 \pm 1.49e-03      $                           &         $  8.90e-03 \pm 2.49e-08      $                                                                  &      $ \pmb{5.46e-03 \pm 1.05e-11}        $                                      \\
\texttt{w8a} & 128                     &     $  2.02e-01 \pm 1.33e-03    $                             &      $     \pmb{3.07e-03 \pm 3.28e-09}    $                                                                       &   $   3.09e-03 \pm 7.04e-11    $                                        \\ \hline
\end{tabular}}

\end{table}

For each dataset, we considered two noise levels, where the level is dictated by the mini-batch size of each stochastic gradient estimate. For all problems, we used mini-batch sizes of 16 and 128. For each dataset and mini-batch size, we ran 20 instances with different
random seeds. A budget of 10 epochs (i.e., number of effective passes over the dataset) was used for all SQP methods. For fair comparison, MLALM was given a budget of 300 epochs for each problem and batchsize, as it does not require linear system solves to generate search directions. This additional budget made MLALM the most expensive method in terms of time, see Appendix \ref{app:time} for more details.

Similarly to the CUTEst problems, we considered 5 stepsizes for SSQP and TSSQP $\{10^{-4}, 10^{-3}, 10^{-2}, 10^{-1}, 1\}$ while the adaptive methods were run with $\eta = 1$. For MLALM, following the prescriptions in \cite{QShi_XWang_HWang_2026}, we set $\beta = \rho = K^{1/4}$, where $K$ is total number of iterations to be performed. We consider 6 possible values for $\alpha = \{0, 0.2, 0.4, 0.6, 0.8, 1\}$ and four possible stepsize choices $\eta = \{10^{-5} K^{-1/4}, 10^{-4}K^{-1/4}, 10^{-3} K^{-1/4}, 10^{-2}K^{-1/4}\}$, given a total of 24 different parameter settings. We follow a similar procedure as in the CUTEst experiments for reporting infeasibility and optimality error, with the caveat that we only record points at the end of each epoch. As before, if a succifiently feasible iterate is found ($\|c_k\|_\infty \leq 10^{-6}$), then we report the sufficiently feasible point with the lowest optimality error, $\|\nabla f(x_k) + J_k^T \yktrue\|_\infty$. Otherwise, we report the least infeasible point. We chose the best parameter setting for each algorithm and batch size via the same procedure as in Section \ref{subsec:cutest}.

The results of these experiments can be seen in Tables \ref{tbl:logistic_regression1} and \ref{tbl:logistic_regression2}. With respect to convergence in feasibility, TSSQP and MLALM perform the best, with TSSQP being the top performing algorithm more frequently than MLALM (13 times vs 4 times, with 1 tie). While not far behind in most instance, SSQP is never the top performing method. On the other hand, with respect to stationarity, SSQP and TSSQP are the top two performing methods, with TSSQP being the best algorithm 10 times versus 7 times for SSQP (with 1 tie). The MLALM algorithm is only the top performing algorithm with respect to feasibility on one instance. From these results, we conclude that the TSSQP method is the most effective overall as it performs well with respect to both measures the most frequently.

\section{Conclusion} \label{sec:conclusion}

In this paper, we propose and analyze a new SQP method for equality constrained optimization with a stochastic objective function. The algorithm uses a stepsize splitting scheme in order to improve upon the worst-case complexity of recently proposed stochastic SQP methods. We show that the proposed method matches the rate of convergence of a determinstic SQP method in terms of constraint violation and obtains the optimal rate for a stochastic method in terms of the gradient of the Lagrangian.

There are number of possible directions of future research. Fundamentally, this stepsize splitting scheme can be incorporated into any of the previously proposed stochastic SQP methods in the literature, including those for rank deficient Jacobians, inequality constraints, and inexact subproblem solutions. \rev{Extending the algorithm to incorporate inexact subproblem solutions could likely be done by enforcing conditions similar to those derived in \cite{FECurtis_DPRobinson_BZhou_2024A} and handling the additional error terms. Such an extension is outside the scope of the current manuscript.}


\bibliographystyle{siamplain}
\bibliography{references}

\appendix
\section{\rev{Numerical Performance of Adaptive Methods}} \label{app:adaptive}

In this appendix, we present our numerical results for the adaptive methods TSSQPU and TSSQPUV. For reference, we also include TSSQP in the results. We note that TSSQPU and TSSQPUV are fully adaptive in the sense that only default parameter settings were used ($\eta = 1$, $\nu = 1$), while the presented results for TSSQP were tuned for the best fixed stepsize $\beta$, using the same procedure as in Section \ref{sec:numerical}.

\begin{figure}[t]
\hspace{-1.5cm}
\includegraphics[width=1.25\textwidth]{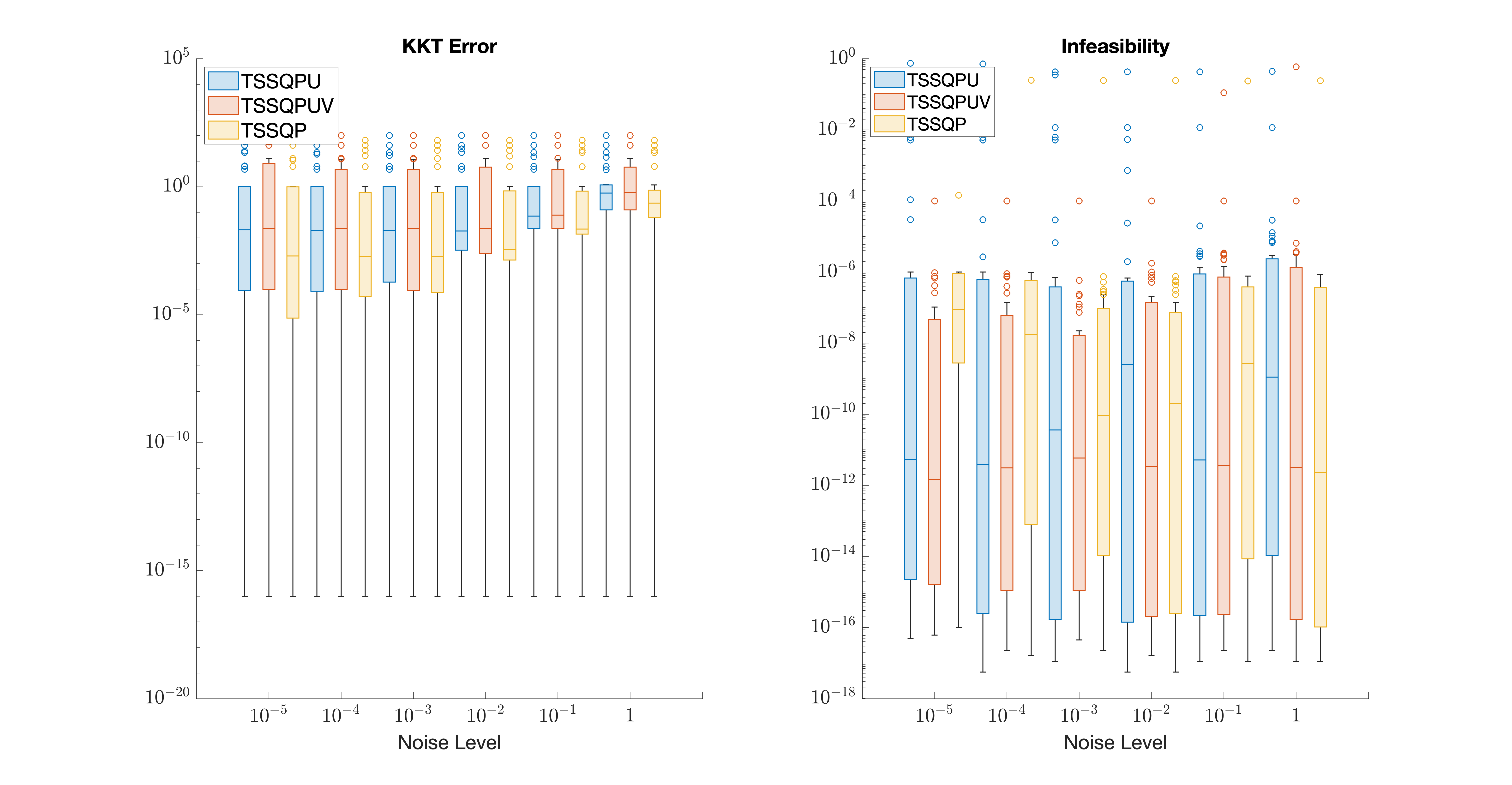}
\caption{Box plots of optimality error (left) and feasibility error (right) across various noise levels on CUTEst problems for TSSQPU and TSSQPUV. TSSQP is included for reference.}
\label{fig:boxplotAdaptive}
\end{figure}

\subsection{Adaptive Methods on CUTEst Problems}

In this subsection, we present the performance of TSSQPU and TSSQPUV on the CUTEst test set. For reference, we plot TSSQP along with TSSQPU and TSSQPUV in Figure \ref{fig:boxplotAdaptive}.

We observe that, with respect to the infeasibility measure, the three methods perform largely the same, with the fully adaptive methods performing slightly better than TSSQP. On the other hand, both adaptive methods struggle with respect to the stationarity error when compared with TSSQP. Considering that both adaptive methods do not use any tuning, while TSSQP does, we find the results promising for the adaptive methods, which may be useful in scenarios where parameter tuning is difficult or expensive. In addition, TSSQPUV may be useful in contexts where computation of $c(x)$ is expensive, as it does not rely on the backtracking procedure in Algorithm \ref{alg:tsssqpls} but still manages to converge to approximate feasibility consistently.

\subsection{Adaptive Methods on Logistic Regression Problems}

The results for the adaptive methods on the logistic regression problems from Subsection \ref{subsec:lr} can be found in Tables \ref{tbl:logistic_regressionapp3} and \ref{tbl:logistic_regressionapp4}, where TSSQP is included for reference.

From these tables, we observe that TSSQPU performs reasonably well, especially with respect to stationarity error, where it is the best algorithm more often than TSSQP. In terms of feasibility error, TSSQP is more often the superior algorithm, but TSSQPU still performs reasonably well, even without tuning to the problem instance. We view this as confirmation of the algorithm's ability to adapt to the problem structure effectively. On the other hand, TSSQPUV completely fails to converge on the majority of problems, both in terms of stationarity and feasibility (the initial infeasibility is $\|c(x_0)\|_\infty \approx 1$, so the instances where TSSQPUV records 1 for feasibility suggests that it never finds a more feasible point than the initial one). These results suggest that the linesearch procedure in Algorithm \ref{alg:tsssqpls} may be very useful for certain problem instances, such as the logistic regression problems explored here, as it is the only difference between TSSQPU and TSSQPUV, while the gap between their results is significant.

\begin{table}[tb]
  \caption{Average feasibility error, along with 95\% confidence intervals, of TSSQPU and TSSQPUV. For reference, we also include TSSQP.  The results for the best-performing algorithm are shown in bold.}
  \label{tbl:logistic_regressionapp3}
\centering
{\tiny
\begin{tabular}{l|c|c|c|c} \toprule
                                &                                 & \multicolumn{1}{c|}{\begin{tabular}[c]{@{}c@{}}TSSQPU\end{tabular}} & \multicolumn{1}{c|}{\begin{tabular}[c]{@{}c@{}}TSSQPUV\end{tabular}} & \multicolumn{1}{c}{\begin{tabular}[c]{@{}c@{}}TSSQP \end{tabular}}                               \\ \hline
\multicolumn{1}{l|}{dataset}  & \multicolumn{1}{c|}{batch} & \multicolumn{1}{c|}{Feasibility}          & \multicolumn{1}{c|}{Feasibility}              & \multicolumn{1}{c}{Feasibility} \\ \hline
\texttt{a9a} & 16                     &       $5.96e-06 \pm 1.14e-09             $                    &     $      5.68e-02 \pm 5.49e-12          $                                                                 &     $ \pmb{3.95e-08 \pm 2.77e-16}           $                                                  \\
\texttt{a9a} & 128                     &      $\pmb{7.65e-05 \pm 1.21e-06}           $                      &        $   1.00e+00 \pm 2.13e-16        $                                                                  &     $ 5.25e-07 \pm 3.07e-15            $                                             \\ \hline
\texttt{ijccn1} & 16                     &    $   8.29e-07 \pm 2.57e-10           $                      &        $   1.06e-06 \pm 3.92e-13           $                                                                 &   $   \pmb{5.17e-09 \pm 1.22e-16}          $                                            \\
\texttt{ijccn1} & 128                     &   $    2.03e-05 \pm 4.30e-09            $                     &        $  1.00e+00 \pm 2.13e-16            $                                                             &    $  \pmb{1.03e-07 \pm 1.33e-16}            $                                  \\ \hline
\texttt{ionosphere} & 16                     &    $   2.84e-04 \pm 4.45e-06         $                        &         $  1.00e+00 \pm 2.13e-16        $                                                                   &      $  \pmb{6.90e-08 \pm 9.06e-16}        $                                   \\ 
\texttt{ionosphere} & 128                     &    $  1.55e-03 \pm 1.39e-05       $                         &       $    1.00e+00 \pm 2.13e-16    $                                                                        &   $   \pmb{4.49e-08 \pm 9.46e-17}  $                                       \\ \hline
\texttt{madelon} & 16                     &     $  \pmb{2.67e-06 \pm 2.50e-06}            $                      &       $    1.00e+00 \pm 2.13e-16           $                                                                &      $2.79e-04 \pm 1.44e-04           $                               \\
\texttt{madelon} & 128                     &      $ 1.45e-04 \pm 8.61e-05         $                        &       $    1.00e+00 \pm 2.13e-16      $                                                                      &   $  \pmb{ 1.31e-04 \pm 7.12e-08  }   $                                          \\ \hline
\texttt{mushrooms} & 16                     &    $   5.57e-06 \pm 1.37e-08        $                         &      $     1.00e+00 \pm 2.13e-16       $                                                                    &     $ \pmb{7.99e-08 \pm 3.06e-15}   $                                        \\
\texttt{mushrooms} & 128                     &   $    \pmb{2.07e-06 \pm 2.88e-07}     $                            &     $      1.00e+00 \pm 2.13e-16     $                                                                       & $      8.67e-07 \pm 4.00e-14  $                                          \\ \hline
\texttt{phishing} & 16                     &     $  2.21e-05 \pm 1.63e-08         $                        &         $  1.00e+00 \pm 2.13e-16         $                                                                   &    $   \pmb{6.01e-08 \pm 2.66e-16}   $                                        \\
\texttt{phishing} & 128                     &     $  \pmb{1.71e-04 \pm 4.64e-07}      $                           &      $     1.00e+00 \pm 2.13e-16   $                                                                        &   $    9.37e-07 \pm 2.33e-14   $                                          \\ \hline
\texttt{sonar} & 16                     &      $ 4.76e-04 \pm 7.92e-07 $                       &         $  1.00e+00 \pm 2.13e-16       $                                                                     &     $ \pmb{8.59e-10 \pm 1.51e-16}    $                                    \\
\texttt{sonar} & 128                     &    $   5.94e-02 \pm 4.50e-04      $                          &         $  1.00e+00 \pm 2.13e-16      $                                                                     &    $  \pmb{2.60e-06 \pm 1.04e-16}   $                                           \\ \hline
\texttt{splice} & 16                     &      $ \pmb{4.70e-04 \pm 7.15e-07}        $                         &         $  1.00e+00 \pm 2.13e-16      $                                                                      &     $\pmb{6.82e-07 \pm 2.01e-14}    $                                          \\
\texttt{splice} & 128                     &    $   1.36e-03 \pm 2.16e-06      $                           &       $    1.00e+00 \pm 2.13e-16    $                                                                       &    $  \pmb{5.46e-07 \pm 1.85e-14}  $                                           \\ \hline
\texttt{w8a} & 16                     &      $ 9.97e-06 \pm 8.54e-10      $                           &         $  1.04e-04 \pm 3.26e-13      $                                                                  &      $ \pmb{2.43e-08 \pm 1.94e-16}        $                                      \\
\texttt{w8a} & 128                     &     $  9.83e-05 \pm 1.67e-07    $                             &      $     1.00e+00 \pm 2.13e-16   $                                                                       &   $   \pmb{3.04e-07 \pm 3.92e-16}    $                                        \\ \hline
\end{tabular}}

  \caption{Average stationarity error, along with 95\% confidence intervals, of TSSQPU and TSSQPUV. For reference, we also include TSSQP.  The results for the best-performing algorithm are shown in bold.}
  \label{tbl:logistic_regressionapp4}
\centering
{\tiny
\begin{tabular}{l|c|c|c|c} \toprule
                                &                                 & \multicolumn{1}{c|}{\begin{tabular}[c]{@{}c@{}}TSSQPU\end{tabular}} & \multicolumn{1}{c|}{\begin{tabular}[c]{@{}c@{}}TSSQPUV\end{tabular}} & \multicolumn{1}{c}{\begin{tabular}[c]{@{}c@{}}TSSQP \end{tabular}}                               \\ \hline
\multicolumn{1}{l|}{dataset}  & \multicolumn{1}{c|}{batch} & \multicolumn{1}{c|}{Stationarity}          & \multicolumn{1}{c|}{Stationarity}              & \multicolumn{1}{c}{Stationarity} \\ \hline
\texttt{a9a} & 16                     &       $3.75e-02 \pm 1.74e-07             $                    &     $      1.59e-01 \pm 1.05e-05          $                                                                 &     $ \pmb{2.92e-02 \pm 2.25e-10}           $                                                  \\
\texttt{a9a} & 128                     &      $1.39e-02 \pm 1.15e-06           $                      &        $   1.47e+03 \pm 1.81e-03       $                                                                  &     \pmb{$ 1.16e-02 \pm 3.80e-10           $}                                             \\ \hline
\texttt{ijccn1} & 16                     &    $   \pmb{4.06e-02 \pm 3.82e-09}           $                      &        $   1.24e-01 \pm 6.21e-07          $                                                                 &   $   4.88e-02 \pm 9.18e-15         $                                            \\
\texttt{ijccn1} & 128                     &   $    \pmb{8.69e-03 \pm 2.67e-08}            $                     &        $  2.00e+03 \pm 2.75e-03           $                                                             &    $  9.10e-03 \pm 3.62e-13            $                                  \\ \hline
\texttt{ionosphere} & 16                     &    $   1.65e-01 \pm 6.96e-05         $                        &         $  2.70e+03 \pm 2.07e-02        $                                                                   &      $  \pmb{1.03e-01 \pm 7.93e-10}       $                                   \\ 
\texttt{ionosphere} & 128                     &    $  \pmb{1.88e-02 \pm 2.76e-05}      $                         &       $    2.70e+03 \pm 4.95e-03    $                                                                        &   $   6.92e-02 \pm 4.21e-10  $                                       \\ \hline
\texttt{madelon} & 16                     &     $  \pmb{1.03e+02 \pm 2.91e+01}            $                      &       $    7.76e+02 \pm 1.71e+01           $                                                                &      $1.57e+02 \pm 4.59e+01          $                               \\
\texttt{madelon} & 128                     &      $ \pmb{8.83e+01 \pm 1.55e+01}         $                        &       $    7.62e+02 \pm 1.08e+01      $                                                                      &   $  \pmb{ 8.79e+01 \pm 3.02e-03 }   $                                          \\ \hline
\texttt{mushrooms} & 16                     &    $   1.25e-02 \pm 2.55e-07        $                         &      $     1.37e+03 \pm 1.03e-02       $                                                                    &     $ \pmb{3.81e-03 \pm 1.56e-10}   $                                        \\
\texttt{mushrooms} & 128                     &   $    \pmb{1.95e-03 \pm 1.53e-04}     $                            &     $     1.37e+03 \pm 2.59e-03     $                                                                       & $      2.48e-03 \pm 5.33e-10  $                                          \\ \hline
\texttt{phishing} & 16                     &     $  \pmb{5.13e-03 \pm 9.04e-08}       $                        &         $  1.56e+03 \pm 2.13e-03         $                                                                   &    $   7.30e-03 \pm 7.03e-11   $                                        \\
\texttt{phishing} & 128                     &     $  9.24e-03 \pm 1.20e-07      $                           &      $     1.56e+03 \pm 7.63e-04   $                                                                        &   $    \pmb{4.04e-03 \pm 8.70e-11}   $                                          \\ \hline
\texttt{sonar} & 16                     &      $ \pmb{8.57e-02 \pm 3.50e-06} $                       &         $  1.59e+03 \pm 1.63e-02      $                                                                     &     $ 1.17e-01 \pm 1.37e-10    $                                    \\
\texttt{sonar} & 128                     &    $   \pmb{5.94e-02 \pm 4.50e-04}      $                          &         $  1.59e+03 \pm 7.26e-03      $                                                                     &    $  1.68e-01 \pm 1.38e-09   $                                           \\ \hline
\texttt{splice} & 16                     &      $ \pmb{2.88e-01 \pm 2.03e-05}        $                         &         $  1.59e+03 \pm 1.53e-01      $                                                                      &     $ 4.15e-01 \pm 1.11e-08   $                                          \\
\texttt{splice} & 128                     &    $   \pmb{4.27e-02 \pm 1.07e-05}      $                           &       $    1.59e+03 \pm 6.15e-02    $                                                                       &    $ 3.97e-01 \pm 5.03e-09  $                                           \\ \hline
\texttt{w8a} & 16                     &      $ 2.81e-02 \pm 5.57e-08      $                           &         $  1.48e-01 \pm 7.84e-07     $                                                                  &      $ \pmb{5.46e-03 \pm 1.05e-11}        $                                      \\
\texttt{w8a} & 128                     &     $     9.25e-03 \pm 2.08e-06    $                             &      $     9.52e+02 \pm 1.65e-03    $                                                                       &   $   \pmb{3.09e-03 \pm 7.04e-11}    $                                        \\ \hline
\end{tabular}}

\end{table}

\section{\rev{Numerical Performance of Linesearch Variants and Stepsize Plots}} \label{app:tssqpa}

In this subsection, we report results for a variant of Algorithm \ref{alg:tsssqpls} in which the term $q_k$ always accumulates (i.e., in line 10 of Algorithm \ref{alg:tsssqpls}, we set $q_k = \hat{q}_k$). The experiments were carried out in the same manner as those in Section \ref{sec:numerical} and we refer to this variant as ``TSSQPA", for always accumulate, in this section. To see the effect of this modification, we only compare directly with TSSQP. In addition, at the end of this appendix, we include plots of the behavior of the adaptive stepsizes computed by TSSQP, TSSQPA and TSSQPU on a logistic regression problem.

\subsection{CUTEst Results}

First, we report our results on the CUTEst problem set, which can be seen in Figure \ref{fig:tssqpa}. We see relatively similar performance between the methods across all noise levels, with a slight advantage in KKT error for TSSQP and a slightly larger advantage in infeasibility error for TSSQPA. This aligns with the idea that TSSQPA always accumulates the $\|c_k\|_1$ in $q_k$, thus leading to overall shorter steps, which may hinder progress in acheiving optimality.

\begin{figure}[t]
\hspace{-1.5cm}
\includegraphics[width=1.25\textwidth]{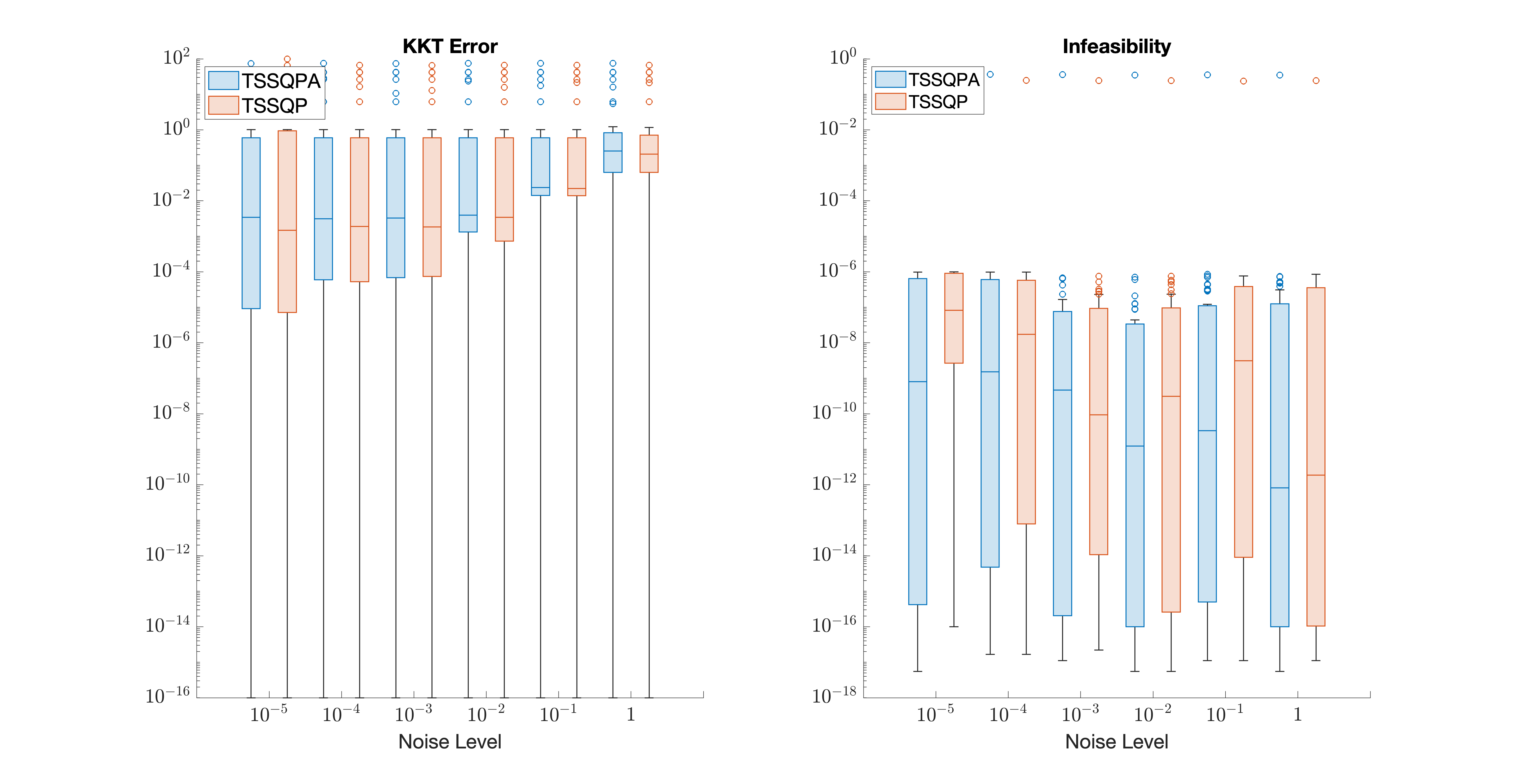}
\caption{Box plots of optimality error (left) and feasibility error (right) across various noise levels on CUTEst problems of TSSQPA vs TSSQP.}
\label{fig:tssqpa}
\end{figure}

\subsection{Logistic Regression Results}

Next, we report the results for the logistic regression problems, which can be seen below in Table \ref{tbl:logistic_regressionapp1} and \ref{tbl:logistic_regressionapp2}. We see that TSSQPA and TSSQP perform very similarly in terms of both feasibility and stationarity, with TSSQPA having a slight edge in feasibility and TSSQP having a slight edge in stationarity. On a whole, these results suggest that both methods are effective and are reasonable choices for these types of problems.

\begin{table}[tb]
  \caption{Average feasibility error, along with 95\% confidence intervals, of TSSQP and TSSQPA.  The results for the best-performing algorithm are shown in bold.}
  \label{tbl:logistic_regressionapp1}
\centering
{\tiny
\begin{tabular}{l|c|c|c} \toprule
                                &                                 & \multicolumn{1}{c|}{\begin{tabular}[c]{@{}c@{}}TSSQPA\end{tabular}} & \multicolumn{1}{c|}{\begin{tabular}[c]{@{}c@{}}TSSQP\end{tabular}}                             \\ \hline
\multicolumn{1}{l|}{dataset}  & \multicolumn{1}{c|}{batch} & \multicolumn{1}{c|}{Feasibility}          & \multicolumn{1}{c|}{Feasibility}             \\ \hline
\texttt{a9a} & 16         &  $2.15e-07 \pm 2.93e-15$                                                                 &     $ \pmb{3.95e-08 \pm 2.77e-16}           $                                                  \\
\texttt{a9a} & 128        &   \pmb{$1.96e-07 \pm 1.70e-14$}                                                                 &     $ 5.25e-07 \pm 3.07e-15            $                                             \\ \hline
\texttt{ijccn1} & 16        & $2.55e-08 \pm 1.43e-16$                                                            &   $   \pmb{5.17e-09 \pm 1.22e-16}          $                                            \\
\texttt{ijccn1} & 128          &  $5.12e-07 \pm 1.63e-16$          &  $  \pmb{1.03e-07 \pm 1.33e-16}            $                                 \\ \hline
\texttt{ionosphere} & 16          &                $3.74e-07 \pm 5.02e-15$                                                       &      $  \pmb{6.90e-08 \pm 9.06e-16}        $                                   \\ 
\texttt{ionosphere} & 128    &  $\pmb{3.90e-08 \pm 2.28e-16}$   &   $   4.49e-08 \pm 9.46e-17  $                                       \\ \hline
\texttt{madelon} & 16 & $\pmb{1.91e-05 \pm 1.18e-06}$ &              $2.79e-04 \pm 1.44e-04           $                               \\
\texttt{madelon} & 128   & $\pmb{3.58e-06 \pm 1.95e-12}$   &         $  1.31e-04 \pm 7.12e-08     $                                          \\ \hline
\texttt{mushrooms} & 16  &  $4.38e-07 \pm 2.35e-13$   &     $ \pmb{7.99e-08 \pm 3.06e-15}   $                                      \\
\texttt{mushrooms} & 128    &   $\pmb{2.40e-07 \pm 3.17e-13}$      & $      8.67e-07 \pm 4.00e-14  $                                          \\ \hline
\texttt{phishing} & 16   & $4.71e-07 \pm 2.39e-15$    &    $   \pmb{6.01e-08 \pm 2.66e-16}   $                                        \\
\texttt{phishing} & 128   &    $\pmb{2.72e-07 \pm 3.23e-14}$                      &   $    9.37e-07 \pm 2.33e-14  $                                          \\ \hline
\texttt{sonar} & 16             &   $\pmb{1.03e-10 \pm 9.84e-14}$                              &     $ 8.59e-10 \pm 1.51e-16    $                                    \\
\texttt{sonar} & 128     &  $\pmb{7.76e-10 \pm 1.63e-16}$  &    $  2.60e-06 \pm 1.04e-16   $                                           \\ \hline
\texttt{splice} & 16   &  $\pmb{7.74e-08 \pm 3.29e-15}$  &     $ 6.82e-07 \pm 2.01e-14    $                                          \\
\texttt{splice} & 128    &  $1.58e-07 \pm 5.32e-15$   &    $  \pmb{5.46e-07 \pm 1.85e-14}  $                                           \\ \hline
\texttt{w8a} & 16      &      $1.04e-07 \pm 2.55e-15$      &      $ \pmb{2.43e-08 \pm 1.94e-16}        $                                      \\
\texttt{w8a} & 128     &    $\pmb{1.61e-07 \pm 1.31e-14}$      &   $   3.04e-07 \pm 3.92e-16    $                                        \\ \hline
\end{tabular}}

  \caption{Average stationarity error, along with 95\% confidence intervals, of TSSQP and TSSQPA.  The results for the best-performing algorithm are shown in bold.}
  \label{tbl:logistic_regressionapp2}
\centering
{\tiny
\begin{tabular}{l|c|c|c} \toprule
                                &                                 & \multicolumn{1}{c|}{\begin{tabular}[c]{@{}c@{}}TSSQPA\end{tabular}} & \multicolumn{1}{c|}{\begin{tabular}[c]{@{}c@{}}TSSQP\end{tabular}}                               \\ \hline
\multicolumn{1}{l|}{dataset}  & \multicolumn{1}{c|}{batch} & \multicolumn{1}{c|}{Stationarity}          & \multicolumn{1}{c|}{Stationarity}        \\ \hline
\texttt{a9a} & 16          &       $3.76e-02 \pm 4.05e-09$                                                             &     $ \pmb{2.92e-02 \pm 2.25e-10}           $                                                  \\
\texttt{a9a} & 128                     &   $1.13e-01 \pm 6.33e-09$                                                                   &     $ \pmb{1.16e-02 \pm 3.80e-10}           $                                             \\ \hline
\texttt{ijccn1} & 16                     &   $4.89e-02 \pm 1.01e-11$                                                                &   \pmb{$   4.88e-02 \pm 9.18e-15$}                                            \\
\texttt{ijccn1} & 128                     &    \pmb{$8.91e-03 \pm 1.45e-11$}                                                            &    $  9.10e-03 \pm 3.62e-13            $                                  \\ \hline
\texttt{ionosphere} & 16     &   $1.11e-01 \pm 9.27e-10$                 &      $  \pmb{1.03e-01 \pm 7.93e-10}        $                                   \\ 
\texttt{ionosphere} & 128      &     $1.35e-01 \pm 3.34e-10$                                                                      &   \pmb{$   6.92e-02 \pm 4.21e-10  $}                                       \\ \hline
\texttt{madelon} & 16         &       \pmb{$1.20e+02 \pm 1.01e+01$}                                                             &      $1.57e+02 \pm 4.59e+01
         $                               \\
\texttt{madelon} & 128          &      \pmb{$1.33e+01 \pm 2.32e-05$}               &   $  8.79e+01 \pm 3.02e-03   $                                          \\ \hline
\texttt{mushrooms} & 16         &      \pmb{$3.79e-03 \pm 5.55e-09$}           &     $ 3.81e-03 \pm 1.56e-10   $                                        \\
\texttt{mushrooms} & 128       &      $6.26e-03 \pm 1.18e-08$        & $      \pmb{2.48e-03 \pm 5.33e-10}  $                                          \\ \hline
\texttt{phishing} & 16      &        $1.12e-02 \pm 2.24e-10$      &    $   \pmb{7.30e-03 \pm 7.03e-11}   $                                        \\
\texttt{phishing} & 128   &    $8.18e-03 \pm 7.24e-10$        &   $    \pmb{4.04e-03 \pm 8.70e-11}   $                                          \\ \hline
\texttt{sonar} & 16    &    \pmb{$1.17e-01 \pm 1.66e-08$}           &     $ \pmb{1.17e-01 \pm 1.37e-10}    $                                    \\
\texttt{sonar} & 128     &    \pmb{$8.13e-02 \pm 2.05e-10$}        &    $  1.68e-01 \pm 1.38e-09   $                                           \\ \hline
\texttt{splice} & 16     &    $4.53e-01 \pm 1.58e-08$            &     \pmb{$4.15e-01 \pm 1.11e-08   $}                                          \\
\texttt{splice} & 128     &   $4.01e-01 \pm 5.01e-09$         &    \pmb{$  3.97e-01 \pm 5.03e-09  $}                                           \\ \hline
\texttt{w8a} & 16               &   $2.04e-02 \pm 3.46e-09$   &      $ \pmb{5.46e-03 \pm 1.05e-11}        $                                      \\
\texttt{w8a} & 128   &        $8.76e-02 \pm 5.26e-09$        &   $   \pmb{3.09e-03 \pm 7.04e-11}    $                                        \\ \hline
\end{tabular}}

\end{table}

\subsection{Stepsize Behavior of Different Adaptive Methods}

In this subsection, we provide plots of the behavior of the adaptive stepsizes computed by three methods, TSSQP, TSSQPA, and TSSQPU on the logistic regression problem w8a with a batch size of 16, as it is one of the largest problems in the collection. For this problem, the best settings for $\beta$ was $10^{-3}$ for TSSQP and $10^{-4}$ for TSSQPA. We used these settings when generating the following plots.

The plots for TSSQP and TSSQPA are given in Figures \ref{fig:tssqpstepsize} and \ref{fig:tssqpastepsize}. We focus on the first 30 iterations of TSSQPA, as the stepsize drops rapidly in that interval then remains largely unchanged for the rest of hte procedure. We see that in both plots, the stepsize lower bound ($\nu/\sqrt{q_k}$) quickly decreases before becoming essentially flat, due to the constraint violation becoming very small. However, in the case of TSSQP, the lower bound stops at a much higher value ($\approx 0.52$) while TSSQPA stops at a lower value ($\approx 0.02$). This smaller stepsize impacts the ability of TSSQPA to acheive as accurate of a stationarity measure, as seen in Table \ref{tbl:logistic_regressionapp2}. Either way, we see that the adaptive stepsize strategy quickly falls to a level such that the algorithm can make significant progress towards feasibility (even if it does increase initially, such as in the case of TSSQPA).

Next, in Figure \ref{fig:betaU}, we plot the behavior of $\beta_k$ when it is chosen adaptively by \eqref{eq:adabq} and \eqref{eq:adastep}. Due to the way $\beta_k$ is constructed, the expected tail behavior is $\beta_k \sim 1/\sqrt{k}$. As such, we fit a curve of the form $a/\sqrt{k}$ and plot this as well. We see that this fit is relatively accurate, suggesting that after an initial adaptive period where $\beta$ falls rapidly to $\beta_k \approx 10^{-1}$, it settles into a decaying stepsize regime with a rate of decay proportional to $1/\sqrt{k}$. We note that this is essentially the behavior predicted by the quantity $\kappa_8(K)$ in Corollary \ref{cor:adacomplexity} (see \eqref{eq:bkbound}).

\begin{figure}[t]
\centering
\includegraphics[width=1\textwidth]{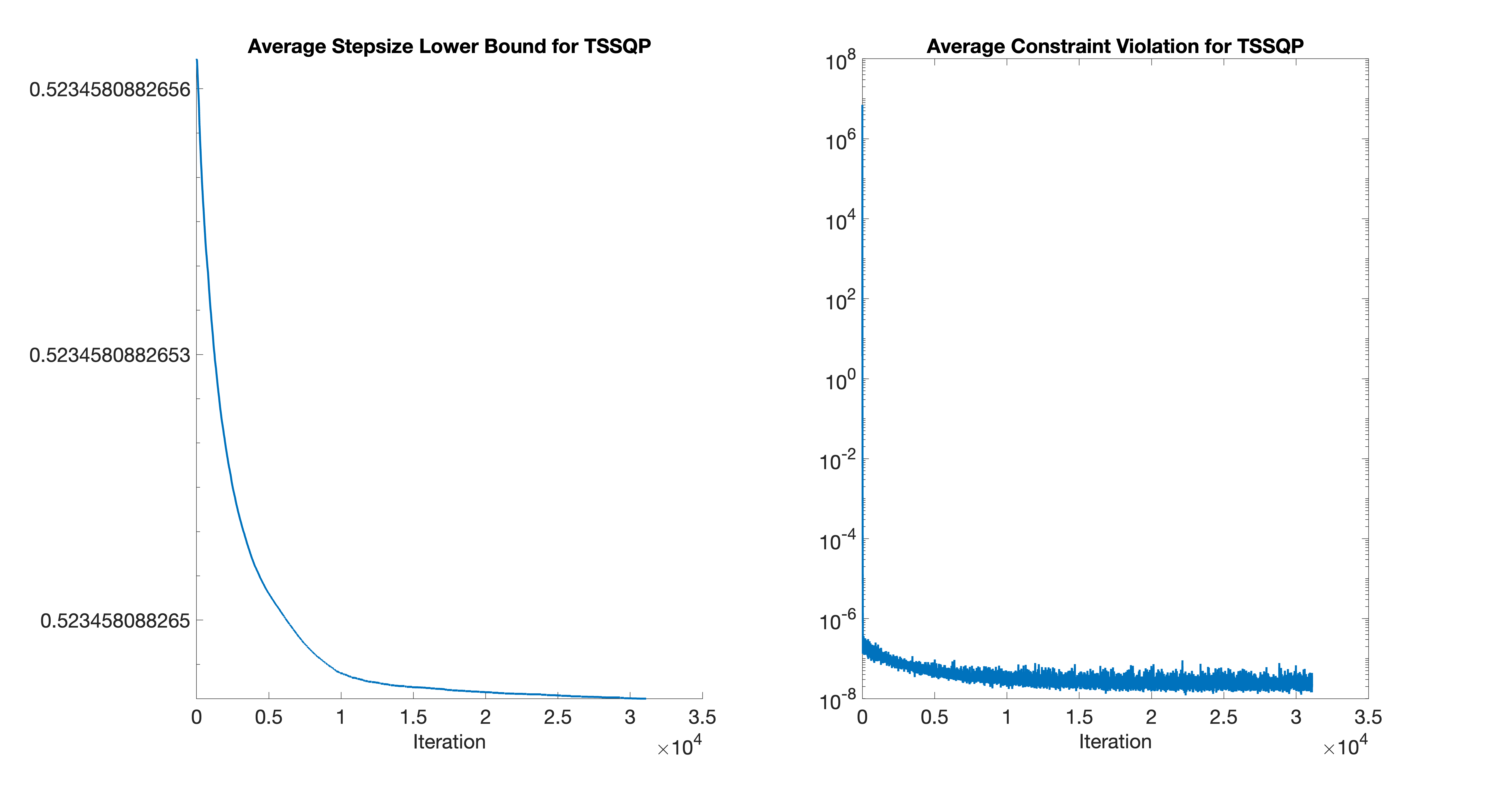}
\caption{Stepsize lower bound of TSSQP on w8a with a batchsize of 16 and the constraint violation.}\label{fig:tssqpstepsize}
\end{figure}

\begin{figure}[t]
\centering
\includegraphics[width=1\textwidth]{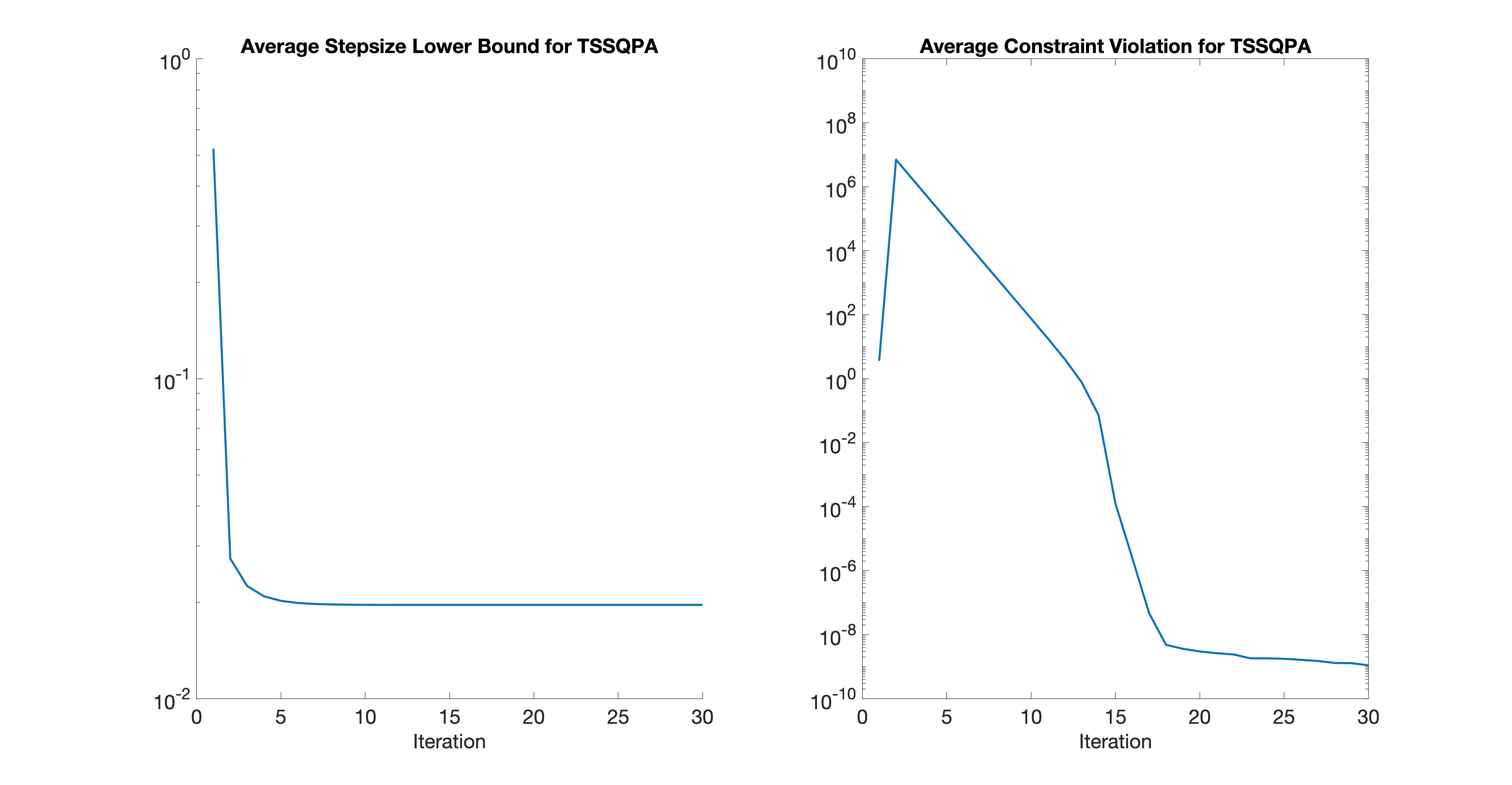}
\caption{Stepsize lower bound of TSSQPA on w8a with a batchsize of 16 and the constraint violation. We focus on the first 30 iterations, as the lower bound changes rapidly then remains largely flat afterwards.}\label{fig:tssqpastepsize}
\end{figure}

\begin{figure}[t]
\centering
\includegraphics[width=1\textwidth]{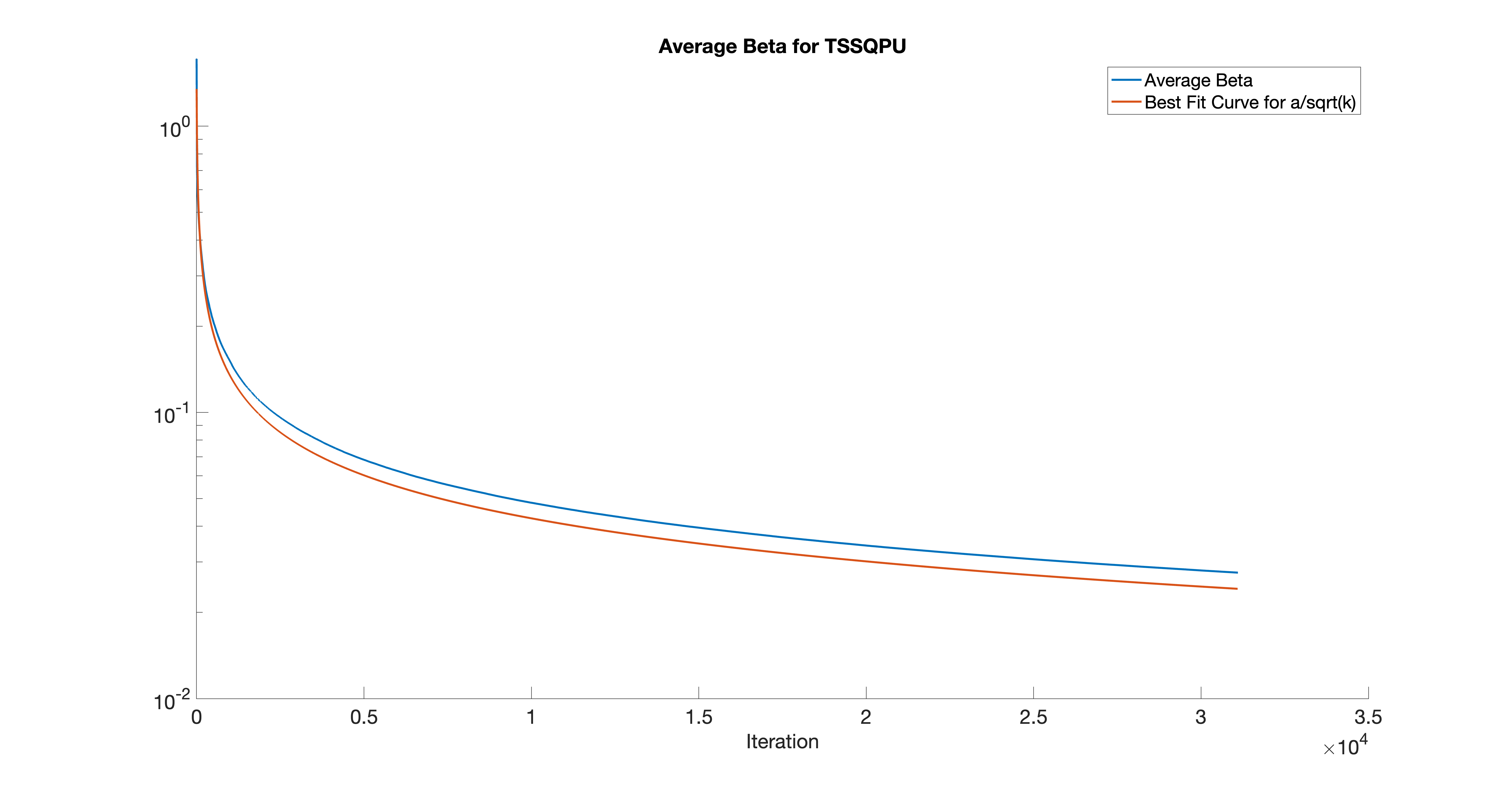}
\caption{Adaptive value of $\beta_k$ for TSSQPU on w8a. We also include a curve which is fit to $a/\sqrt{k}$, with $a = 4.266$, to compare with the expected tail behavior.}\label{fig:betaU}
\end{figure}

\section{\rev{Wallclock Performance of Methods}} \label{app:time}

In this appendix, we provide information on the wallclock performance of the numerical experiments. Caution should be exercised when considering the results of this appendix as the numerical experiments were run in a cluster environment, so there is no guarantee all experiments were run on identical hardware. We present two main pieces of information in this appendix: overall wallclock time of the methods and, for TSSQP, the relative wallclock time required to compute the decomposition $u$ and $v$ given a search direction $d$.

\subsection{Time Comparison for CUTEst Problems}

In this subsection, we compare the time for SSQP and TSSQP to solve the CUTEst problems from Subsection \ref{subsec:cutest}. To compare these quantities, for each problem, we compute the following measure:
\begin{equation} \label{eq:timecompare}
    \frac{T_{SSQP}(i) - T_{TSSQP}(i)}{\max\{T_{SSQP}(i), T_{TSSQP}(i)\}},
\end{equation}
where $T_{SSQP}(i)$ and $T_{TSSQP}(i)$ are the average wallclock times of SSQP and TSSQP on problem $i$, respectively. This measure is always in $[-1,1]$, with numbers closer to $-1$ representing SSQP taking less wallclock time than TSSQP and numbers closer to 1 representing TSSQP taking less time than SSQP. These results are plotted in Figures \ref{fig:timeComp0}, \ref{fig:timeComp2}, and \ref{fig:timeComp4} for noise levels of $1$, $10^{-2}$, and $10^{-4}$, respectively. We see that TSSQP consistently takes less time than SSQP, which may be a product of the early stopping done on these problems, as outlined in Section \ref{subsec:cutest}. However, this clearly shows that the additional cost of computing the decomposition and backtracking does not increase the computational costs so as to outweigh their benefits.

\begin{figure}[t]
\centering \includegraphics[width=1\textwidth]{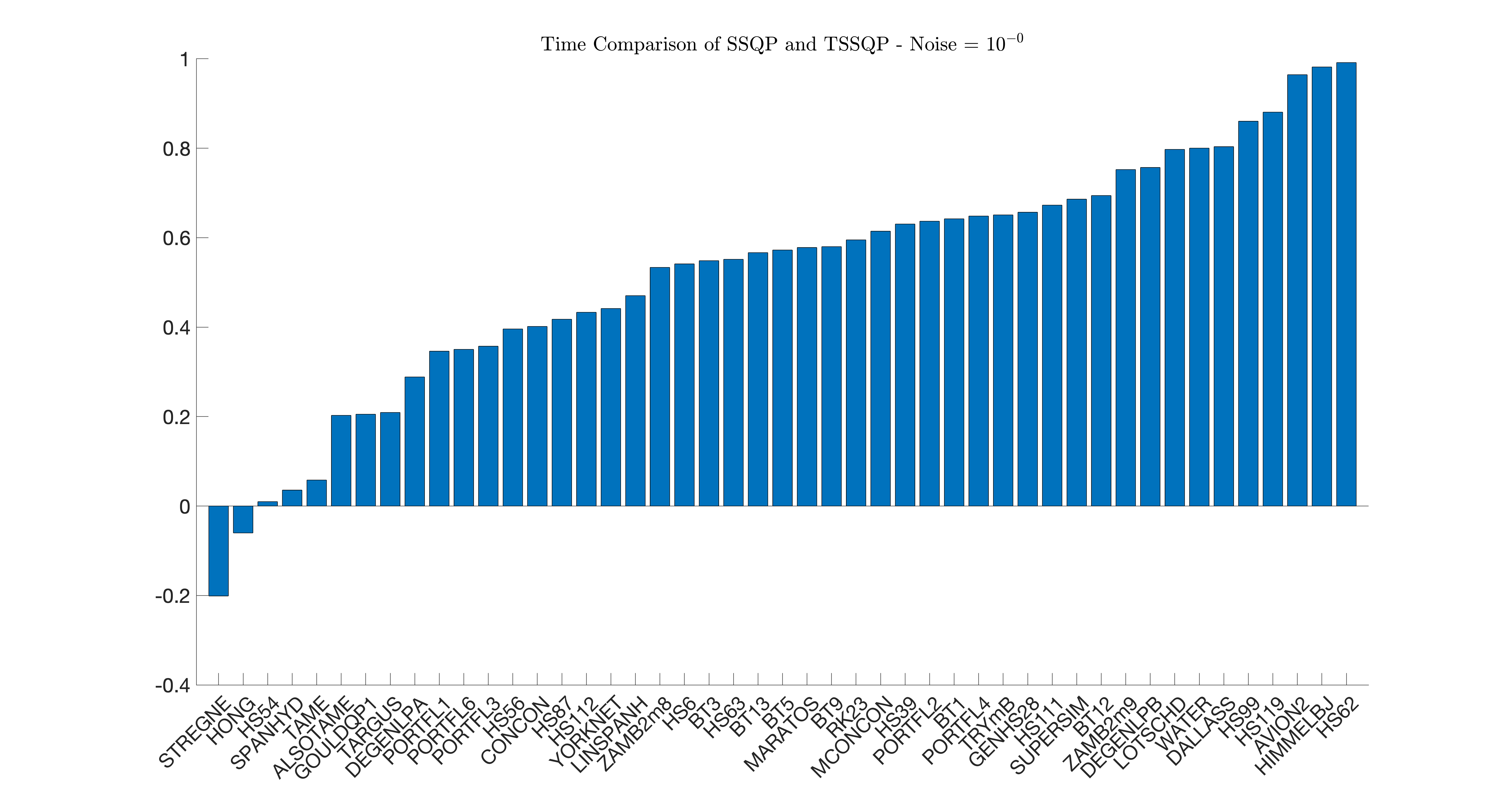}
\caption{Plots related to the measure \eqref{eq:timecompare} with $\epsilon_N = 1$.}
\label{fig:timeComp0}
\end{figure}

\begin{figure}[t]
\centering
\includegraphics[width=1\textwidth]{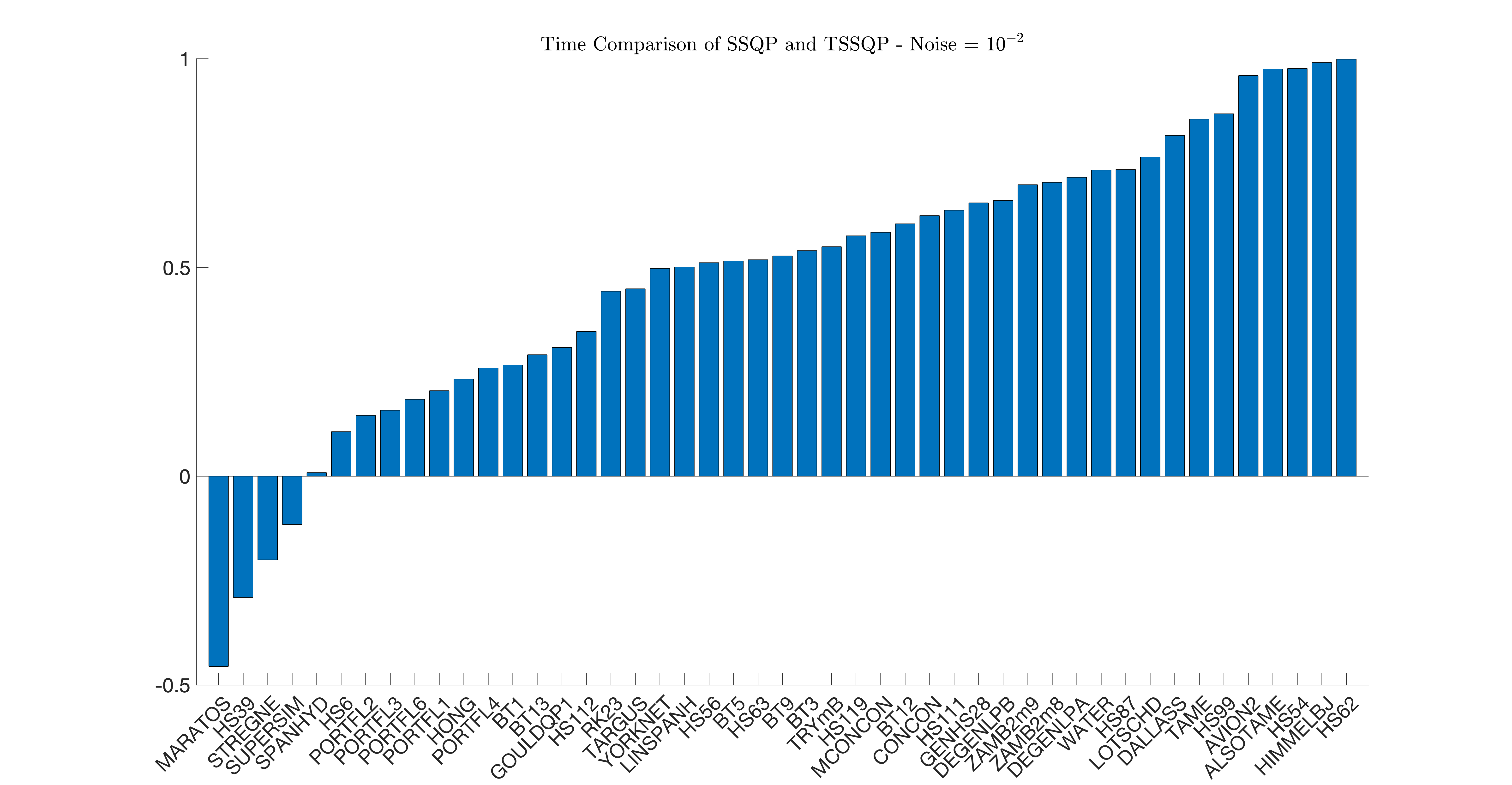}
\caption{Plots related to the measure \eqref{eq:timecompare} with $\epsilon_N = 10^{-2}$.}
\label{fig:timeComp2}
\end{figure}

\begin{figure}[t]
\centering
\includegraphics[width=1\textwidth]{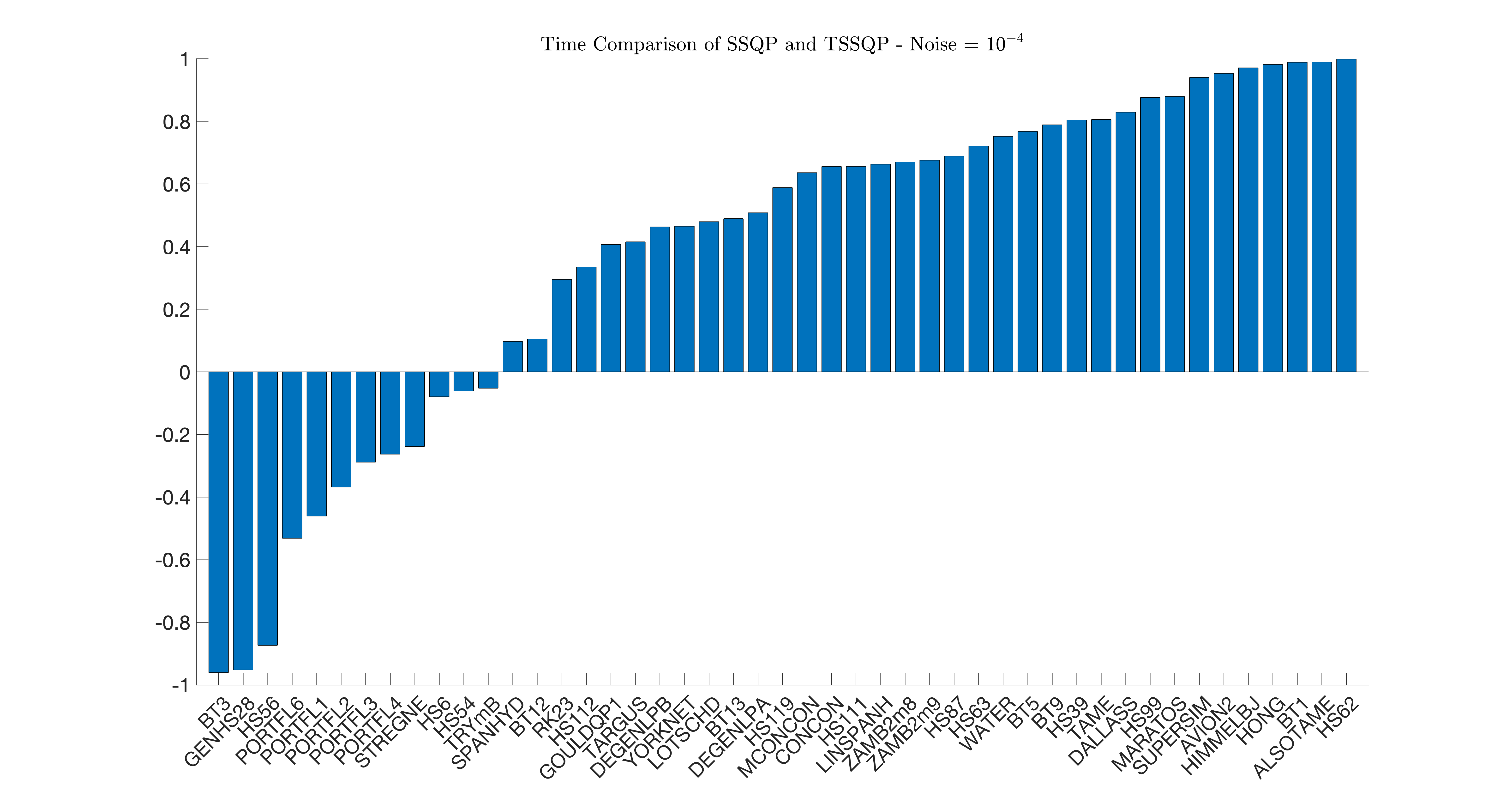}
\caption{Plots related to the measure \eqref{eq:timecompare} with $\epsilon_N = 10^{-4}$.}
\label{fig:timeComp4}
\end{figure}

In addition to these plots, we also provide plots which calculate the relative time to compute the decomposition $u_k$ and $v_k$, i.e.
\begin{equation} \label{eq:decomptime}
    T_{decomp}(i)/T_{TSSQP}(i),
\end{equation}
where $T_{decomp}(i)$ is the time spent computing $u_k$ and $v_k$ after obtaining $p_k$ on problem $i$. These results are given in Figures \ref{fig:decompTime0}, \ref{fig:decompTime2}, and \ref{fig:decompTime4}. We see that for most problems, the relative amount of time spent finding the decomposition $u_k$ and $v_k$ is no more that 15\% of the total runtime. Thus, while computing this decomposition requires additional work, it does not overwhelm the cost of the full algorithm and appears to be beneficial enough to be worth the cost of this additional overhead. There are, however, a few instances, for which this cost is significantly higher, up to 60\% of the time for one instance when the noise level was $10^{-4}$. In such cases, other methods which do not rely on this decompostion may be preferred.

\begin{figure}[t]
\centering
\includegraphics[width=1\textwidth]{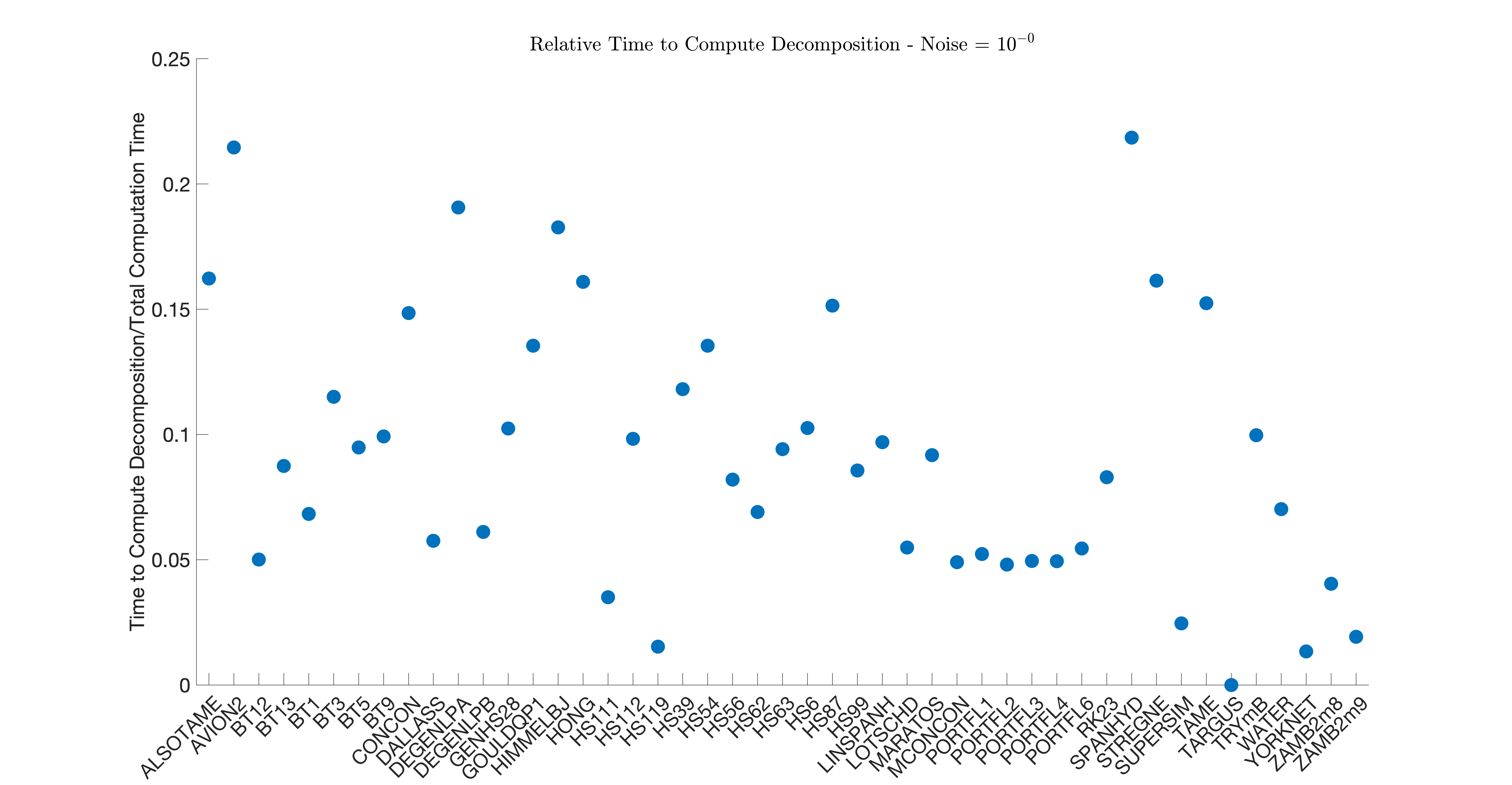}
\caption{Relative time to compute decomposition $u$ and $v$ in TSSQP, as measured by \eqref{eq:decomptime} with noise level $\epsilon_N = 1$.}
\label{fig:decompTime0}
\end{figure}

\begin{figure}[t]
\centering
\includegraphics[width=1\textwidth]{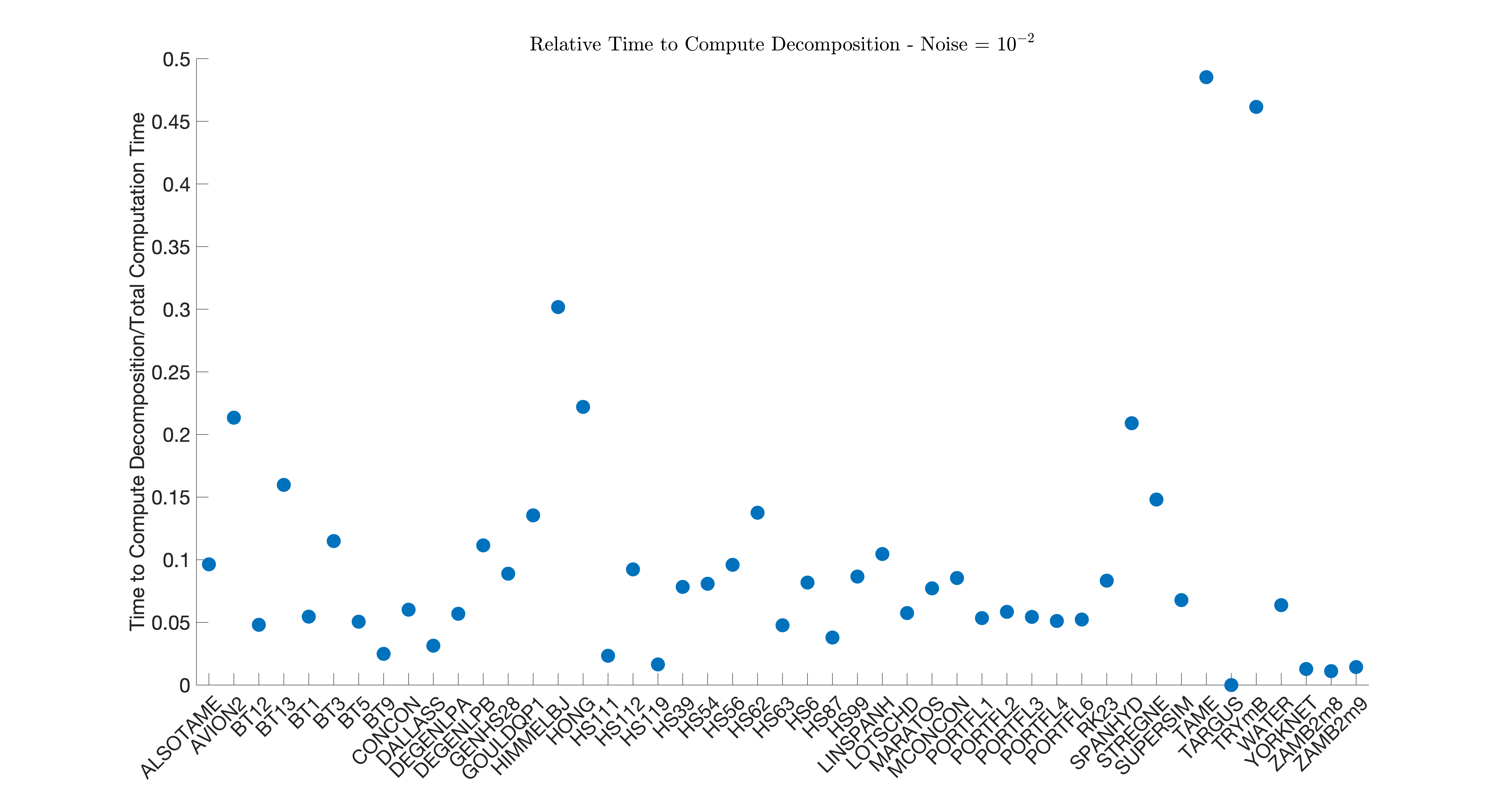}
\caption{Relative time to compute decomposition $u$ and $v$ in TSSQP, as measured by \eqref{eq:decomptime} with noise level $\epsilon_N = 10^{-2}$.}\label{fig:decompTime2}
\end{figure}

\begin{figure}[t]
\centering
\includegraphics[width=1\textwidth]{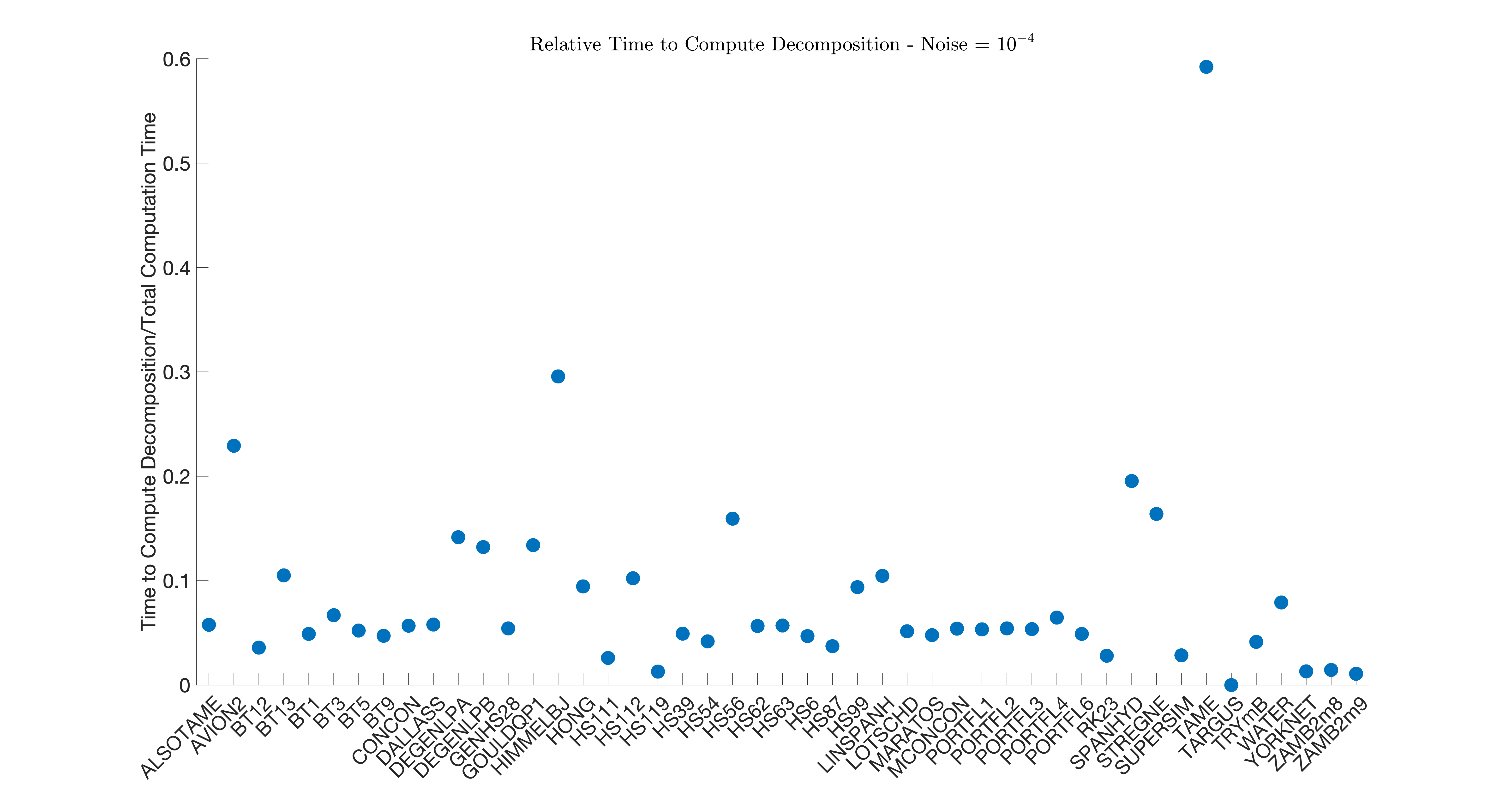}
\caption{Relative time to compute decomposition $u$ and $v$ in TSSQP, as measured by \eqref{eq:decomptime} with noise level $\epsilon_N = 10^{-4}$.}\label{fig:decompTime4}
\end{figure}

\subsection{Logistic Regression Results}

We conclude this appendix with wallclock times of the logistic regression experiments in Section \ref{subsec:lr}. We include MLALM, SSQP, and TSSQP for comparison. Recall that MLALM was given a significantly larger budget (300 epochs) than the SQP methods (10 epochs), to account for the difference in step computation costs. In addition, for TSSQP we include a column, called relative decomposition time, which records \eqref{eq:decomptime} for the logistic regression experiments. The values can be found in Table \ref{tbl:log_times}. We see that MLALM almost always takes the most amount of time of the algorithms, validating that the number of additional epochs was sufficient for fair comparison. In addition, we see that TSSQP usually takes less time to complete than SSQP, validating that computing the decomposition and using the backtracking procedure in Algorithm \ref{alg:tsssqpls} are only modest costs that do not detract from effectiveness the algorithm. Lastly, the final column reports the relative decomposition time, which shows that the computing the decomposition is less than 10\% across all problem instances and under 5\% for most.

\begin{table}[tb]
  \caption{Wallclock times for MLALM, SSQP, and TSSQP on logistic regression problems. The relative decomposition, computed by \eqref{eq:decomptime} is the final column for TSSQP.}
  \label{tbl:log_times}
\centering
{\tiny
\begin{tabular}{l|c|c|c|c|c} \toprule
                                &                                 & \multicolumn{1}{c|}{\begin{tabular}[c]{@{}c@{}}MLALM\end{tabular}} & \multicolumn{1}{c|}{\begin{tabular}[c]{@{}c@{}}SSQP\end{tabular}} & \multicolumn{2}{c}{\begin{tabular}[c]{@{}c@{}}TSSQP \end{tabular}}                               \\ \hline
\multicolumn{1}{l|}{dataset}  & \multicolumn{1}{c|}{batch} & \multicolumn{1}{c|}{Total Time}          & \multicolumn{1}{c|}{Total Time}              & \multicolumn{1}{c}{Total Time} & \multicolumn{1}{c}{Decomp Time} \\ \hline
\texttt{a9a} & 16                     &       $1.19e+02             $                    &     $      4.76e+01          $                                                                 &     $ 3.09e+01           $  & 0.05                                                \\
\texttt{a9a} & 128                     &      $6.64e+01          $                      &        $   7.43e+00        $                                                                  &     $ 5.35e+00            $          & 0.04                                   \\ \hline
\texttt{ijccn1} & 16                     &    $   7.21e+01           $                      &        $   7.87e+01           $                                                                 &   $   4.11e+01          $   & 0.04                                         \\
\texttt{ijccn1} & 128                     &   $    2.36e+01            $                     &        $  1.19e+01            $                                                             &    $  8.13e+00            $ & 0.03                                  \\ \hline
\texttt{ionosphere} & 16                     &    $   6.42e-01         $                        &         $  1.39e+00        $                                                                   &      $  7.75e-01        $  & 0.04                                 \\ 
\texttt{ionosphere} & 128                     &    $  9.80e-02        $                         &       $      1.52e-01   $                                                                        &   $   4.51e-02  $ & 0.03                                      \\ \hline
\texttt{madelon} & 16                     &     $  1.51e+01           $                      &       $    2.64e+00           $                                                                &      $1.29e+01           $   & 0.05                            \\
\texttt{madelon} & 128                     &      $   1.11e+01        $                        &       $   3.93e-01       $                                                                      &   $  8.50e-01
   $              & 0.03                            \\ \hline
\texttt{mushrooms} & 16                     &    $   2.67e+01        $                         &      $     6.90e+00      $                                                                    &     $ 9.79e+00   $    & 0.04                                    \\
\texttt{mushrooms} & 128                     &   $    2.56e+01     $                            &     $      1.10e+00     $                                                                       & $      1.69e+00  $    & 0.03                                      \\ \hline
\texttt{phishing} & 16                     &     $  1.95e+01         $                        &         $  8.28e+00         $                                                                   &    $   1.26e+01   $  & 0.04                                      \\
\texttt{phishing} & 128                     &     $  8.57e+00      $                           &      $     1.26e+00   $                                                                        &   $    2.44e+00   $  & 0.03                                        \\ \hline
\texttt{sonar} & 16                     &      $ 2.46e-01 $                       &         $  1.63e-01       $                                                                     &     $ 1.53e+00    $     & 0.09                               \\
\texttt{sonar} & 128                     &    $   6.34e-02      $                          &         $  3.08e-02      $                                                                     &    $  2.26e-02   $              & 0.05                             \\ \hline
\texttt{splice} & 16                     &      $ 1.68e+00        $                         &         $  9.33e-01      $                                                                      &     $1.19e+00    $  & 0.05                                        \\
\texttt{splice} & 128                     &    $   3.41e-01      $                           &       $    1.14e-01    $                                                                       &    $   9.44e-02 $  & 0.05                                         \\ \hline
\texttt{w8a} & 16                     &      $ 3.24e+02      $                           &         $     7.27e+01   $                                                                  &      $    7.27e+01    $ & 0.05                                     \\
\texttt{w8a} & 128                     &     $  4.89e+02    $                             &      $     9.17e+00    $                                                                       &   $   1.26e+01     $    & 0.03                                    \\ \hline
\end{tabular}}
\end{table}

\end{document}